\documentclass[11pt]{amsart}

\usepackage{amssymb}
\usepackage{amsmath}
\usepackage{amscd}
\usepackage{curves}
\usepackage{epsfig}
\usepackage{geometry}
\usepackage{graphicx}
\usepackage{pstricks}
\usepackage{pstricks-add}
\usepackage{pst-grad}
\usepackage{pst-plot}
\usepackage{verbatim}
\usepackage{latexsym}
\usepackage{eucal}
\usepackage{amsfonts}
\usepackage{enumerate}
\numberwithin{equation}{section}
\newtheorem{theorem}{Theorem}[section]
\newtheorem{lemma}[theorem]{Lemma}
\newtheorem{prop}[theorem]{Proposition}
\newtheorem{corollary}[theorem]{Corollary}

\newtheorem{definition}[theorem]{Definition}

\def \bpf {\begin{proof}}
\def \epf {\end{proof}}
\def \beq {\begin{equation*}}
\def \eeq {\end{equation*}}
\def \bsp{\begin{split}}
\def \esp{\end{split}}

\def \tgamma {\tilde{\gamma}}

\def \CI {{C^\infty}}

\def \diff {{\operatorname{Diff}}}
\def \ty {{\tilde{y}}}
\def \teta {{\tilde{\eta}}}
\def \tF {{\tilde{F}}}
\def \tG {{\tilde{G}}}
\def \tS {{\tilde{S}}}
\def \tW {{\tilde{W}}}
\def \tH {{\tilde{H}}}

\def \tPhi {{\tilde{\Phi}}}
\def \id {{\operatorname{id}}}
\def \vtheta {\vartheta}

\def \mca {{\mathcal A}}

\def \mcc {{\mathcal C}}

\def \mce {{\mathcal E}}
\def \mcf {{\mathcal F}}
\def \mcg {{\mathcal G}}
\def \mch {{\mathcal H}}

\def \mcl {{\mathcal L}}

\def \mcw {{\mathcal W}}
\def \mco {{\mathcal O}}
\def \mcp {{\mathcal P}}

\def \mcr {{\mathcal R}}
\def \mcs {{\mathcal S}}

\def \mcu {{\mathcal U}}
\def \mcv {{\mathcal V}}
\def \mcz {{\mathcal Z}}

\def \teta {\widetilde{\eta}}

\def \mbc {{\mathbb C}}

\def \mbr {{\mathbb R}}
\def \mr {{\mathbb R}}
\def \mn {{\mathbb N}}

\def\ha {\frac{1}{2}}
\def \oq {\frac{1}{4}}
\def \hb {\hbar}

\def \nsq {\frac{n^2}{4}}

\def \wtf {\widetilde{F}}

\def \wta {{\widetilde{A}}}
\def \wtb {{\widetilde{B}}}
\def \wtc {{\widetilde{C}}}

\def \wtu {{\widetilde{U}}}

\def \ka {\kappa}

\def \intx {\mathring{X}}

\def \intmco {\mathring{\mco}}
\def \ff {{\operatorname{ff}}}
\def \Id {\operatorname{Id}}

\def \div {\operatorname{div}}

\def \im {\operatorname{Im}}
\def \re {\operatorname{Re}}
\def \diag{\operatorname{Diag}}
\def \odiag{\overline{\operatorname{Diag}}}

\def \mcn {\mathcal{N}}

\def \ga {{\gamma}}
\def \eps {\varepsilon}   
\def \vphi {\varphi}   
\def \la {\lambda}   
\def \La {\Lambda}   
\def \lan {\langle}   
\def \ran {\rangle}   
\def \del {\delta}   
\def \lap {\Delta}
\def \p {\partial}

\def \novt {\frac{n}{2}}
\def \xo {{X\times_0 X}}

\def \beqq {\begin{equation}}
\def \eeqq {\end{equation}}
\def \Xh {X^2_{0, \hbar}}
\def \mck {\mathcal{K}}
\def \soh {\frac{\sigma}{h}}

\def \fnt {\frac{n}{2}}
\def \mc {\mbc}

\def \soh {\frac{\sigma}{h}}
\def \ioh {\frac{i}{h}}

\def \wtla  {\widetilde{\Lambda}}
\def \knsq {\frac{\ka_0 n^2}{4}}

\numberwithin{equation}{section}

\begin{document}
\title[High Energy Resolvent Estimates]{High Energy Resolvent Estimates on Conformally Compact Manifolds with Variable Curvature at Infinity}
\author{Ant\^onio S\'a Barreto}
\address{Ant\^onio S\'a Barreto\newline \indent Department of Mathematics, Purdue University \newline
\indent 150 North University Street, West Lafayette Indiana,  47907, USA}
\email{sabarre@purdue.edu}
\author{Yiran Wang}
\address{Yiran Wang \newline \indent Institute for Advanced Study, The Hong Kong University of Science and Technology\newline
\indent Lo Ka Chung Building, Lee Shau Kee Campus, Clear Water Bay, Kowloon, Hong Kong}
\email{yrwang.math@gmail.com}
\begin{abstract}   We  construct a semiclassical parametrix for the resolvent of the Laplacian acing on functions on non-trapping conformally compact manifolds  with variable sectional curvature at infinity,  we use it  to prove high energy resolvent estimates and to show existence of resonance free strips  of arbitrary height away from the imaginary axis.  We then use the results of Datchev and Vasy on gluing semiclassical resolvent estimates to extend these results  to conformally compact manifolds  with normal hyperbolic trapping. 
\end{abstract}
\maketitle
\tableofcontents

\section{Introduction}\label{introduction}

 Conformally compact manifolds form a special class of complete Riemannian manifolds  with negative sectional curvature near infinity.  We will use 
 $\intx$ to denote the interior of a  $C^\infty$ manifold  $X$ of dimension $n+1$ with boundary $\p X.$   We shall say that  $\rho\in C^\infty (X)$ defines $\p X,$ or $\rho$ is a boundary defining function,  if   $\rho>0$ in  $\intx,$  $\{\rho=0\}= \p X$ and $d\rho\neq 0$ at $\partial X.$ We shall say that $(\intx,g)$ is a conformally compact manifold (CCM) if  $\rho^2g$ is a $C^\infty$  non-degenerate Riemannian metric up to $\p X,$ which will be called a conformal compactification of $g.$    The hyperbolic space serves as the model for this class: $\intx$ is the open ball $\mathbb{B}=\{z\in \mr^{n+1}: |z|<1\},$ the defining function of $\p \mathbb{B}$ is $\rho=1-|z|^2,$ and the metric $g=\frac{4 dz^2}{(1-|z|^2)^2},$ where $dz^2$ is the Euclidean metric.    
 
 For a fixed boundary defining function  $\rho,$ the metric $g$ induces a metric $h_0$ on $\p X$ given by $h_0=\rho^2 g|_{\p X}.$ There are infinitely many defining functions of $\p X$ and any two of these functions $\rho$ and $\tilde \rho$ must satisfy $\rho=e^f \tilde\rho,$ with $f\in C^\infty(X).$ 
 So,  if $h_0=\rho^2 g|_{\p X}$ and $\tilde h_0={\tilde \rho}^2 g|_{\p X},$ then
  $h_0=e^{2f}\tilde h_0.$ Thus there are also infinitely many conformal compactifications of $g,$ and only the conformal class of $\rho^2g|_{\p X}$ is uniquely defined by $(\intx,g).$

 According to \cite{MM,Rafe}  a CCM $(\intx,g)$ is  complete and its sectional curvature approach $\left| d\rho|_{\p X}\right|_{h_0}^2$  as $\rho\downarrow 0$ along any any $C^\infty$ curve $\gamma \rightarrow \p X.$ We shall denote
 \begin{gather}
 \kappa=\left| d\rho|_{\p X}\right|_{h_0}, \;\ h_0=\rho^2 g|_{\p X}, \label{curv0}
 \end{gather}
   Notice that it follows from \eqref{curv0} that  $\ka$ does not  depend on the choice of $\rho.$  We shall say that  CCM for which  $\ka$ is constant are asymptotically hyperbolic manifolds (AHM).   There is a huge literature on conformally compact manifolds related to conformal field theory,  however the CCM studied there are usually Einstein manifolds, and hence they are AHM, see for example the survey by Anderson  \cite{Anderson}.

The scattering theory of hyperbolic manifolds-- where the sectional curvature is constant-- has a very long history beginning with the work of Fadeev, Fadeev \& Pavlov and Lax \& Phillips \cite{Fa,FaPa,LP}.  Scattering  on hyperbolic quotients was studied by Agmon \cite{Agmon1} and Perry \cite{Perry1,Perry2}, and  on general AHM by Mazzeo and Melrose \cite{MM}  and it has since become a very active field, see for example \cite{BP,GZ,GrZw,HoSa,IK,JS,Mar,Vasy} and references cited there.    The literature on scattering on general CCM is much shorter.  Mazzeo \cite{Rafe} studied Hodge theory on CCM and  Borthwick \cite{Borthwick} adapted the construction of Mazzeo and Melrose \cite{MM} to analyze the meromorphic continuation of the resolvent for the Laplacian on CCM with variable curvature at infinity.    

We know from the work of Mazzeo
 \cite{Rafe,Rafe1}, see also Theorem 1.1 of \cite{Borthwick}, that  if $(\intx,g)$ is a CCM, the essential spectrum of the Laplacian $\Delta_g$ consists of $[\frac{\ka_0^2n^2}{4},\infty),$ where
 $\ka_0=\min_{y\in \p X} \ka(y),$ and it is absolutely continuous. 
The resolvent in this case is defined by
\begin{gather}
 R(\la)=\left(\Delta_g-\ka_0^2\nsq -  \la^2\right)^{-1}.\label{resolv}
 \end{gather}
By the spectral theorem  $R(\la)$ is an analytic family of operators bounded on $L^2(X)$ for $|\im \la|>>0.$   
We shall  assume that $R(\la)$ is holomorphic on $\im\la>>0;$  it continues meromorphically to $\im \la>0$ with poles on the imaginary axis given by the square root of the negative eigenvalues of $\Delta_g,$ and we will study its meromorphic continuation to $\im \la<0.$    Mazzeo and Melrose \cite{MM} showed that if $(X,g)$ is an AHM, $R(\la)$ extends meromorphically to $\mc\setminus -\frac{i}{2} \mn$ and  Guillarmou \cite{Guillarmou} showed that the points $-\frac{i}{2}\mn$ can be essential singularities of $R(\la),$ unless the metric is even.   Borthwick  \cite{Borthwick}  extended the results of \cite{MM} to CCM and established  the meromorphic continuation of $R(\la)$ to a region of the complex plane that excludes some intervals contained on the imaginary axis  about the points $-\frac{i}{2}\mn$ and also a region near $\la=0,$ see Fig. \ref{Hol}  However, no estimates were given for $R(\la).$  
  
   In certain applications, for example in the study of long time behavior of the wave  or Schr\"odinger equation or  the analysis of the spectral measure, it becomes  necessary to understand the behavior of the resolvent at high energies.   Cardoso and Vodev  \cite{CaVo}, and more recently Datchev \cite{Datchev}  studied the high energy behavior of the resolvent  on the real axis  for a class of manifolds that include CCM, see \cite{Guillarmou-AB}.  Melrose, S\'a Barreto and Vasy \cite{MSV}  constructed a semiclassical parametrix for the resolvent for small perturbations of the hyperbolic space and used it to prove high energy resolvent estimates and the distribution of resonances.   Wang \cite{Wang}  extended the results of \cite{MSV}  to  non-trapping AHM. Chen and Hassel \cite{ChenHa} also constructed a semiclassical parametrix for the resolvent on  non-trapping AHM and used it to study the spectral measure, restriction theorems, spectral multiplier \cite{ChenHa1} and Strichartz estimates \cite{ChenHa2}.

  We say that a CCM $(\intx,g)$ is non-trapping if any geodesic $\gamma(t) \rightarrow \p X$ as $\pm t\rightarrow \infty$ and  we shall assume  throughout this paper, with the exception of Section \ref{Resest1}, that $(\intx, g)$ is non-trapping.   We prove the following
  
 \begin{theorem}\label{resest}  Let $(\intx,g)$ be a non-trapping CCM of dimension $n+1,$ $n\geq 1,$ and let $\rho\in C^\infty(X)$ be a boundary defining function.   For any $M>0,$ there exist $K>0$  such that $\rho^a R(\la) \rho^b$ continues holomorphically from $\im\la>0$ to the region $\im \la> - M$, $|\re \la|>K$ provided  $a, b >\frac{- \im\la}{\ka_0}$ and $a+ b\geq 0.$ Moreover, there exists $N>0,$ independent of $M,K,a,b,$ and $C>0$ such that 
 \begin{gather}
 \|\rho^a R(\la) \rho^b f\|_{L^2(X)} \leq C |\la|^{N} \|f\|_{L^2(X)}, \label{res-est0}
 \end{gather}
 see Fig.\ref{Hol}.
 \end{theorem}
 
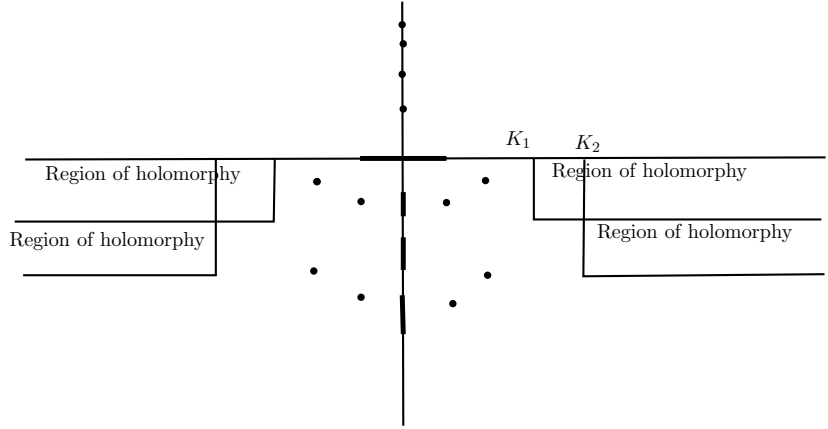
\begin{figure}[h!]
\psscalebox{.7 .7} 
{
\begin{pspicture}(0,-4.0300617)(15.500061,4.0300617)
\psline[linecolor=black, linewidth=0.04](7.46,4.03)(7.48,-4.0299997)
\psline[linecolor=black, linewidth=0.04](0.3,1.0300003)(15.5,1.0700003)
\psdots[linecolor=black, dotsize=0.14](7.48,1.9900002)
\psdots[linecolor=black, dotsize=0.14](7.46,2.6500003)
\psdots[linecolor=black, dotsize=0.14](7.48,3.2300003)
\psdots[linecolor=black, dotsize=0.14](7.46,3.5900004)
\psline[linecolor=black, linewidth=0.094](6.66,1.0500003)(8.3,1.0500003)
\psline[linecolor=black, linewidth=0.094](7.48,0.4100003)(7.48,-0.049999695)
\psline[linecolor=black, linewidth=0.094](7.48,-0.4499997)(7.48,-1.0699997)
\psline[linecolor=black, linewidth=0.094](7.46,-1.5499997)(7.46,-1.5499997)(7.48,-2.2899997)
\psline[linecolor=black, linewidth=0.04](5.04,1.0500003)(5.02,-0.1499997)(0.1,-0.1499997)
\psline[linecolor=black, linewidth=0.04](9.96,1.0500003)(9.96,-0.109999694)(15.42,-0.109999694)
\psline[linecolor=black, linewidth=0.04](3.92,1.0500003)(3.92,-1.1699997)(0.26,-1.1699997)
\rput[bl](8.48,-2.4699998){}
\rput[bl](7.46,-1.1899997){}
\rput[bl](9.42,1.2500004){$K_1$}
\rput[bl](10.74,1.1700003){$K_2$}
\psdots[linecolor=black, dotsize=0.14](5.84,0.6100003)
\psdots[linecolor=black, dotsize=0.14](5.84,0.6100003)
\psdots[linecolor=black, dotsize=0.14](9.04,0.6300003)
\psdots[linecolor=black, dotsize=0.14](6.68,0.2300003)
\psdots[linecolor=black, dotsize=0.14](8.3,0.2100003)
\psdots[linecolor=black, dotsize=0.14](6.68,-1.5899997)
\psdots[linecolor=black, dotsize=0.14](5.78,-1.0899997)
\psdots[linecolor=black, dotsize=0.14](9.08,-1.1699997)
\psline[linecolor=black, linewidth=0.04](10.92,1.0300003)(10.9,-1.1899997)(15.48,-1.1499997)
\psdots[linecolor=black, dotsize=0.14](8.42,-1.7099997)
\rput[bl](10.3,0.6300003){Region of holomorphy}
\rput[bl](11.16,-0.5299997){Region of holomorphy}
\rput[bl](0.68,0.5700003){Region of holomorphy}
\rput[bl](0.0,-0.6899997){Region of holomorphy}
\end{pspicture}
}
\caption{On any CCM, $R(\la)$ is meromorphic on $\mc$ excluding the dark intervals, \cite{Borthwick}. If the CCM is non-trapping, then $R(\la)$ is holomorphic on strips near the real axis.  The dots indicate possible poles,  also called resonances, and very little is known about their distribution. }
\label{Hol}
\end{figure}

These are the first high energy resolvent estimates on CCM with variable curvature at infinity.  The proof of Theorem \ref{resest}  follows in part the strategy of Melrose, S\'a Barreto and Vasy \cite{MSV} and Wang \cite{Wang}, but it has several new features including  the appearance of Lagrangian manifolds with polyhomogeneous singularity, their parametrization and the study of the Lagrangian distributions associated with such manifolds. 

 In the case of AHM with even metric in the sense of Guillarmou \cite{Guillarmou}, Vasy developed a semiclassical machinery that gives sharp semiclassical resolvent estimates without the need to construct a parametrix, but it is not clear whether it works in the case of variable curvature.   However Vasy's techniques apply to more general Lorentzian manifolds \cite{Vasy1,Vasy2}  to which the results proved here do not  apply.
 
Finally we combine Theorem  \ref{resest} and a result of Datchev and Vasy  \cite{DaVa1}  and Wunsch and Zworski \cite{WuZw} to extend the results of  Section \ref{Resest1} to CCM that  may have hyperbolic trapping.  We will assume that  $x \in C^\infty(X)$ is a boundary defining function and as in \cite{DaVa1}, let
\begin{gather*}
X= X_0\cup X_1, \;\ X_0=\{ x< 2\eps \}, \;\ X_1=\{ x> \eps \}.
\end{gather*}

Let $(\intx, g_0)$ be a non-trapping CCM and let $g$ be a $C^\infty$ metric on $\intx$ such that $g=g_0$ in $X_0$ and suppose that in 
$X_1$ the trapped set of $g,$  that is,  the set of maximally extended geodesics of $g$ which 
are precompact, is normally hyperbolic in the sense of Wunsch and Zworski, see section 1.2 of \cite{WuZw}, see also section 5.3 of \cite{DaVa1}. Then
 \begin{theorem}\label{resest-gen}  Let $(\intx, g)$ be a as above. If $\eps$ is small enough, then  for any $M>0,$ there exists $K>0$ such that $\rho^a R(\la) \rho^b$ continues holomorphically  from $\im \la>> 0$ to the region $\im \la>-M$, $|\re \la|>K$ provided 
 $a, b >\frac{- \im\la}{\ka_0}$ and $a+ b\geq 0$. Moreover, there exist $C>0$ and $N >0$ such that
 \begin{gather}
 \|\rho^a R(\la) \rho^b f\|_{L^2(X)} \leq C |\la|^{N} \|f\|_{L^2(X)}.\label{HERE}
 \end{gather}
 \end{theorem}

\section{The resolvent at fixed energy}\label{MMB-results}

 We briefly recall  the analysis 
of the  Schwartz kernel of the resolvent done by Mazzeo and Melrose \cite{MM}, and Borthwick \cite{Borthwick}.   In the interior $\intx\times \intx,$ one can use the Hadamard parametrix construction to show that the Schwartz kernel of  the operator $R(\la)$ defined in \eqref{resolv}, which we denote by  $R(\la,z,z'),$  is a distribution in $C^{-\infty}(\intx \times \intx)$ conormal  to the diagonal
\begin{gather*}
\diag=\{(z,z')\in X\times X: z=z'\}.
\end{gather*}
However, this only gives very limited mapping properties of the operator $R(\la).$ The main difficulty is to understand the behavior of $R(\la,z,z')$  near the boundary faces and especially near the intersection  $\diag \cap \, (\p X \times \p X).$ The behavior of $R(\la,z,z')$ near the boundary can be combined with Schur's lemma to prove weighted $L^2(X)$ estimates for $R(\la),$ see for example \cite{MM,Mazzeo-Edge,MSV}.  To analyze  the behavior of $R(\la,z,z'),$  as $z,z'\rightarrow \p X$  in all possible regimens, Mazzeo and Melrose introduced the $0$-stretched product of $X\times X,$ and we recall their construction.   Let 
\beq
\p \odiag = \{(z, z) \in \p X\times \p X\} = \odiag \cap(\p X\times \p X).
\eeq
As a set, the $0$-stretched product space is 
\beq
X\times_0 X = (X\times X)\backslash \p\odiag \sqcup S_{++}(\p \odiag),
\eeq
where $S_{++}(\p\odiag)$ denotes the inward pointing spherical bundle of $T_{\p\odiag}^*(X\times X)$. Let 
\begin{gather}
\beta_0: X\times_0 X \rightarrow X\times X \label{zero-blow-up}
\end{gather}
be the blow-down map. Then $X\times_0 X$ is equipped with a topology and smooth structure  of a manifold with corners for which $\beta_0$ is smooth. The manifold $\xo$  has three boundary hypersurfaces: the left and right faces $L=\overline{\beta_0^{-1}(\p X\times \intx)},$  $R=\overline{\beta_0^{-1}(\intx \times \p X)},$  and the front face $\ff= \overline{\beta_0^{-1}(\p\odiag)}$. The lifted diagonal is denoted by $\diag_0=\overline{\beta_0^{-1}(\diag\setminus \p \odiag)}$.   It has three codimension two corners given by the intersection of  two of these boundary faces, and a codimension three corner given by the intersection of  all the three faces. See Figure \ref{fig1}.  
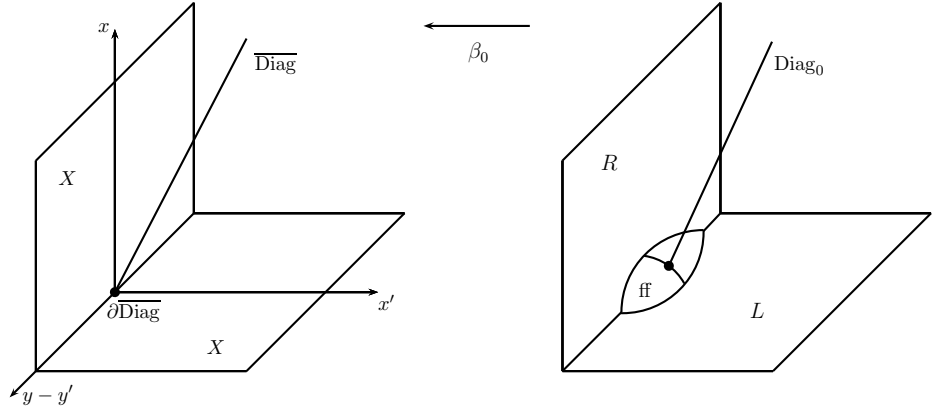
\begin{figure}[htbp]
\centering
\scalebox{0.7}
{
\begin{pspicture}(0,-3.8529167)(17.765833,3.8729167)
\rput(3.7458334,-0.14708334){\psaxes[linewidth=0.04,labels=none,ticks=none,ticksize=0.10583333cm](0,0)(0,0)(4,4)}
\rput(0.74583334,-3.1470833){\psaxes[linewidth=0.04,labels=none,ticks=none,ticksize=0.10583333cm](0,0)(0,0)(4,4)}
\psline[linewidth=0.04cm,arrowsize=0.05291667cm 2.0,arrowlength=1.4,arrowinset=0.4]{->}(3.7458334,-0.14708334)(0.24583334,-3.6470833)
\psline[linewidth=0.04cm,dotsize=0.07055555cm 3.0]{*-}(2.2458334,-1.6470833)(4.7458334,3.1729167)
\psline[linewidth=0.04cm](0.74583334,0.85291666)(3.7458334,3.8529167)
\psline[linewidth=0.04cm](4.7458334,-3.1470833)(7.7458334,-0.14708334)
\rput(2.2458334,-1.6470833){\psaxes[linewidth=0.04,arrowsize=0.05291667cm 2.0,arrowlength=1.4,arrowinset=0.4,labels=none,ticks=none,ticksize=0.10583333cm]{->}(0,0)(0,0)(5,5)}
\usefont{T1}{ptm}{m}{n}
\rput(5.295833,2.6979167){$\odiag$}
\usefont{T1}{ptm}{m}{n}
\rput(2.6,-2){$\p\odiag$}
\usefont{T1}{ptm}{m}{n}
\rput(7.4058332,-1.8420833){$x'$}
\usefont{T1}{ptm}{m}{n}
\rput(2.0358334,3.3779166){$x$}
\usefont{T1}{ptm}{m}{n}
\rput(1.35,0.5179167){$X$}
\usefont{T1}{ptm}{m}{n}
\rput(4.1658335,-2.6820834){$X$}
\usefont{T1}{ptm}{m}{n}
\rput(0.97583336,-3.6020834){$y - y'$}
\rput(13.745833,-0.14708334){\psaxes[linewidth=0.04,labels=none,ticks=none](0,0)(0,0)(4,4)}
\rput(10.745833,-3.1470833){\psaxes[linewidth=0.04,labels=none,ticks=none](0,0)(0,0)(4,4)}
\psline[linewidth=0.04cm](13.745833,3.8529167)(10.745833,0.85291666)
\psline[linewidth=0.04cm](17.745832,-0.14708334)(14.745833,-3.1470833)
\psline[linewidth=0.04cm](11.865833,-2.0270834)(10.745833,-3.1470833)
\psarc[linewidth=0.04](13.425834,-2.0270834){1.56}{90.0}{180.0}
\psarc[linewidth=0.04](11.885834,-0.50708336){1.54}{268.5312}{1.27303}
\psline[linewidth=0.04cm](13.745833,-0.14708334)(13.425834,-0.48708335)
\psarc[linewidth=0.04](12.155833,-1.9770833){1.03}{28.855661}{82.11686}
\psline[linewidth=0.04cm,dotsize=0.07055555cm 3.0]{*-}(12.765833,-1.1470833)(14.725833,3.1129167)
\usefont{T1}{ptm}{m}{n}
\rput(15.245833,2.6579165){$\diag_0$}
\usefont{T1}{ptm}{m}{n}
\rput(14.445833,-1.9970833){\large $L$}
\usefont{T1}{ptm}{m}{n}
\rput(11.635834,0.8229167){\large $R$}
\usefont{T1}{ptm}{m}{n}
\rput(12.325833,-1.6020833){$\ff$}
\psline[linewidth=0.04cm,arrowsize=0.05291667cm 2.26,arrowlength=1.4,arrowinset=0.4]{->}(10.1258335,3.4129167)(8.105833,3.4129167)
\usefont{T1}{ptm}{m}{n}
\rput(9.155833,2.9029167){\large $\beta_0$}
\end{pspicture}
}
\caption{The $0$-stretched product space $X\times_0X$.}
\label{fig1}
\end{figure}

  It is convenient  to find a suitable boundary defining function which can be used to express the metric  in a simple form.   The proof of Lemma 2.1 of \cite{Gr}, which is written in  the case of AHM,  can be easily adapted with almost no changes for CCM to  show that fixed a representative  $H_0\in [\rho^2g|_{\p X}]$ of the equivalence class of $\rho^2g|_{\p X},$  there exists a unique boundary defining function $x$ in a neighborhood of $\p X$ such that
\begin{equation}\label{prod}
x^2 g = \frac{dx^2}{\ka^2(y)} + H(x), \; H(0)=H_0, \; \kappa(y)=\left| d\rho|_{\p X}\right|_{H_0} \text{ on }   [0,\eps) \times \p X,
\end{equation}
where $H(x)$ is a $C^\infty$ family of Riemannian metrics on $\p X$ parametrized by $x.$    One can extend $x$  to $X$ by setting it equal to a constant on a compact subset of $\intx.$

In these coordinates
\begin{gather}
\Delta_g= \ka(y)^2( -(x\p_x)^2 +n x\p_x+ x^2 F(x,y) \p_x) + x^2 \Delta_{H} - x^2 H^{ij}(x,y)(\p_{y_i}\log \ka(y)) \p_{y_j}, \label{formlap}
\end{gather}
where  $\Delta_H$ is the Laplacian with respect to the metric $H(x,y)=\sum_{ij} H_{ij}(x,y) dy_i dy_j,$ and $H^{-1}=(H^{ij}),$ and $F(x,y)=\ha \p_x \log(\operatorname{det} H(x,y))$. Here we used the convention that repeated indices indicate sum over those indices.

  The Lie algebra of vector fields that vanish on $\p X$ is denoted by $\mcv_0(X)$ and as in \cite{MM}, the space of $0$-differential operator of order $m,$ denoted by  $\diff_0^m(X),$ are those of the form
 \begin{gather*}
 P=\sum_{|\alpha|\leq m } a_\alpha V_1^{\alpha_1} V_2^{\alpha_2}\ldots V_m^{\alpha_m}, \;\  V_j \in \mcv_0(X), \; a_\alpha \in C^\infty(X).
 \end{gather*}
 In  local coordinates in which \eqref{prod} holds,  $\p X=\{x=0\},$ $y \in \p X,$ $\mcv_0(X)$ is spanned by
over $C^\infty(X)$ by $\{x\p_x, x\p_y\}$ and in view of \eqref{formlap},  $\Delta_g\in \diff_0^2(X).$

The key point here is that $\diag_0$ meets the boundary of $\xo$  at the front face $\ff,$ and does so transversally. Therefore  one can use the product structure \eqref{prod} to extend $\xo,$ $\diag_0$  and $\beta_0^*\Delta_g$  across the front face.  We also observe that  near $\diag_0,$ the lifted vector fields in $\mcv_0(X)$ are smooth, tangent to the boundary, but the lift of 
 $\Delta_g$ from either the left or right factor is elliptic near $\diag_0$ across $\ff.$  Mazzeo and Melrose defined the class  $\Psi_0^{m}(X)$  of pseudodifferential operators of order 
$m$  acting on half-densities whose Schwartz kernels lift under  $\beta_0$ defined in \eqref{zero-blow-up} to a distribution which is conormal of order $m$ to $\diag_0$ and  vanish to infinite order at all faces, except the front face.  So the Schwartz kernel of $A\in\Psi_0^m(X)$ is of the form $K_A(z,z') |dg(z')|^\ha$, with $K_A$ as described above, so in particular is $C^\infty$ up to the front face. One can use the Hadamard parametrix construction to  find an operator $G_0(\la)\in \Psi_0^{-2}(X)$ such that 
 $(\Delta_g-\nsq-\la^2)G_0(\la)-\Id= E_0(\la)$ where 
 $\beta_0^*E_0\in C^\infty(\xo)$ and is supported in a neighborhood of $\diag_0.$

 Next one needs to remove the error $E_0(\la),$ and to do that  Mazzeo and Melrose  introduced another class of operators whose kernels are singular at the right and left faces.  In the case of AHM, this class will be denoted by $\Psi_0^{m,\alpha,\gamma}(\xo),$ $\alpha, \gamma \in \mc.$  An operator $P\in \Psi_0^{m,\alpha,\gamma}(\xo)$ if it can be written as a sum $P=P_1+P_2,$ where $P_1\in \Psi_0^m(X)$ and the Schwartz kernel $K_{P_2}|dg(z')|^\ha$
of the operator $P_2$ is such that $K_{P_2}$
lifts under $\beta_0$ to a conormal distribution which is smooth up to the front face, and which satisfies the following conormal regularity with respect to the other faces
\begin{equation}
\mcv_b^{k}  \beta_0^* K_{P_2} \in \rho_L^\alpha \rho_R^\gamma L^\infty(\xo),  \forall \;\ k \in \mn, \label{boundary-regularity}
\end{equation}
where $\mcv_b$ denotes the space of $C^\infty$ vector fields on $\xo$ which are tangent to the right and left faces, but can be transversal to the front face.
Mazzeo and Melrose showed that  $R(\la)$ depends meromorphically on $\la\in \mc\setminus -\frac{i}{2}\mn,$ and 
\begin{gather}
R(\la) \in  \Psi_0^{-2,\novt-i\la, \novt-i\la}(X). \label{Mazz-Mel}
\end{gather}

In the case of CCM,  Borthwick \cite{Borthwick}  used the same strategy of Mazzeo and Melrose \cite{MM}  to analyze the kernel of $R(\la),$ and he showed that  in this case it is necessary to work with  functions that have polyhomogeneous conormal singularities of variable order at the left and right faces.  We define
\begin{gather}
\begin{gathered}
\vphi \in \mcp(\xo) \Longleftrightarrow  \vphi \in C^\infty \text{ in the interior of } \xo \text{ and }  \\   \vphi  \text{ has a polyhomogeneous expansion at }  R  \text{ and } L.
\end{gathered}\label{def-variable-polyhomogeneiy}
\end{gather}
In other words,  in local coordinates $x=(x_1,x_2,x''),$  valid up to $\ff,$  where  $L=\{x_1=0\},$ $R=\{x_2=0\}$  
\begin{gather}
\begin{gathered}
\vphi\in \mcp(\xo) \Longleftrightarrow 
\vphi \sim   \sum_{j_1,j_2=0}^\infty \sum_{k_1=0}^{j_1} \sum_{k_2=0}^{j_2} x_1^{j_1}x_{2}^{j_2}  (\log x_1)^{k_1} (\log x_2)^{k_2} F_{j_1,j_2,k_1,k_2}(x''), 
\end{gathered}\label{as-ph-exp}
\end{gather}
where  $F_{j_1,j_2,k_1,k_2}$ in $C^\infty,$  in the sense that if $\vphi_{J_1,J_2}(x)$ is the function given by the sum on the right truncated for $j_1\leq J_1$ and $j_2\leq J_2,$ then for any $J_1,J_2 \in \mn$  and $\del>0$ there exists $C(J_1,J_2,\del)$ such that
\begin{gather*}
|\vphi(x_1,x_2, x'')- \vphi_{J_1,J_2}(x_1,x_2, x'')| \leq C(J_1,J_2,\del) x_1^{J_1-\eps} x_2^{J_2-\eps}.
\end{gather*}
This a very small class of polyhomogeneous distributions, but it is necessary to work in this class-- for example we will take Borel sums of the form $\sum_j h^j a_j,$ where $a_j$ is polyhomogeneous, so it is necessary to control $k_m$ in terms of $j_m,$ $m=1,2,$  and we will also need that products of such distributions remain in the same class.   Borthwick also needed to introduce the spaces 
\begin{gather}
\mck_{ph}^{\alpha,\beta}(\xo)=\{ \vphi: \rho_L^{-\alpha}\rho_R^{-\beta} \vphi \in \mcp(\xo)\}, \;\ \alpha, \beta \in C^\infty(\xo)\},
\end{gather}
 and showed that $\mck_{ph}^{\alpha,\beta}$ is invariantly defined, and  it also only depends on the values of 
 $\beta|_{\{\rho_R=0\}}$ and $\alpha|_{\{\rho_L=0\}}.$  One can then define the corresponding space of  pseudodifferential operators: Given $\alpha, \beta \in C^\infty(\xo)$ one says that 
\begin{gather}
\begin{gathered}
P\in \Psi_{0,ph}^{m,\alpha,\beta}(X)  \text{ if } P=P_1+P_2, \;\ P_1\in \Psi_0^m(X), \text{ and the kernel of } P_2 \text{ satisfies } \\  \beta_0^*K_{P_2}= K_{2}|dg(z')|^\ha, \; K_2 \in \mck_{ph}^{\alpha,\beta}(\xo). \end{gathered}\label{bor-space}
\end{gather}

 In the general case when $\ka$ is not constant, let $\ka_0=\min_{\p X} \ka,$ let
\begin{gather}
\mu(\la,y)= \frac{1}{\ka(y)} \sqrt{ \la^2-\nsq(\ka(y)^2-\ka_0^2)}, \label{def-conjfact}
\end{gather}
and extend $\mu(\la,y)$ to a  $C^\infty$ function on $X.$ One way of doing this would be to define $\mu(\la,x,y)=\mu(\la,y)$ in a tubular neighborhood of $\p X,$  where $x$ is a defining function of $\p X,$ and extend it as a constant further to the interior.  Borthwick \cite{Borthwick} proved that
\begin{gather}
R(\la) \in \Psi_{0,ph}^{-2,\novt-i \mu_L,\novt-i \mu_R}(X), \label{Bor-res}
\end{gather}
where $\mu_\bullet$ is the lift of $\mu(x,y)$ from the $\bullet$ factor, $\bullet=R,L,$ and moreover $R(\la)$ continues meromorphically from $\{\im \la >>0\}$  to 
\begin{gather*}
\mc \setminus \{\la \in \mc: \mu(\la,y)\in -\frac{i}{2} \mn_0 \text{ for some } y \in \p X\},
\end{gather*}
see Fig.\ref{Hol}.

\section{The semiclassical zero stretched product and operator spaces}\label{semiclassical-spaces}

  In this section we define the semiclassical blow-up, which is the same as in \cite{MSV}, define the operator spaces that will be used in the construction of the parametrix, with the exception of the polyhomogeneous semiclassical Lagrangian distributions  with respect to a polyhomogeneous Lagrangian submanifold that will be defined in Section \ref{PHLD}, and which is one of the novelties in this paper.  We will  state the theorem about the structure of the parametrix using notation that will be introduced in Section \ref{PHLD}.  In Section \ref{SPGC}, as an example and in preparation to the proof in the general case, we construct the parametrix in the particular case of geodesically convex CCM. 

We are interested in the uniform  behavior of $R(\la),$ defined in \eqref{resolv}  as $|\re\la|\uparrow \infty$  and $|\im\la|<M,$ and so we turn this into a semiclassical problem by setting  $h = 1/\re\la$ and regard $h$ as a small parameter,  and if we multiply  $\Delta_g-\knsq-\la^2$ by $h^2=(\re \la)^{-2},$ we define

\begin{gather}
\begin{gathered}
 P(h, \sigma,D)\  \dot=  \ h^2( \Delta_g-\knsq-\la^2) \ \dot=  \ h^2(\Delta_g-\knsq) -\sigma^2 \\
  \text{ where } \sigma = 1 + i h\im\la, \;\  \sigma \in \Omega_\hb=(1-ch, 1+ch)\times i (-Ch, Ch). 
\end{gathered}\label{defoph}
\end{gather}

As in \cite{Borthwick,MM}, we need to work on the blown up space $\xo$ to construct the Schwartz kernel of the resolvent, but now we also have the semiclassical parameter $h \in [0,1),$ and so we have to work in $X\times_0 X \times [0,1).$  We shall adapt the spaces defined by Melrose, S\'a Barreto and Vasy \cite{MSV} to inlcude distributions that have polyhomogeneous expansions at the right and left faces.  As in \cite{MSV}, we use $\hb$ to denote the semiclassical nature of a mathematical object,  so as not to confuse it with a family of such objects depending on $h,$ but we use $h$ as the parameter itself.
On $\xo \times [0, 1)$, the submanifolds $\diag_0 \times[0,1)$ and $X\times_0 X\times \{0\}$ intersect transversally. As in section 3 of \cite{MSV}, we blow up this intersection.  This gives the space $X_{0, \hb}^2$, and we shall denote the associated blow-down map by 
\begin{gather*}
\beta_{0, h}: X_{0, \hb}^2 \longrightarrow X\times_0 X\times [0, 1).
\end{gather*}
The composition of the blow down maps $\beta_0$ and $\beta_{0,\hbar}$  will be denoted by
\begin{gather*}
\beta_\hbar = \beta_{0, \hbar}\circ\beta_0 : X_{0, \hbar}^2 \longrightarrow X\times_0 X\times [0, 1).
\end{gather*}

As a compact manifold with corners, $X_{0, \hbar}^2$ has five boundary faces, see Fig.\ref{semiblow}. The left and right faces, denoted by $\mcl, \mcr,$ are the closure of $\beta_{0,\hbar}^{-1} (L \times[0, 1) ), \beta_{0, \hbar}^{-1}(R \times[0, 1))$ respectively.  The front face $\mcf$ is the closure of $\beta_{0, \hbar}^{-1}(\ff\times [0, 1) \backslash (\p \diag_0 \times \{0\})).$ The semiclassical front face $\mcs$ is the closure of $\beta_{0, \hbar}^{-1}(\diag_0\times\{0\})$. Finally, the semiclassical face $\mca$ is the closure of $\beta_{0,\hbar}^{-1}( (X\times_0 X\backslash \diag_0)\times \{0\})$. The lifted diagonal denoted by $\diag_\hbar$ is the closure of  $\beta_{0, \hbar}^{-1}(\diag_0\times(0, 1))$.  
 
\begin{figure}[h]
\centering
\scalebox{0.6} 
{\begin{pspicture}(0,-4.58)(18,5.0)
\psline[linewidth=0.04cm](5.58,0.0)(9.74,0.0)
\psline[linewidth=0.04cm](1.58,-4.0)(5.98,-3.98)
\psarc[linewidth=0.04](4.58,0.0){1.0}{0.0}{90.0}
\psarc[linewidth=0.04](0.58,-4.0){1.0}{0.0}{90.0}
\psline[linewidth=0.04cm,linestyle=dashed,dash=0.16cm 0.16cm, arrowsize=0.05291667cm 2.0,arrowlength=1.4,arrowinset=0.4]{->}(0.78,-3.8)(-0.42, -5)
\psline[linewidth=0.04cm](4.58,1)(4.58,5)
\psline[linewidth=0.04cm](0.58,-3)(0.58,1)
\psline[linewidth=0.04cm](0.58,-3)(4.58,1)
\psline[linewidth=0.04cm](1.58,-4)(5.58,0)
\pspolygon[linewidth=0.04](1.29,-3.29)(5.29,0.71)(6.4,3.8)(2.48,0.0)
\usefont{T1}{ptm}{m}{n}
\rput(7.96,3.505){\large $\text{Diag}_0\times[0, 1)_\hbar$}
\usefont{T1}{ptm}{m}{n}
\rput(7.96,1.505){\large $\{h = 0\}$}
\usefont{T1}{ptm}{m}{n}
\rput(0.7,-4.4){$h$}
\psline[linewidth=0.04cm](15.58,0.0)(19.74,0.0)
\psline[linewidth=0.04cm](11.58,-4.0)(15.98,-3.98)
\psarc[linewidth=0.04](14.58,0.0){1.0}{63.434948}{90.0}
\psarc[linewidth=0.04](10.58,-4.0){1.0}{0.0}{90.0}
\psline[linewidth=0.04cm](10.58,-3)(10.58,1)
\psline[linewidth=0.04cm](14.58,1)(14.58,5)
\psline[linewidth=0.04cm](10.58,-3)(14.58,1)
\psline[linewidth=0.04cm](11.58,-4)(15.58,0)
\pspolygon[linewidth=0.04](11.29,-3.29)(15.1,0.48)(16.17,3.15)(12.48,-0.2)
\usefont{T1}{ptm}{m}{n}
\rput(11.11,3.585){\large $\beta_{0,\hbar}$}
\usefont{T1}{ptm}{m}{n}
\rput(11.8,-0.165){\large $\text{Diag}_\hbar$}
\usefont{T1}{ptm}{m}{n}
\rput(15.99,-2.335){\large $\mcr$}
\usefont{T1}{ptm}{m}{n}
\rput(12.45,1.665){\large $\mcl$}
\usefont{T1}{ptm}{m}{n}
\rput(12.48,-2.755){\large $\mcf$}
\psline[linewidth=0.04cm,linestyle=dashed,dash=0.16cm 0.16cm](15.02,0.88)(16.0,3.58)
\psline[linewidth=0.04cm](15.46,0.5)(16.56,3.28)
\psarc[linewidth=0.04](15.24,0.72){0.3}{146.30994}{324.4623}
\psarc[linewidth=0.04](16.27,3.43){0.31}{145.30484}{347.4712}
\psarc[linewidth=0.04](14.58,0.0){1.0}{0.0}{29.47589}
\psline[linewidth=0.04cm,arrowsize=0.05291667cm 2.26,arrowlength=1.4,arrowinset=0.4]{->}(12.1258335,4)(10.105833,4)
\usefont{T1}{ptm}{m}{n}
\rput(15.72,1.545){\large $\mcs$}
\usefont{T1}{ptm}{m}{n}
\rput(18.25,1.385){\large $\mca$}
\end{pspicture} }
\caption{The semiclassical blown-up space. The figure on the right is $X^2_{0, \hbar}$  and the figure on the left is $X\times_0 X\times [0,1).$}
\label{semiblow}
\end{figure}
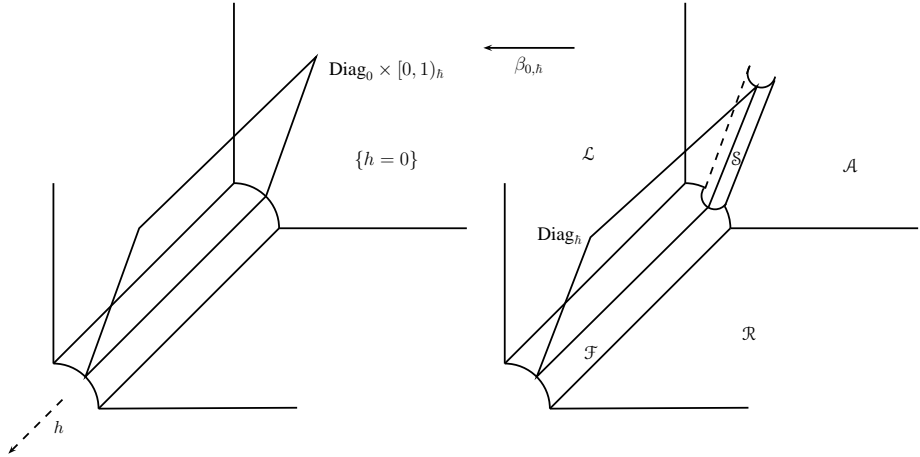

We follow the strategy of  \cite{MSV} and we will find an operator $G(h,\sigma)$ such that 
\begin{gather}
P(h,\sigma,D) G(h,\sigma)=\Id+ E(h,\sigma), \label{param-0}
\end{gather}
 where $\beta_\hb^*K_{E(h,\sigma)}$ (where $K_{\bullet}$  denotes the Schwartz kernel of   the operator $\bullet $) vanishes to infinite order on the left
face $\mcl,$  the zero-front face $\mcf,$  the semiclassical front face $\mcs$ and the semiclassical face $\mca.$  Then we just use Schur's Lemma to prove that
the error term is bounded as an operator acting between weighted $L^2(X)$ spaces and its norm goes to zero as $h\downarrow 0.$ 

 As in the work of Mazzeo and Melrose \cite{MM} and Borthwick \cite{Borthwick}, we will show that behavior $G(h,\sigma)$ at the right and left faces is determined by the indicial roots.   We say that  $\alpha \in \mc$ is an indicial root of $P(h,\sigma,D)$ if for any 
 $v\in C^\infty(X)$ there exists $V\in C^\infty(X)$ such that
\begin{gather*}
P(h,\sigma,D) (x^\alpha v) = x^{\alpha+1}  V
\end{gather*}
In view of \eqref{formlap},  if $v\in C^\infty(X)$ there exists $V \in C^\infty(X)$ such that
\begin{gather*}
P(h,\sigma,D) x^\alpha v=(h^2 (-\ka^2(y)( \alpha^2 -n\alpha)-\knsq)-\sigma^2) x^{\alpha} v + x^{\alpha+1} V,
\end{gather*}
and so $\alpha$ is an indicial root  if $h^2\ka^2(y)(\alpha^2-n\alpha)+\sigma^2+h^2 \knsq=0,$ which gives
\begin{gather}
\alpha(h,\sigma,y)= \novt \pm  i\frac{\sigma}{h} \mu, \;\  \mu=\frac{1}{\ka(y)} \left(1-\frac{n^2h^2(\ka^2(y)-\ka_0^2)}{4\sigma^2}\right)^\ha.\label{sc-ind-roots}
\end{gather}
This is the semiclassical version of the function $\mu(\la,y)$ defined in \eqref{def-conjfact}.  Since we assume the resolvent is holomorphic on $\im\la>>0,$ we will pick the negative sign in \eqref{sc-ind-roots}. When $\ka$ is constant, $\mu= \frac{1}{\ka},$  which is what appears in \eqref{Mazz-Mel}, but in general $\mu\in C^\infty(\p X)$ and it appears in \eqref{Bor-res}.    

The function $\ka\in C^\infty(\p X),$ can be extend to $C^\infty(X)$ by first setting $\ka(x,y)=\ka(y)$ in a tubular neighorhood of $\p X,$ where $(x,y)$ are as in \eqref{prod}, and then extending $\ka(x,y)$ to $X.$   We then define $\ka_R= \beta_0^*\ka(z')$ as the lift of $\ka(z')$ from the right factor  and similarly $\ka_L= \beta_0^*\ka(z)$ as the lift from the left factor.  

 We shall define 
\begin{gather*}
\mu_\bullet(m)=  \frac{1}{\ka_\bullet(m)} \left(1-\frac{n^2h^2(\ka_\bullet^2(m)-\ka_0^2 )}{4\sigma^2}\right)^\ha, \;\ \bullet=R,L, \; m\in \xo.
\end{gather*}
 If $h_0$ is such that
 \begin{gather}
\frac{n^2 h^2}{4|\sigma^2|} (\ka_\bullet^2- \ka_0^2)<\ha, \; \sigma=1+h\sigma', \ h\in [0,h_0], \ \sigma'\in (-c,c)\times (-C,C), \label{defh0}
\end{gather}
then $\mu_\bullet$ is a holomorphic function of  $(\ka_\bullet^2- \ka_0^2)\frac{n^2 h^2}{4\sigma^2},$ and since $\sigma=1+h\sigma',$ 
\begin{gather}
\mu_\bullet-\frac{1}{\ka_\bullet} = h^2 \nu_\bullet(h,\sigma',m)\sim  h^2 \sum_{j=0}^\infty \mu_{j,\bullet}(\sigma', m) h^{2j}, \;\ \bullet=R, L, \;\ \sigma=1+ h\sigma'. \label{real-anal}
\end{gather}

As for the boundary defining functions of $R$ and $L,$  we may choose $\rho_R$ and $\rho_L$ such that $\rho_R=\rho_L=1$ near $\diag_0.$  With these choices of $\ka_R,$ $\ka_L,$ $\rho_R$ and $\rho_L,$ we define
\begin{gather}
\tgamma= \mu_R \log \rho_R+ \mu_L \log \rho_L. \label{def-tgamma}
\end{gather}

Borthwick's cosntruction  suggests that the Schwartz kernel of a semiclassical  parametrix of 
\begin{gather*}
R(h,\sigma)=(h^2(\Delta_g-\knsq)-\sigma^2)^{-1},
\end{gather*}
 should roughly be of the form $G(h,\sigma)= (G_1(h,\sigma)+ G_2(h,\sigma)) |dg(z')|^\ha,$ where
\begin{gather}
\begin{gathered}
 \beta_\hb^*G_1(h,\sigma) \text{ is supported near } \mcs, \beta_\hb^* U_2 \text{ vanishes to infinite order at  } \mcs \text{ and }  \\
 \beta_0^*G_2(h,\sigma)= \rho_R^\novt \rho_L^\novt e^{-i \soh \tgamma}U_2(h,\sigma)
\end{gathered} \label{defn-conj}
\end{gather}
where  $U_2(h,\sigma)$ is a distribution that has polyhomogeneous expansions at the right and left faces.  Notice that 
$\tgamma=0$ near $\diag_0,$ and $\beta_0^*G_1$ is supported near $\diag_0,$ so $\beta_0^*G_1=e^{i\soh \tgamma} \beta_0^*G_1.$ But also notice that  in view of \eqref{real-anal},  
\begin{gather}
\begin{gathered}
\tgamma= \gamma- h^2 \beta, \;\  \gamma= \frac{1}{\ka_R}  \log \rho_R+ \frac{1}{\ka_L} \log \rho_L, \text{ and } \\
 \beta=\nu_L(h,\sigma',m) \log\rho_L+  \nu_R(h,\sigma',m) \log\rho_R,
 \end{gathered} \label{def-gamma}
 \end{gather}
and therefore
\begin{gather}
e^{i\soh \tgamma }= \rho_R^{ih\sigma \nu_R(h,\sigma',m)}\rho_L^{ih\sigma \nu_L(h,\sigma',m)} e^{i\soh \gamma}. \label{def-gamma-osc}
\end{gather}
So the oscillatory part of $e^{i\soh \tgamma}$  is given by  $e^{i\soh \gamma}$ while the remainder of the expansion of $e^{i\soh \tgamma}$ should be viewed as part of a semiclassical symbol.  However, while this symbol is polyhomogeneous, the powers of $\rho_\bullet$ that appear in the expansion will depend on $\nu(h,\sigma',m)$  and is not  an element of  the class of polyhomogeneous distributions defined above.


As in Section \ref{MMB-results} and \cite{MSV}, we define the class of semiclassical pseudodifferential operators in two steps,  first we define the space
$\Psi_{0,\hb}^m(X)$ which consists of operators whose kernel $K_P(z,z',h)\,|dg(z')|^\ha$ is such that  
\begin{gather*}
\beta_\hb^* K_P= \rho_{\mcs}^{-\novt-1} \widetilde K, 
\end{gather*}
where $\widetilde K$  is supported near $\diag_\hb$ and
\begin{gather*}
\mcv_{\diag_\hb}^k \widetilde K\in {}^\infty H_{loc}^{-m-\frac{n+1}{2}}(\Xh), \;\ k \in \mn, \\ \mcv_{\diag_\hb} \text{ consists of } C^\infty  \text{ vector fields  tangent to } \diag_\hb,
\end{gather*}
where the space ${}^\infty H_{loc}^{s}$ is the Besov space defined in the appendix B of Volume 3 of \cite{Hormander}, and  since $\Xh$ has dimension $2n+3,$ and $\diag_\hb$ has dimension $n+2,$ this choice makes $\widetilde K$ a conormal distribution of order $m-\oq$ to $\diag_\hb,$ see Theorem 18.2.8 of \cite{Hormander}.

   The analogue of the space \eqref{boundary-regularity} is given by
\begin{gather}
\mck^{a,b,c}(X_{0,\hb})=\{ K \in L^{\infty}(X_{0,\hb}): \mcv_b^m K \in \rho_{\mcl}^a \rho_{\mca}^b \rho_R^c \rho_{\mcs}^{-\novt-1}L^\infty(X_{0,\hb}), \;\ m\in \mn \}, \label{sc-con}
\end{gather} 
where $\mcv_b$ denotes the Lie algebra of vector fields which are tangent to $\mcl,$ $\mca$ and $\mcr.$  Again, as in \cite{MM}, we define the space $\Psi_{0,\hb}^{m,a,b,c}(X)$ as the operators $P$ which can be expressed in the form $P=P_1+P_2,$ with $P_1\in \Psi_{0,\hb}^{m}(X)$ and  the kernel 
$K_{P_2}\,|dg(z')|^\ha$ of $P_2$ is such $\beta_{\hb}^* K_{P_2} \in \mck^{ a,b,c}(X_{0,\hb}).$

 Once we construct the parametrix near the semiclassical front face, we obtain errors that vanish to infinite order at $\mcs$ and this will allow us to work in the space $\xo\times [0,1).$ As in the case of fixed energy $\la,$  if $\ka$ is not constant, one expects polyhomogeneous expansions  at the right and left faces instead of \eqref{sc-con},  so we define $\mcp(\xo\times [0,1))$  as 
\eqref{def-variable-polyhomogeneiy} and \eqref{as-ph-exp},  in other words

\begin{gather}
\begin{gathered}
\vphi \in \mcp(\xo \times [0,1)) \Longleftrightarrow \vphi \in C^\infty \text{ in the interior of } \xo\times [0,1)  \text{ and up to } \mcf \text{ and } \\ \vphi  \text{ has a polyhomogeneous expansion at } L \times [0,1), \;  R \times [0,1),
\end{gathered}\label{var-phg-xh}
\end{gather}
where the definition of a polyhomogeneous expansion is the one given in \eqref{as-ph-exp}. For $\mu,\zeta\in C^\infty(\xo),$ we define the space
\begin{gather}
\mck_{ph}^{\mu,\zeta}(\xo\times [0,1))= \{\vphi \in C^{-\infty}(\xo\times [0,h)): \rho_L^{-\mu}  \rho_R^{-\zeta}  \vphi \in \mcp(\xo\times [0,h))\}. \label{phg-Xh}
\end{gather}
We remark that it follows from  Lemma 2.3 of \cite{Borthwick} that these spaces do not depend on the choice of the defining functions $\rho_L,\rho_\mca,$  $\rho_\mcr$ or 
$\rho_\mcs,$ and it only depends on $\mu|_{\{\rho_\mcl=0\}}$  and $\zeta|_{\{\rho_\mcr=0\}}.$  

Since the sectional curvature is negative in a collar neighborhood of $\p X,$ it follows  that  a CCM $(\intx,g)$ has a uniform radius of injectivity, in other words  there exists $\del>0$ such that for every $z\in \intx,$ geodesic normal coordinates are valid  in a ball $B(z,\del)=\{z'\in \intx: r(z,z')<\del\},$ where $r$ is the length of the geodesic joining $z$ and $z'.$   This is equivalent to saying that
$r(z,z')$ is well defined in a neighborhood of $\diag,$ and in fact this implies that  $\beta_0^* r$ is well defined in a neighborhood of $\diag_0$ and up to a neighborhood of the front face. 

Next we state the theorem which gives the structure of the parametrix for non-trapping CCM.  We use the spaces of Lagrangian distributions 
$I_{ph}^k(\xo,\La^*,\Omega^\ha)$  in the statement  of the theorem, but we will not define it until Section \ref{PHLD}. For now we just say that this is the space  semiclassical Lagrangian distributions associated with a Lagrangian submanifold  and symbols which have polyhomogeneous singularities at the right and left faces of $\xo.$

\begin{theorem}\label{para-structure}  For  $h\in (0,h_0)$ with $h_0$ as in \eqref{defh0} and $\sigma\in \Omega_\hb=(1-ch,1+ch)\times i (-Ch,Ch),$  there exists  $G(h,\sigma)=G_0(h,\sigma)+ G_1(h,\sigma)+ G_2(h,\sigma)+ G_3(h,\sigma)+G_4(h,\sigma)$   such that
 \begin{gather}
 \begin{gathered}
 G_0\in\Psi_{0,\hb}^{-2}(X),\\
G_1(h,\sigma)= e^{i\soh r} U_1(h,\sigma), \;\ U_1(h,\sigma) \in \Psi_{0,\hb}^{-\infty,\infty, -\novt-1,\infty}(X), \\ \beta_\hb^* U_1(h,\sigma) \text{ supported near } 
 \mcs \text{ and away from } \diag_\hb, \\
\beta_0^*K_{G_2(h,\sigma)} = e^{-i\soh \tgamma} U_2(h,\sigma), \;\ U_2(h,\sigma)  \in \rho_R^{\novt} \rho_L^{\novt} I_{ph}^{\ha}(\xo,\La^*,\Omega^\ha),\\
\beta_0^* K_{G_3(h,\sigma)}=e^{-i\soh \tgamma} U_3(h,\sigma), \; U_3 \in h^\infty \mck_{ph}^{\novt,\novt}(\xo\times [0,h_0)), \\
\beta_0^*K_{G_4(h,\sigma)}= e^{-i\soh \tgamma} U_4(h,\sigma), \; U_4 \in h^\infty  \rho_{\ff}^\infty \mck_{ph}^{\novt, \novt}(\xo\times [0,h_0)), 
\end{gathered}  \label{kernelofg}
 \end{gather}
and such that,
 \begin{gather}
 P(h,\sigma,D) G(h,\sigma)-\Id \in \rho_{\mcf}^\infty \rho_{\mcs}^\infty e^{-i\soh \tgamma} \Psi_{0,\hb}^{-\infty, \infty,\infty,\novt}(X).\label{errorofg}
 \end{gather}
\end{theorem}

\section{The construction of the parametrix near $\diag_\hb$ }\label{Steps12}

 For now  we will carry out the first two steps in the proof  of Theorem \ref{para-structure}, and we will construct $G_0(h,\sigma)$ and  $G_1(h,\sigma).$  This construction  takes place near $\diag_\hb$ uniformly up to $\mcs$ and $\mcf,$ but away from the right and left faces, so the fact that the asymptotic sectional curvature is not  constant does not play a significant role in these steps.   We first remove  the singularity at the diagonal, and then remove the error at the semiclassical front face. 

  We  first  prove the following
 \begin{lemma}\label{step1-const}   There exists $G_0(h,\sigma) \in\Psi_{0,\hb}^{-2}(X)$  holomorphic in $\sigma\in \Omega_\hb,$ $h\in(0,h_0)$  such that 
 \begin{gather}
P(h,\sigma,D)G_0(h,\sigma)-\Id= E_0(h,\sigma) \in \Psi_{0,\hb}^{-\infty}(X) \label{defe0}
\end{gather}
 with $\beta_\hb^* G_0(h,\sigma)$ and $\beta_\hb^* E_0(h,\sigma)$ supported in a neighborhood of $\diag_{\hb}$ in $\Xh$ that only intersects the boundary of $\Xh$ at  the semiclassical front face $\mcs$ and the front face $\mcf.$ 
 \end{lemma} 
 \begin{proof}
 This is easily done in the interior, since $\beta_\hb^*P(h,\sigma,D)$ is elliptic there.  In this case,  the construction of $G_0(h,\sigma)$ in the interior  follows for example  from the standard Hadamard parametrix construction.  We will show that this also works uniformly up to $\mcf$ and $\mcs.$   As mentioned before, this is possible because the lifted diagonal intersects the boundary of $\Xh$ transversally only at $\mcf$ and $\mcs,$ see Fig.\ref{fige0}.  The proof is as in \cite{MM} and was used in the semiclassical setting in \cite{MSV}.   We first compute the lift of the operator $\beta_\hb^*P(h,\sigma)$ and show it is uniformly transversally
elliptic at $\diag_{\hb}$ up to the boundary of $\Xh.$    
\begin{figure}
%
\psscalebox{.7 .7} 
{
\begin{pspicture}(0,-4.6078844)(7.5094423,4.6078844)
\psline[linecolor=black, linewidth=0.04](0.7999875,4.5899377)(3.2399874,3.3899376)
\psline[linecolor=black, linewidth=0.04](5.2199874,2.5699377)(7.4199877,1.3899378)
\psbezier[linecolor=black, linewidth=0.04](3.2399874,3.3699377)(2.7599876,2.6099377)(4.5199876,1.8299377)(5.2399874,2.5499377)
\psline[linecolor=black, linewidth=0.04](5.2599874,2.5299377)(5.2599874,-2.4900622)
\psline[linecolor=black, linewidth=0.04](3.2599876,3.3499377)(3.2399874,-1.8500623)
\psbezier[linecolor=black, linewidth=0.04](3.2199874,-1.7900623)(2.7399874,-2.5500622)(4.5399876,-3.1700623)(5.2599874,-2.4500623)
\psline[linecolor=black, linewidth=0.04](5.2999873,-2.4900622)(7.4999876,-3.6700623)
\psline[linecolor=black, linewidth=0.04](3.7999876,2.3699377)(3.7999876,-2.7300622)(0.6199875,-4.3100624)(0.51998746,1.5899377)(3.8399875,2.3899376)
\rput{-271.09122}(3.538724,-3.8668945){\rput[bl](3.7399874,-0.13006225){$\diag_h$}}
\psline[linecolor=black, linewidth=0.04, linestyle=dashed, dash=0.17638889cm 0.10583334cm](3.3999875,2.5899377)(0.019987488,1.7899377)(0.21998748,-3.9700623)(3.2999876,-2.4500623)
\psline[linecolor=black, linewidth=0.04, linestyle=dashed, dash=0.17638889cm 0.10583334cm](4.4599876,2.2499378)(0.9799875,1.2099377)(1.0599875,-4.590062)(4.5999875,-2.7900622)
\rput{-1.180323}(0.030340627,0.08538842){\rput[bl](4.1599874,-1.4300623){$\mcs$}}
\rput[bl](5.8599873,0.22993775){$\mca$}
\rput[bl](4.0399876,-3.9100623){$\mcf$}
\rput{-297.3035}(-2.4163306,-2.3734872){\rput[bl](0.7399875,-3.1700623){The support of $\beta_\hb^* E_0$}}
\psline[linecolor=black, linewidth=0.04, linestyle=dashed, dash=0.17638889cm 0.10583334cm](1.4399875,-0.8100622)(3.2799876,-1.7500622)
\psline[linecolor=black, linewidth=0.04, linestyle=dashed, dash=0.17638889cm 0.10583334cm](4.6199875,-2.7900622)(4.6399875,2.2499378)
\psline[linecolor=black, linewidth=0.04, linestyle=dashed, dash=0.17638889cm 0.10583334cm](3.4599874,2.6099377)(3.3999875,-2.4500623)
\end{pspicture}
}
\caption{The support of $\beta_\hb^*E_0$}
\label{fige0}
\end{figure}
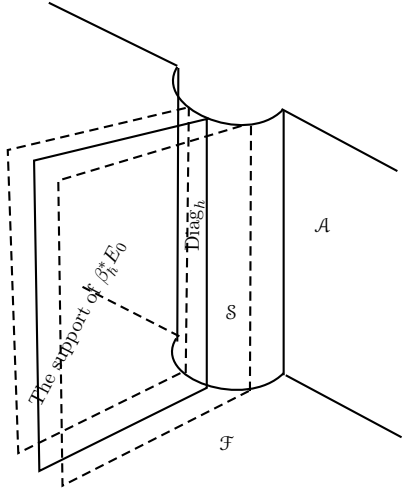

  Since we are interested in a neighborhood  of  the diagonal  we  work on a partition of unity  given by $U_j\times U_j,$ where $U_j$ is a neighborhood  intersecting  $\p X,$ and $(x,y)$ are local coordinates in $U_j$ for which \eqref{prod} works.  So $\Delta_g$ is given by \eqref{formlap} and the operator $P(h,\sigma)$ given by \eqref{defoph}.  Then  $\diag\cap (\p X\times \p X)=\{ x = x' = 0, y = y' \}$.   One can use polar coordinates $\rho_\ff=(x^2+y^2+|y-y'|^2)^\ha,$ $\rho_R=x'/\rho_\ff,$ $\rho_L=x/\rho_\ff$ and $Y=(y-y')/\rho_\ff,$  but since  $\diag_0$ does not intersect either the right or left faces of $\xo,$ $\rho_R>C>0$ near $\diag_0,$  it is easier to  use projective coordinates
\beqq\label{cord1}
X = x/x', \ \ Y = (y - y')/x', \ \ x',\ \ y'.
\eeqq 
Here $X$ is a boundary defining function for $L$ and $x'$ is a boundary defining function for $\ff$. We deduce from \eqref{formlap} and \eqref{defoph} that $
\beta_0^*P(h, \sigma)$ is given by
\begin{gather}
\begin{gathered}
\beta_0^* (P(h, \sigma))  = \\ h^2\kappa^2(y'+x'Y)[-(X\p_X)^2 +n(X\p_X)- F(x'X,y'+x'Y) X^2 x'  \p_X] + h^2X^2\lap_{H(x' X, y' + x' Y)}(D_Y) \\
- h^2X^2\sum_{i}\p_{Y_i}(\log \kappa(y'+x'Y))H^{ij}\p_{Y_j} - \knsq h^2- \sigma^2. 
\end{gathered}\label{exp-lap1}
\end{gather}
Here $\lap_H(D_Y)$ means the derivatives in $\lap_H$ are in $Y$ variable.    Notice that,
\begin{gather}
\begin{gathered}
\beta_0^* (P(h, \sigma,D)) = h^2 (\Delta_{g^0}-\knsq)- \sigma^2, \text{ where } \\
g^0(x',y',X,Y, dX dY) = \frac{dX^2}{\ka^2(y'+x'Y) X^2} + \frac{1}{X^2}\sum_{i,j=1}^n H_{ij}(x'X, y'+ x'Y) dY_i dY_j,
\end{gathered}\label{pbmetric}
\end{gather}
and we interpret $g^0$ is a $C^\infty$ family of metrics parametrized by $(x',y'),$ across the front face where $x'=0.$  We are working in the region near $\diag_0=\{X=1, Y=0\},$ so the metric $g^0$ is $C^\infty,$ and we can introduce geodesic normal coordinates around  $\{X=1, Y=0\},$ which is $C^\infty$ manifold parametrized by $(x',y').$    Fixed $z'\in \intx$, let  $r(z')$ be the geodesic distance from $z'.$ As we discussed above, this is well defined if  $r$ small, and one can interpret $\beta_0^*r$ as the distance along the geodesics of the metic \eqref{pbmetric} for fixed $(x',y').$  Since we are working on a compact region, geodesic normal coordinates  $(r,\theta)$ hold in a neighborhood of $\diag_0,$ across the front face for the metic $g_0,$ and hence
\begin{gather*}
g^0= dr^2 + r^2 \mch(r,\theta, d\theta),
\end{gather*}
where $\mch$ is a smooth 2-tensor, and by abuse of notation we are using $r=\beta_0^*r.$  Therefore, in view of \eqref{defoph} and \eqref{pbmetric}, the operator $\beta_0^*P(h,\sigma,D)$ is given by
\begin{gather}
\beta_0^*(P(h,\sigma,D))= h^2( -\p_r^2  - A(r,\theta) \p_r + \frac{1}{r^2} Q(r,\theta,D_\theta)- \knsq) - \sigma^2, \label{lap-polar}
\end{gather}
which is of course elliptic up to the front face.  Next we blow-up the intersection 
\begin{gather*}
\diag_0 \cap \{h=0\}=\{r=0, h = 0\}.
\end{gather*}
 We will work with projective coordinates:
 \begin{gather}
 \begin{gathered}
 h, \; R=r/h, \; x', y', \text{ valid in the region } |r/h| \text{ bounded, } \\
 r, \; H= h/r, \; x', y',  \text{ valid in the region } |h/r| \text{ bounded. }
 \end{gathered} \label{cords}
 \end{gather}

  At first we are interested in the region where $\diag_\hb$ meets $\mcf$ and $\mcs,$ where we may use the first set of projective coordinates in \eqref{cords}, and one has
\begin{gather}
\beta_\hb^*(P(h,\sigma,D))= -\p_R^2- hA(hR,\theta) \p_R+ \frac{1}{R^2}Q(hR,\theta, D_\theta)- h^2\knsq- \sigma^2, \label{lap-pc-2}
\end{gather}
Now it is evident that $\beta_\hbar^*P(h,\sigma,D)$ is an elliptic operator uniformly up to the front face $\mcf=\{x'=0\}$ and the semiclassical front face $\mcs=\{h=0\}.$    One can use these coordinates to extend the manifold 
$\Xh,$  lifted diagonal $\diag_\hb$,  and the operator $\beta_\hb^*(P(h,\sigma,D))$ across the faces $\mcf$ and $\mcs,$ and by doing so we get an elliptic differential operator of order $2$. One can then apply the Hadamard parametrix construction to show that there exist $G_0(\sigma)\in\Psi_{0, \hbar}^{-2}(X)$ and  $E_0(\sigma)\in \Psi_{0, \hbar}^{-\infty}(X)$ holomorphic in $\sigma$ such that
$P(h, \sigma)G_0(h,\sigma) = \Id + E_0(h, \sigma).$ Moreover,  one can cut-off $G_0(h,\sigma),$ and  $E_0(h,\sigma)$ make sure they are supported near $\diag_\hbar.$
\end{proof}

The next step is to  remove the error $E_0,$ and we do so in the following
\begin{lemma}\label{step2-const}  There exists an operator $G_1(h,\sigma)$ holomorphic in $\sigma\in \Omega_\hb,$ such that
 \begin{gather*}
G_{1}(h,\sigma)= G_{1,0}(h,\sigma)+e^{i\frac{\sigma}{h}r }U_1(h,\sigma),
\ G_{1,0}(h,\sigma)\in\Psi_{0,\hb}^{-\infty}(X),
\ U_1(h,\sigma)\in \Psi_{0,\hb}^{-\infty,\infty,-\frac{n}{2}-1,\infty}(X),
\end{gather*}
where $r$ is defined above, and such that the pull back of kernel $\beta_\hb^* K_{G_{1,0}}$ is supported near $\diag_{\hb},$ while $K_{\beta_\hb^*U_1}$  is supported near the semiclassical front face $\mcs$,  but away from $\diag_\hb,$ and
 \begin{gather}
  \begin{gathered}
 P(h,\sigma,D)G_{1}(h,\sigma)-E_0(h,\sigma)=E_1(h,\sigma), \;\\
E_1(h,\sigma)=E'_1(h,\sigma)+ e^{i \soh r} F_1(h,\sigma), \text{ with } \\
E'_1(h,\sigma)\in\rho_\mcs^\infty\Psi_{0,\hb}^{-\infty}(X),\ \ F_1(h,\sigma)
 \in  \rho_\mcs^\infty \Psi_{0,\hb}^{-\infty,\infty,-\frac{n}{2},\infty} (X)
 \end{gathered}\label{remove-e0}
  \end{gather}
 Moreover,  $\beta_{\hb}^* K_{E_1}$  is supported near $\diag_\hb,$ and hence away from $\mcl$ and $\mcr,$  $\beta_\hb^*K_{F_1}$  vanishes to infinite order at the semiclassical front face $\mcs$ and is of the form $\rho_\mca^{-\novt} C^\infty$ near the semiclassical face $\mca.$
\end{lemma}
 In other words, the error term  $E_1$ is such that the kernel of
$E'_1$ vanishes to infinite order at all boundary faces, while
the kernel of $F_1$ lifts to a $\CI$ function which is supported near $\mcs$ (and in particular vanishes to infinite order at the right and left faces),  vanishes to infinite order at the semiclassical front face, and its singularity and $\mca$ is of the form $\rho_{\mca}^{-\frac{n}{2}}\CI.$  See Fig.\ref{suppG1}.  The proof of this Lemma is quite involved,  was carried out in section 5 of \cite{MSV}, and will not be redone here. We will just say that  it makes no difference to this construction whether $\ka$ is constant or not.  The reason is that its main ingredient is the structure of the normal operator at the semiclassical front face, but in this case,  on the fiber over a point $(x',y'),$ the normal operator is the Laplacian associated with the Euclidean metric $\frac{1}{\ka(y')} dX^2+ H^{ij}(x',y',dY),$ and the $(x',y')$ play the role of parameters and do not enter in the proof of asymptotic behavior of the operators described in the lemma.
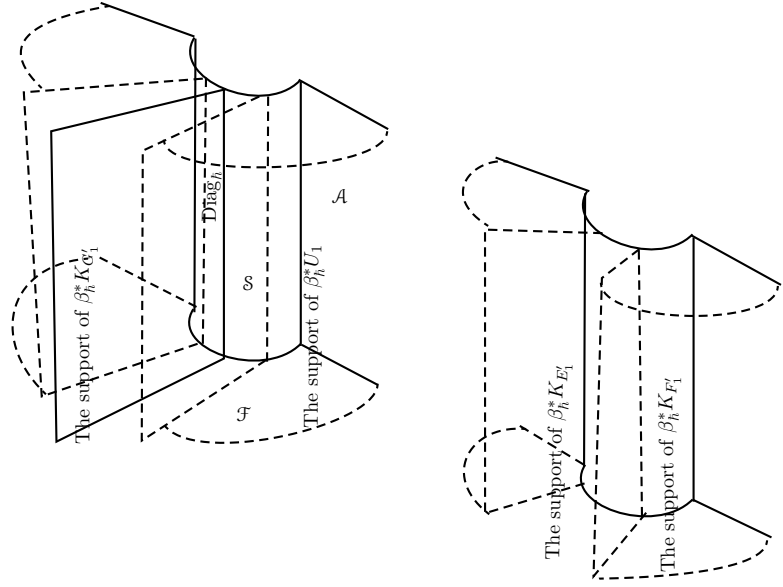
\begin{figure}
\psscalebox{.7 .7} 
{
\begin{pspicture}(0,-5.491438)(14.62459,5.491438)
\psline[linecolor=black, linewidth=0.04](1.6949197,5.472617)(3.4749198,4.8326173)
\psline[linecolor=black, linewidth=0.04](5.45492,4.012617)(7.1549196,3.0726173)
\psbezier[linecolor=black, linewidth=0.04](3.4749198,4.8126173)(2.9949198,4.052617)(4.7549195,3.272617)(5.47492,3.9926171)
\psline[linecolor=black, linewidth=0.04](5.49492,3.9726171)(5.49492,-1.0473828)
\psline[linecolor=black, linewidth=0.04](3.4949198,4.7926173)(3.4749198,-0.40738282)
\psbezier[linecolor=black, linewidth=0.04](3.4549198,-0.3473828)(2.9749198,-1.1073828)(4.7749195,-1.7273828)(5.49492,-1.0073829)
\psline[linecolor=black, linewidth=0.04](5.5149198,-1.0273829)(6.95492,-1.7873828)
\psline[linecolor=black, linewidth=0.04](4.0349197,3.8126173)(4.0349197,-1.2873828)(0.85491973,-2.8673828)(0.75491977,3.032617)(4.0749197,3.8326173)
\rput{-271.09122}(5.2116,-2.6865795){\rput[bl](3.9749198,1.3126172){$\diag_\hb$}}
\psline[linecolor=black, linewidth=0.04, linestyle=dashed, dash=0.17638889cm 0.10583334cm](3.6349196,4.032617)(0.23491974,3.772617)(0.57491976,-2.0073829)(3.5349197,-1.0073829)
\psline[linecolor=black, linewidth=0.04, linestyle=dashed, dash=0.17638889cm 0.10583334cm](4.705993,3.6926172)(2.494086,2.662922)(2.4749198,-2.8473828)(4.83492,-1.3179091)
\rput{-1.180323}(0.0,0.09053392){\rput[bl](4.39492,0.012617188){$\mcs$}}
\rput[bl](6.0949197,1.6726172){$\mca$}
\rput[bl](4.2749195,-2.467383){$\mcf$}
\psline[linecolor=black, linewidth=0.04, linestyle=dashed, dash=0.17638889cm 0.10583334cm](1.6749197,0.6326172)(3.5149198,-0.30738282)
\psline[linecolor=black, linewidth=0.04, linestyle=dashed, dash=0.17638889cm 0.10583334cm](4.85492,-1.3473828)(4.87492,3.6926172)
\psline[linecolor=black, linewidth=0.04, linestyle=dashed, dash=0.17638889cm 0.10583334cm](3.6949198,4.052617)(3.6349196,-1.0073829)
\rput{89.8194}(3.3088424,-8.534055){\rput[bl](5.93492,-2.6073828){The support of $\beta_\hb^* U_1$}}
\rput{-269.76743}(-1.2457212,-4.574091){\rput[bl](1.6549197,-2.9073827){The support of $\beta_\hb^* K_{G_1'}$}}
\psline[linecolor=black, linewidth=0.04](9.19492,2.532617)(10.974919,1.8926172)
\psbezier[linecolor=black, linewidth=0.04](10.934919,1.8926172)(10.45492,1.1326172)(12.21492,0.35261717)(12.934919,1.0726172)
\psline[linecolor=black, linewidth=0.04](12.91492,1.0526172)(14.61492,0.11261719)
\psline[linecolor=black, linewidth=0.04](12.95492,1.0126172)(12.95492,-4.007383)
\psline[linecolor=black, linewidth=0.04](10.894919,1.8526171)(10.87492,-3.3473828)
\psbezier[linecolor=black, linewidth=0.04](10.894919,-3.3073828)(10.41492,-4.067383)(12.21492,-4.6873827)(12.934919,-3.967383)
\psline[linecolor=black, linewidth=0.04](12.934919,-3.967383)(14.434919,-4.6873827)
\psline[linecolor=black, linewidth=0.04, linestyle=dashed, dash=0.17638889cm 0.10583334cm](11.925993,0.7926172)(11.214087,0.24292201)(11.05492,-5.427383)(12.05492,-4.2179093)
\psline[linecolor=black, linewidth=0.04, linestyle=dashed, dash=0.17638889cm 0.10583334cm](11.225993,1.0953741)(8.994086,1.162922)(8.995002,-4.2273827)(10.884301,-3.6610246)
\psline[linecolor=black, linewidth=0.04, linestyle=dashed, dash=0.17638889cm 0.10583334cm](11.91492,0.77261716)(11.974919,-4.2873826)
\rput{88.7591}(5.2197275,-15.588783){\rput[bl](10.57492,-5.1273828){The support of $\beta_\hb^* K_{E_1'}$}}
\rput{-269.1561}(7.6162734,-18.199677){\rput[bl](12.7749195,-5.347383){The support of $\beta_\hb^* K_{F_1'}$}}
\psbezier[linecolor=black, linewidth=0.04, linestyle=dashed, dash=0.17638889cm 0.10583334cm](2.8949196,2.8726172)(2.8949196,2.0726173)(7.1349196,2.2326171)(7.1349196,3.032617)
\psbezier[linecolor=black, linewidth=0.04, linestyle=dashed, dash=0.17638889cm 0.10583334cm](2.9149199,-2.5673828)(3.5349197,-3.3473828)(6.95492,-2.5673828)(6.95492,-1.7673829)
\psbezier[linecolor=black, linewidth=0.04, linestyle=dashed, dash=0.17638889cm 0.10583334cm](1.8549198,5.412617)(0.85491973,5.2326174)(-0.5050803,4.8726172)(0.57491976,3.8526173)
\psbezier[linecolor=black, linewidth=0.04, linestyle=dashed, dash=0.17638889cm 0.10583334cm](1.4549197,0.6126172)(-0.10508026,0.4926172)(-0.42508027,-0.9473828)(0.65491974,-1.9673828)
\psbezier[linecolor=black, linewidth=0.04, linestyle=dashed, dash=0.17638889cm 0.10583334cm](11.19492,0.15261719)(11.19492,-0.6473828)(14.55492,-0.7073828)(14.55492,0.092617184)
\psbezier[linecolor=black, linewidth=0.04, linestyle=dashed, dash=0.17638889cm 0.10583334cm](11.159719,-5.471801)(11.181339,-5.4626493)(14.33174,-5.4738135)(14.35012,-4.7429647)
\psbezier[linecolor=black, linewidth=0.04, linestyle=dashed, dash=0.17638889cm 0.10583334cm](9.434919,2.4526172)(8.79492,2.552617)(8.07492,1.8326172)(9.05492,1.2326171)
\psline[linecolor=black, linewidth=0.04, linestyle=dashed, dash=0.17638889cm 0.10583334cm](9.65492,-2.5673828)(10.87492,-3.3073828)
\psbezier[linecolor=black, linewidth=0.04, linestyle=dashed, dash=0.17638889cm 0.10583334cm](9.7749195,-2.6273828)(8.697025,-2.5273829)(8.124393,-3.5273829)(9.07492,-4.2473826)
\end{pspicture}
}
\caption{The supports of $\beta_\hb^* G_1,$ $\beta_\hb^* K_{U_1},$ $\beta_h^*K_{E_1'}$ and $\beta_\hb^* K_{F_1}.$}
\label{suppG1}
\end{figure}

The lift of the kernel of  $E_1'\in \rho_S^\infty\Psi_{0,\hb}^{-\infty}(X)$ is supported near $\diag_\hb,$  so  it will be ignored from now on, and it will be part of the final error term in the construction of the parametrix.  On the other hand, the term lift of the kernel of $e^{-i \soh r} F_1(h,\sigma)$   does not vanish at the semiclassical face $\mca$ and therefore cannot be ignored.  So we need to find an operator $G_2(\sigma,\hb)$ such that
\begin{gather*}
P(h,\sigma,D) G_2(h,\sigma)-e^{i \soh r} F_1(h,\sigma)= E_3(h,\sigma),
\end{gather*}
where $E_3(h,\sigma)$ vanishes to infinite order at the semiclassical front face $\mca.$    Since $\beta_\hb^*K_{F_1}$ vanishes to infinite order at the semiclassical face, it follows that $\beta_0^*K_{F_1}$ vanishes to infinite order at $\diag_0\times [0,1),$ and we shall work on $\xo\times [0,1) $ rather than on $\Xh.$

As suggested above in \eqref{defn-conj} it is natural to try to construct an operator $G_2(h,\sigma)$  such that
\begin{gather*}
\beta_0^*G_2(h,\sigma)= e^{-i\soh \tgamma} U_2(h,\sigma), \;\ \tgamma= \mu_R \log \rho_R+\mu_L \log \rho_L,
\end{gather*}
and so  we define
\begin{gather}
\begin{gathered}
P_L(h,\sigma,D)=\beta_0^*( h^2(\Delta_{g(z)}-\knsq)-\sigma^2), \;\
P_R(h,\sigma,D)= \beta_0^*( h^2(\Delta_{g(z')}-\knsq)-\sigma^2),
\end{gathered}\label{defpRL}
\end{gather}
where $R$ and $L$ indicate the lift of $\Delta_g$ from the respective factor of $X\times X,$ and 
\begin{gather}
\begin{gathered}
e^{i\soh \gamma}P_{\bullet}(h,\sigma,D) e^{-i\soh \gamma}=  P_{\bullet,\gamma}(h,\sigma, D), \; \bullet=L,R,\\
 e^{i\soh \tgamma}P_{\bullet}(h,\sigma,D) e^{-i\soh \tgamma}=  P_{\bullet,\tgamma}(h,\sigma, D), 
\end{gathered}\label{def-plga}
\end{gather}
Therefore, 
\begin{gather*}
\beta_0^*(P(h,\sigma,D) G_2(h,\sigma)-e^{i \soh r} F_1(h,\sigma))=P_L(h,\sigma,D) e^{-i\soh \tgamma} U_2(h,\sigma) -e^{i \soh \beta_0^*r} \beta_0^*K_{F_1(h,\sigma)}= \\
e^{-i\soh \tgamma}( P_{L,\tgamma}(h,\sigma,D) U_2(h,\sigma) - e^{i \soh(\beta_0^* r+\tgamma)} \beta_0^*K_{F_1(h,\sigma)}, \\
\end{gather*}

But if we choose $\rho_R$ and $\rho_L$ such that  with $\rho_R=\rho_L=1$ on the support of $\beta_0^*K_{F_1},$ $\diag_0,$  then $\tgamma=0$ on the support of 
$\beta_0^*K_{F_1},$ $\diag_0,$ and so we should find $U_2(h,\sigma)$ such that
\begin{gather*}
 P_{L,\tgamma}(h,\sigma,D) U_2(h,\sigma) - e^{i \soh \beta_0^*r} \beta_0^*K_{F_1(h,\sigma)}= h^\infty e^{-i\soh \tgamma} \mce(h,\sigma), \;\ \mce\in \mck_{ph}^{\novt,\novt}(\xo\times [0,1)).
 \end{gather*}
  
  Near $\diag_0,$ $e^{i\soh \beta_0^*r} F_1$ is a semiclassical Lagrangian distribution with respect to the manifold obtained by flowing-out the subset of the conormal bundle of the diagonal where $\{\beta_0^*p_L=\beta_0^*p_R=0\},$ ($\beta_0^*p_\bullet$ is the principal symbol of $\beta_0^*P_\bullet(h,\sigma,D),$ $\bullet=R,L,$)   along the integral curves of $H_{\beta_0^*p_\bullet}.$    We will analyze this Lagrangian submanifold in detail before proceeding with the construction of the parametrix.


\section{The underlying Lagrangian submanifolds }\label{Geo-Pre}

In this section we will discuss  the underlying Lagrangian submanifolds that will be used in the construction of the parametrix, and we assume that $(\intx,g)$ is a non-trapping CCM.  

The Riemannian metric $g$ in the interior $\intx$ induces an isomorphism between the tangent and cotangent bundles of $\intx:$
\begin{gather}
\begin{gathered}
\mathcal{G}: T_z\intx\longrightarrow T_z^* \intx \\
v\longmapsto g(z)(v, \cdot),
\end{gathered}\label{isomor}
\end{gather}
which in turn induces a dual metric on $T^* \intx$ given by
\begin{gather}
g(z)^*(\xi,\eta)= g(z)( \mcg^{-1} \xi, \mcg^{-1}\eta). \label{dual-met}
\end{gather}
We shall denote
\begin{gather*}
 |\zeta|_{g^*(z)}^2=g^*(z)(\zeta,\zeta).
 \end{gather*}

In local coordinates we have
\begin{gather*}
g(z)(v,w)=\sum_{i,j} g_{ij}(z) v_i w_j \text{ and }  g(z)^*(\xi,\eta)=\sum_{i,j} g^{ij} (z)\xi_i \eta_j,  \\
\text{ where  the matrices } (g_{ij})^{-1}=(g^{ij}).
\end{gather*}

The cotangent bundle  $T^*\intx$ equipped with the canonical $2$-form  $\omega$ is a symplectic manifold. In local coordinates $(z, \zeta)$ in $T^* \intx,$ 
$\omega = \sum_{j = 1}^{n+1} d\zeta_j\wedge dz_j.$ If  $f\in C^\infty(T^*\intx; \mathbb{R}),$ its Hamilton vector field $H_f$  is defined to be the vector field that satisfies  $\omega(\cdot, H_f) = df,$ and in local coordinates, 
\begin{equation*}
H_f = \frac{\p f}{\p\zeta}\cdot\frac{\p}{\p z} - \frac{\p f}{\p z}\cdot\frac{\p}{\p \zeta}.
\end{equation*} 

We also recall, from for example section 2.7 of \cite{AM}, that if
\begin{gather}
p(z,\zeta)=\ha(|\zeta|_{g^*(z)}^2-1), \label{energy}
\end{gather}
the  integral curves of of the Hamilton vector field $H_p$ are called bicharacteristics, and the projection of the bicharacteristics contained $\mcn=\{p=0\}$ to $\intx$ are  geodesics of the metric $g.$   In other words, if $(z',\zeta')\in T^*\intx$ and $p(z',\zeta')=0,$  and if $\gamma(r)$ is a curve such that
\begin{gather*}
\frac{d}{dr} \gamma(r)=H_p(\gamma(r)), \\
\gamma(0)=(z',\zeta'),
\end{gather*}
then,  with $\mcg$ given by \eqref{isomor},
\begin{gather}
\begin{gathered}
\gamma(r)= \left(\alpha(r), \mcg( \frac{d}{dr} \alpha(r)) \right), \text{ where } \\
\alpha(r)= \exp_{z'}(rv), \;\  v\in T_{z'} \intx, \;\ |v|_g^2=1, \;\ \mcg(v)= \zeta',
\end{gathered}\label{geod-1}
\end{gather}
 $\exp_{z'}(r \bullet)$ denotes the exponential map on $T_{z'} \intx.$     
  We can identify $T^*(\intx \times \intx)= T^* \intx \times T^*\intx,$ and  according to this we shall  use  $(z,\zeta,z',\zeta')$ to denote a point in $T^*(\intx \times \intx),$  $(z,\zeta)$ will denote a point on the left factor and   $(z',\zeta')$ will denote a point on the right factor.   We shall denote
  \begin{gather}
 \La_0=\{(z,\zeta,z',\zeta'):  \; z=z', \; \zeta= -\zeta'\, \; p(z,\zeta)=p(z',\zeta')=0\}. \label{cosphere}
 \end{gather}

  We shall study the manifold obtained by the flow-out of $\La_0$  by $H_p.$ In other words, 
  \begin{gather}
  \begin{gathered}
  \La=\{(z,\zeta,z',-\zeta') \in T^*(\intx\times \intx)\setminus 0: (z,\zeta)\in \mcn, (z',-\zeta')\in \mcn \\  \text{ lie on the same integral curve of } H_p  \}, 
  \end{gathered}\label{defla0}
    \end{gather}
   and one can also write
  \begin{gather}
  \La= \{(z,\zeta,z',-\zeta')\in \mcn\times \mcn: (z,\zeta)= \exp(t H_p)(z',-\zeta') \text{ for some } t \in \mr\}, \label{newdefla}
  \end{gather}
  where $\exp(t H_p)(z',-\zeta')$ denotes the map that takes $(z',-\zeta')$ to the point obtained by traveling a time $t$ along the integral curve of $H_p$ through the point $(z',-\zeta').$  
   
      The non-trapping assumption guarantees that $\intx$ is pseudo-convex with respect to $p,$ see Definition 26.1.10 of \cite{Hormander}, and  according  to Theorem 26.1.13 of \cite{Hormander}, $\La$ is a $C^\infty$ Lagrangian submanifold which is closed in $T^*(\intx\times \intx)\setminus 0$  equipped with the canonical form $\omega=\pi_L^*\omega_L+\pi_R^*\omega_R,$ where $\omega_\bullet $ is the canonical form on the $\bullet$-factor, $\bullet=R,L,$ $\pi_\bullet: T^*X \times T^*X \mapsto T^* X$ is the projection on the $\bullet$-factor  In canonical coordinates $\omega =d\zeta\wedge dz+ d\zeta'\wedge d z'.$  Notice that $\La$ is not conic because we are taking $|\zeta|_{g^*(z)}=|\zeta'|_{g^*(z')}=1.$ 
      
  We also observe that $\La \setminus \La_0=\La_R\cup \La_L$ where 
\begin{gather*}
\La_L=\{(z,\zeta,z',\zeta')\in \La: \; (z,\zeta) \text{ lies after } (z',-\zeta') \text{ on the bicharacteristic of } H_p\}, \\
\La_R=\{(z,\zeta,z',\zeta')\in \La: \; (z,\zeta) \text{ lies before } (z',-\zeta') \text{ on the bicharacteristic of } H_p\}.
\end{gather*}
The non-trapping assumption implies that $\La_L\cap \La_R=\emptyset,$  as otherwise we would have a closed bicharacteristic.

  Another way of interpreting $\La_L$ and $\La_R$ is to define $p_R,p_L\in C^\infty(T^*(\intx\times \intx))$ as the function $p$ on the right and left factors, i.e.  $p_R=p(z',\zeta')$ and $p_L=p(z,\zeta),$  and then think of $\La_R$ as the flow-out of $\Sigma$ under $H_{p_R}$ and of $\La_L$ as the flow-out of  $\Sigma$ under $H_{p_L}$ for positive times:
  \begin{gather}
  \La\setminus \Sigma=\bigcup_{t_1>0,t_2>0} \exp(t_1 H_{p_R}) \circ \exp(t_2 H_{p_L}) \La_0. \label{joint-flow-out}
  \end{gather}
  
   The map $\beta_0$ defined in \eqref{zero-blow-up} induces a map  $\beta_0:T^*(\xo) \longmapsto T^*(\intx \times \intx),$ and we want to understand the behavior of $\beta_0^* \La$ up to $\p T^*(\xo).$  Even though  $\xo$ is not a $C^\infty$ manifold,  the product structure  \eqref{prod}  valid in a tubular neighborhood of $\p X$  can be lifted to $\xo$ and it gives a way of doubling $\xo$ across its boundary and extending the metric 
  $\beta_0^*(x^2g),$ where $x$ is the boundary defining function in \eqref{prod}, and we denote this extension  by $\widetilde{\xo}.$  So we may think of $\xo$ and a submanifold with corners of a $C^\infty$ manifold $\widetilde{\xo}.$
   
    It follows from \eqref{joint-flow-out} that $\beta_0^*\La$ is given by the joint flow-out of the  $\beta_0^*\La$ by  $H_{\beta_0^*p_R}$ and $H_{\beta_0^*p_L}$  and our goal is to understand its behavior up to $\p T^*(\xo).$  

  Recall, see for example \cite{Martinez,Zworski-Sc},  that if $P(z,h,D)$  is a semiclassical differential operator on a manifold $M,$ which in local coordinates is given by $P(z,h,D)=\sum_{|\alpha|\leq m} a_\alpha(h,z) (h D_z)^\alpha,$ with $D=\frac{1}{i} \p,$  $a_{\alpha}\in C^\infty([0,1] \times M),$ one defines its semiclassical principal symbol as
  \begin{gather*}
  \sigma^{sc}(P(z,h,D))(z,\zeta)= \sum_{|\alpha| \leq m} a_\alpha(0,z) \zeta^\alpha,
  \end{gather*}
  and this is invariantly defined as a function on $T^*M.$   
  
  If $\vphi\in C^\infty(M),$  and $h\in (0,1),$  we shall encounter operators obtained from $P(h,z,D)$   by conjugation of the type
   \begin{gather*}
 P_\vphi(h,z,D)=  e^{i\frac{\vphi}{h}}  P(z,h,D) e^{-i\frac{\vphi}{h}}= \sum_{|\alpha|\leq m} a_\alpha(h,z) (h D_z-d\vphi)^\alpha,
   \end{gather*}
   and hence these remain semiclassical differential operators whose semiclassical principal symbols are given by
   \begin{gather}
 p_\vphi(h,z,\zeta)= \sum_{|\alpha|\leq m} a_\alpha(0,z)(\zeta - d\vphi)^\alpha=\sigma^{sc}(P(z,h,D))(z,\zeta-d\vphi), \label{scps0}
   \end{gather}
  where for $(z,\zeta) \in T^*M,$ $(z, \zeta-d \vphi)$ denotes the shift along the fiber direction by $-d\vphi.$    Notice that the transformation $(z,\zeta)\mapsto (z, \zeta-d\vphi)$ preserves the symplectic form of $T^*M.$  
    
  We also obtain semiclassical operators by conjugating standard differential operators 
  $P(z,D)=\sum_{|\alpha|\leq m} a_{\alpha}(z) D^\alpha:$     
  \begin{gather*}
 e^{i\frac{\vphi}{h}} h^m P(z,D) e^{-i\frac{\vphi}{h}}=h^m P(z, D- \frac{1}{h} d\vphi)=\\  \sum_{|\alpha|=m} a_\alpha(z)(hD - d\vphi)^\alpha + 
 \sum_{|\alpha|<m} h^{m-|\alpha|} a_\alpha(z)(hD - d\vphi)^\alpha.
 \end{gather*} 
 But in this case, the semiclassical principal symbol  of the resulting operator is equal to
 \begin{gather}
  \begin{gathered}
 \sigma^{sc} ( e^{i\frac{\vphi}{h}} h^m P(z,D) e^{-i\frac{\vphi}{h}}) = \sum_{|\alpha|=m} a_\alpha(z)(\zeta - d\vphi)^\alpha=\sigma_m(P(z,D))(z,\zeta-d\vphi),
 \end{gathered}\label{scps}
  \end{gather}
  where $\sigma_m(P(z,D))$ is the principal symbol of $P(z,D).$

 In the present case, we will work with the operators $P_L(h,\sigma,D)$ and  $P_R(h,\sigma,D)$ defined in \eqref{defpRL},  with $\sigma=1+h \sigma',$  and $\sigma'\in (-c,c)\times i (-C,C).$   If  $\gamma$ is as in \eqref{def-gamma},   $P_{\bullet,\gamma}(h,\sigma,D)$ was defined in  \eqref{def-plga}, and
  so,  if $p_\bullet(m,\nu)$ denotes the semiclassical principal symbol of $P_\bullet(h,\sigma,D),$ then according to \eqref{scps} the semiclassical principal symbol of $\ha P_{\bullet,\gamma}(h,\sigma,D)$ is given by
  \begin{gather}
p_{\bullet,\gamma}(m,\nu)=  \ha \sigma^{sc}(P_{\bullet,\gamma}(h,\sigma,D))= \ha p_\bullet(m,\nu- d\gamma),\label{scps-1}
  \end{gather}
where we  used that $\sigma=1+h\sigma',$ $\sigma'\in (-c,c)\times i(-C,C).$  This corresponds to a change in the fiber variables, and we denote
\begin{gather}
\begin{gathered}
S_\ga:T^*(\xo) \longrightarrow T^*(\xo) \\
(m,\nu) \longmapsto (m,\nu- d\gamma).
\end{gathered}\label{shift-map}
\end{gather}
This map is $C^\infty$ in the interior of $T^*(\xo),$ and preserves the symplectic structure, and observe that $p_{\bullet,\ga}= p_{\bullet}\circ S_\ga.$
Observe that if $\tgamma$ is as in \eqref{def-tgamma},  and
\begin{gather*}
P_{\bullet,\tgamma}(h,\sigma,D)=e^{i\soh \tgamma} P_\bullet(h,\sigma,D)e^{-i\soh \tgamma},
\end{gather*}
then
\begin{gather}
P_{\bullet,\tgamma}(h,\sigma,D)-P_{\bullet,\gamma}(h,\sigma,D)= h^2 Q(h,\sigma,D), \label{def-pga2}
\end{gather}
and hence if $p_{\bullet,\tgamma}$ is the principal symbol of $P_{\bullet,\tgamma},$
\begin{gather}
p_{\bullet,\ga}=p_{\bullet,\tgamma}. \label{pgatga}
\end{gather}

We can now state the main result of this section:
 \begin{theorem}\label{soj}
Let $(\intx, g)$ be a non-trapping CCM,  let $T^*(\xo)$ denote the cotangent bundle of the manifold $\xo.$  Let $\rho_L, \rho_R$ be  boundary defining functions of $L$ and $R$ respectively, let $\ka_R$ and $\ka_L$ and $\gamma$ be defined as above. Let $p_{R,\gamma}(m,\nu)$  and 
$p_{L,\gamma}(m,\nu)$ be defined by \eqref{scps-1}, and let  $H_{p_{\bullet,\gamma}},$ $\bullet=R,L,$ be the  corresponding Hamilton vector fields   with respect to the canonical 2-form of $T^*(\xo).$  Let $\wtla_0=\beta_0^* \La_0,$ and let $\La^*$ denote the Lagrangian submanifold obtained by the joint flow-out of   $\wtla_0$ under $H_{p_{R,\gamma}}$ and $H_{p_{L,\gamma}},$ in other words
\begin{gather*}
\La^*= \bigcup_{t_1, t_2\geq 0}\exp(t_1 H_{p_{R,\gamma}}) \circ \exp(t_2 H_{p_{L,\gamma}} )\wtla_0.
\end{gather*}
  Then $\La^*$ is a $C^\infty$ Lagrangian submanifold in the interior of  $T^*(\xo).$  If $\ka$ is constant,  then $\La^*$ extends to a $C^\infty$ compact submanifold with corners of  $T^*(\xo).$  Moreover $\La^*\cap T_{\{\rho_\bullet=0\}}^*(\xo),$ $\bullet=L,R,$ is a $C^\infty$ Lagrangian submanifold of $T^*\{\rho_\bullet=0\}$ and $\La^*\cap T_{\{\rho_R=\rho_L=0\}}^*(\xo)$ is a $C^\infty$ Lagrangian submanifold of $T^*\{\rho_L=\rho_R=0\}.$  
    If $\ka(y)$ is not constant,   $\La^*$ extends to a compact submanifold with corners of  $T^*(\xo),$  the extension is  $C^\infty$ up to the front face, but has polyhomogeneous singularities at $T_{\{\rho_R=0\}}^*(\xo)$ and at $T_{\{\rho_L=0\}}^*(\xo).$  However, 
$\La^*\cap T_{\{\rho_\bullet=0\}}^*(\xo),$ $\bullet=L,R,$ is a $C^\infty$ Lagrangian submanifold of $T^*\{\rho_\bullet=0\}$ and
$\La^*\cap T_{\{\rho_R=\rho_L=0\}}^*(\xo)$ is a $C^\infty$ Lagrangian submanifold of $T^*\{\rho_L=\rho_R=0\}.$  

Moreover, if $(x_0,\xi_0)=(x_{0,3}, \ldots, x_{0,2n+2}, \xi_{0,3}, \ldots, \xi_{0,2n+2})$ are local symplectic  coordinates in $T^*\{\rho_R=\rho_L=0\},$ valid near  $q\in T^* \{\rho_L=\rho_R=0\}\cap \La^*$  such that  $L\cap R \cap \ff=\{ x_{0,3}=0\},$
then there exist symplectic local coordinates $(x,\xi)$ in $T^*(\xo)$ valid near $q$ in which $L=\{x_1=0\},$ $R=\{x_2=0\},$  $\ff=\{x_3=0\},$ and  such that on $\La^*,$ and for  $3\leq m \leq 2n+2,$
\begin{gather}
\begin{gathered}
x_m= x_m(x_1, x_2, x_0,\xi_0) \sim x_{0,m}+  \sum_{j_1, j_2=1}^\infty\sum_{k_1=0}^{j_1} \sum_{k_2=0}^{j_2}
 x_1^{j_1}(\log x_1)^{k_1} x_2^{j_2}(\log x_2)^{k_2} X_{m,j_1,j_2,k_1, k_2}(x_0,\xi_0), \\
\xi_1= \xi_1(x_1,x_2, x_0,\xi_0)\sim   \sum_{j_1, j_2=1}^\infty\sum_{k_1=0}^{j_1+1} \sum_{k_2=0}^{j_2+1}
 x_1^{j_1}(\log x_1)^{k_1} x_2^{j_2}(\log x_2)^{k_2}\Xi_{1,j_1,j_2,k_1,k_2}(x_0,\xi_0), \\
 \xi_2= \xi_2(x_1,x_2, x_0,\xi_0)\sim   \sum_{j_1, j_2=1}^\infty\sum_{k_1=0}^{j_1+1} \sum_{k_2=0}^{j_2+1}
 x_1^{j_1}(\log x_1)^{k_1} x_2^{j_2}(\log x_2)^{k_2}\Xi_{2,j_1,j_2,k_1,k_2}(x_0,\xi_0), \\
\xi_m=\xi_m(x_1, x_2, x_0, \xi_0) \sim \xi_{0,m}+  \sum_{j_1, j_2=1}^\infty\sum_{k_1=0}^{j_1+1} \sum_{k_2=0}^{j_2+1}
 x_1^{j_1}(\log x_1)^{k_1} x_2^{j_2}(\log x_2)^{k_2}\Xi_{m,j_1,j_2,k_1,k_2}(x_0,\xi_0),
\end{gathered}\label{param-A0}
\end{gather}
where the coefficients $X_{\star,j_1,j_2,k_1,k_2}$ and $\Xi_{\star, j_1,j_2,k_1,k_2},$ $\star=1, 2, \ldots, 2n+2,$ are $C^\infty$ functions.
Similar expansions are valid near the right and left faces, away from the corner.
\end{theorem}

The main point in \eqref{param-A0} is that the  $x_m$ variables have  polyhomogeneous  expansions in $(x_1,x_2),$ according to \eqref{as-ph-exp}, but the expansions of the $\xi_m$ variables are only a little worse, and the power of the $\log x $ term can be at most one order higher than the power of $x.$  The  proof of Theorem \ref{soj}  will be done at the end of this section after we prove a sequence of lemmas.   One can see  that the result of Theorem \ref{soj} is independent of the extension of $\ka$ or the choice of $\rho_R$ or $\rho_L.$   If  $\rho_L^*, \rho_R^*$ are  boundary defining functions of the left and right faces, then $\rho_L = \rho_L^* e^{f_L}$ and $\rho_R =  \rho_R^* e^{f_R}$ for some $f_L, f_R \in C^\infty(\xo).$ Therefore  
$\gamma^*- \gamma= \frac{1}{\ka_L} f_L+\frac{1}{\ka_R } f_R,$ and the map $(m, \nu )\mapsto(m, \nu+d(\ga-\ga^*) )$ is a global symplectomorphism of $T^*(\xo),$ and so it does not change the structure of the manifold $\La^*.$  Similarly, if $\frac{1}{\ka_\bullet}-\frac{1}{\tilde\ka_\bullet}=\rho_\bullet f_\bullet,$  with $f_\bullet \in C^\infty,$ $\bullet=R,L,$  then $\gamma-\tilde \gamma= \rho_R \log\rho_R f_R+ \rho_L\log \rho_L f_L,$ and this will only introduce polyhomogeneous terms which  one can check  will not affect the proof.

The main point  in the proof of Theorem \ref{soj} is that $\diag_0$ does not intersect $R$ or $L$ and intersects $\ff$ transversally, see Fig.\ref{fig1}.   We will  show that  if $\ka(y)$ is constant,  and if $\wp_\bullet= \frac{1}{\rho_\bullet}p_{\bullet,\gamma},$ $\bullet=R,L,$  the vector fields  $H_{\wp_\bullet}$ are  $C^\infty$  in the interior of $T^*(\xo)$ and  up to $\ff,$ and are tangent to $\ff.$ Moreover  if $\ka$ is constant, $H_{\wp_\bullet}$ is  $C^\infty$ up to $\p T^*(\xo)$ and  transversal to  $\{\rho_\bullet=0\},$ and so $\La^*$ extends up to 
$\p T^*(\xo).$  When $\ka$ is not constant $H_{\wp_\bullet}$ has logarithmic singularities at $\{\rho_\bullet=0\},$ but its integral curves  are well defined up to $\{\rho_\bullet=0\},$ and hence the manifold $\La^*$ extends up to  $\p T^*(\xo),$ but with polyhomogeneous singularities.    In the case of AHM, this was observed in \cite{ChenHa,Wang}, and in \cite{MSV} in the particular case where $(X,g)$ is a perturbation of the hyperbolic space.

The first lemma describes the behavior of the  vector fields $H_{p_{\bullet,\gamma}},$ $\bullet=R,L$ defined in Theorem \ref{soj} up to the boundary.
  \begin{lemma} \label{vfields}  Let  $p_{\bullet,\gamma}$ be defined in Theorem \ref{soj},  and let  $\wp_\bullet=\frac{1}{\rho_\bullet} p_{\bullet,\gamma},$  $\bullet=R,L.$    If $\ka$ is constant, then
 $\wp_\bullet$ is $C^\infty$ up to $\p T^*(\xo),$  is transversal  to $\{\rho_\bullet=0\}$ and is tangent to the other two faces.  If $\ka$ is not constant
$\wp_\bullet$  has polyhomogeneous singularities at $\{\rho_\bullet=0\},$ but it is smooth up to the other two faces and tangent to both.  If  $x=(x_1, \ldots, x_{n+1})$ are local coordinates in  $\xo$  in which
$ L=\{x_1=0\},$ $R=\{x_2=0\},$  and $\ff= \{x_3=0\},$   and if $\xi=(\xi_1, \ldots, \xi_{2n+2})$ denotes the dual variable to $x,$ then $H_{\wp_L}$  and $H_{\wp_R}$ satisfy
\begin{gather}
\begin{gathered}
H_{\wp_L}=  A_{1}(x,\xi) \p_{x_1}+ A_{2}(x,\xi)  x_2 \p_{x_2} +  A_{3}(x,\xi) x_3 \p_{x_3} +  \sum_{k=4}^{2n+2} A_{k} (x,\xi)\p_{x_k} + \\
   \sum_{k=1}^{2n+2} \wta_{k} (x,\xi)\p_{\xi_k},\\
\text{ where } \\
A_1(x,\xi)= -\ka_L + G_1(x) x_1^2\log x_1+ x_1 \sum_{k=1}^{2n+2} B_{1k}(x) \xi_k,  \\
A_{j}(x,\xi)=  F_j(x) + G_j(x) x_1 \log x_1+  x_1 \sum_{k=1}^{2n+2} B_{jk}(x) \xi_k, \; 2\leq j \leq 2n+2, \\
\wta_j(x,\xi)=  \wtf_j(x) + E_j(x) \log x_1 + D_j(x) (\log x_1)^2+ \sum_{k=1}^{2n+2} \wtb_{jk}(x) \xi_j + \\
\sum_{k=1}^{2n+2} (\log x_1) C_{jk}(x) \xi_j+ \sum_{k,l=1}^{2n+2} F_{jkl}(x) \xi_j \xi_k,
\end{gathered}\label{fields}
\end{gather}
where  $F_j,$ $\wtf_j,$ $G_j,$ $B_{jk},$  $\wtb_{jk},$ $C_{jk},$ $D_j,$ $E_j$ and  $F_{jkl},$  and are $C^\infty$ functions. The formula for the vector field $H_{\wp_R}$ is obtained from this one by switching  $x_1$ and $x_2,$  $\xi_1$ and $\xi_2.$  When $\ka$ is constant  all the coefficients of $\log x_1$ and 
$(\log x_1)^2$  are equal to zero.
\end{lemma}
 \begin{proof}    First, observe that if $\tilde x$ are coordinates in which $L=\{\tilde{x}_1=0\},$ $R=\{\tilde x_2=0\}$ and $\ff=\{\tilde x_3=0\},$ then $x_j=X_j(\tilde x) \tilde{x}_j,$ $j=1,2,3,$ where $X_j(\tilde x)>0,$ The dual variables $\tilde \xi_j$ would be linear combinations of $\xi_j$ with $C^\infty$ coefficients depending on $x$ only. Therefore the vector fields $H_{\wp_\bullet}$ would have the same form in the new coordinates.
 
   We choose a boundary defining function $x$ such that \eqref{prod} holds.   In this case $\Delta_g$ is given by \eqref{formlap} and  therefore
 \begin{gather}
 \begin{gathered}
h^2(\Delta_{g(z)}-\knsq)-1 = \left(\ka(y)^2 (x hD_x)^2+ x^2 H^{jk}(x,y) hD_{y_j}hD_{y_k}-1\right)- \\ ih \ka(y)^2 (n+ xF(x,y)) x hD_x + i h \sum_j B_{k} x^2 h D_{y_k} - h^2 \knsq.
 \end{gathered}\label{semiclass-Lap}
 \end{gather}
 We  will find $P_{L} (h,m,D-d\gamma)=P_{L,\gamma}(h,m,D),$  compute  its semiclassical principal symbol and its Hamiltonian. It is only necessary to work in local coordinates valid near $\p(\xo)$ and we divide it in four regions and work in projective coordinates valid in each region:  
\begin{enumerate}[{A.}]
\item  Near $L$ and away from $R\cup \ff,$ or near $R$ and away from $L \cup \ff.$ 
\item  Near $L \cap \ff$ and away from $R$, or near $R\cap \ff$ and away from $L.$ 
\item  Near $L\cap R$ but away from $\ff.$ 
\item  Near $L\cap R \cap \ff.$
\end{enumerate}

  First we analyze region A, near $ L $ but away from $R$ and $\ff.$  The case near $ R$ but away from $ L$ and $ \ff$ is identical.   Since we are away from $R,$ we have $\rho_R>\del,$ for some $\del>0,$  and hence $\log \rho_R$ is $C^\infty.$ In this region we may take $x_1=x$ as a defining function of $L,$  we set $\gamma= \frac{1}{\kappa(y)}\log x_1$.   We shall denote the other coordinates $y=(x_2,\ldots, x_{2n+2})$ and the respective dual variables by $\eta.$    Even though this does not match the notation of the statement of the lemma,  it would be more confusing if we renamed the $y$ variables. 
  
   Also observe that the map $(x,\xi) \longmapsto (x, \xi-d(\frac{1}{\ka_R}\log\rho_R))$ is $C^\infty$  in the region where $\rho_R>\del,$  it does not affect the form of the vector fields in \eqref{fields}, hence  the statements about $\wp_L$  in the lemma are true in this region whether we take $\gamma=\frac{1}{\ka(y)} \log x_1 +\frac{1}{\ka_R} \log \rho_R.$  In the case near $R$ but away from $L$ and $\ff$ one sets 
  $x_2=x'$ and $\gamma=\frac{1}{\ka(y')} \log x_2.$

We see from \eqref{semiclass-Lap} that 
  \begin{gather*}
  \ha  p_{L,\gamma}(x_1,y,\xi_1,\eta)=\ha(p_{L}(x_1,y, \xi_1-\p_{x_1}\gamma, \eta-\p_y\gamma)-1)= \\ \ha\ka(y)^2(x_1\xi_1- \frac{1}{\ka(y)} )^2 + \ha x_1^2 H^{jk}(x_1,y)( \eta_j -  a_j  \log x_1)( \eta_k - a_k \log x_1)-\ha, \\
   \text{ where }  a_j=\p_{y_j} \ka(y)^{-1},
  \end{gather*}
   and so
  \begin{gather*}
\wp_L=\frac{1}{x_1}  p_{L,\gamma}=  -\ka(y)\xi_1+\ha \ka(y)^2x_1\xi_1^2+ \ha x_1 H^{jk}(x_1,y)(\eta_j-a_j\log x_1)(\eta_k - a_k\log x_1).
  \end{gather*}
Of course there are no $\log x_1$ terms when $\ka$ is constant,   and it follows from a direct computation that $H_{\wp_L}$ is of the desired form.

Next we work in region B near $L\cap \ff,$ but away from $ R$.  The case near $R\cap \ff$ but away from $ L$ is very similar.   In this case, $\rho_R=x'/R>\del,$ and so it is more convenient to use projective coordinates
\begin{equation}\label{eqc1}
x_1 = \frac{x}{x'},\ \ x_3= x', \;\   Y = \frac{y - y'}{x'}, \text{ and }  y'. 
\end{equation}

In this case, $X$ is a boundary defining function for $L$ and $x_3$ is a boundary defining function for $\ff,$ and it suffices to take $\gamma=a(y'+x_3Y) \log X,$ where 
$a=\frac{1}{\ka}.$ Therefore,   if $\xi_j$ and $\eta_j$ denote the dual variables to $x_j$ and $Y_j,$
\begin{gather*}
p_{L,\gamma}(x_1,Y,x_3,y',\xi,\eta)=p_{L}(x_1,Y,x_3,y',\xi_1- \p_{x_1}\ga, \eta-\p_Y\ga)= \\ \ha \ka^2(y'+x_3Y)(x_1\xi_1-\frac{1}{\ka(y'+x_3Y)})^2 + \\ \ha x_1^2 H^{jk}(x_3X,y'+x_3Y)( \eta_j- a_j(y'+x_3Y) \log x_1) ( \eta_k- a_k(y'+x_3Y) \log x_1)-\ha,
\end{gather*}
where $a_j=\p_{y_j} a,$ and so we conclude that
\begin{gather*}
\wp_L= -\ka(y'+x_3Y) \xi_1 + \ha \ka^2(y'+x_3Y) x_1\xi_1^2 + \\  \ha x_1  H^{jk}(x_3x_1,y'+x_3Y)( \eta_j-a_j \log x_1)(\eta_k-a_k \log x_1),
\end{gather*}
and hence \eqref{fields} follows from a direct coputation. Again, when $\ka$ is  constant,  $a_j=0,$ and  there are no $\log x_1$  terms, so  $H_{\wp_L}$ is $C^\infty$ and  $H_{\wp_L}$ is transversal to $ L.$

In region C, near $L\cap R$ and away from $ \ff,$  $x=x_1$ and $ x'=x_2$ are boundary defining functions for $ L$ and  $ R$ respectively. In this case, as discussed above, we may define
\beq
\gamma = a(y)\log x_1 +a(y') \log x_2, \text{ where } a(y)=\frac{1}{\ka(y)}.
\eeq
Since we are working with $P_{L,\ga},$ the operator does not have derivatives in $D_{x_2}$ or $D_{y'},$ and the computations are exactly the same as in region A.
Since $H_{\wp_L}$ does not have a term in $\p_{x_3},$ it is tangent to $\{x_3=0\}= R.$

Finally, we analyze  region D near  the co-dimension $3$ corner $ (L\cap \ff\cap R).$   We work in projective coordinates, and as  in \cite{MSV},  set 
$\rho_{\ff}=y_1 - y_1' \geq 0$ and define projective coordinates
\begin{equation}\label{eqc2}
x_3 = y_1 - y_1', \ \  x_1 = \frac{x}{y_1 - y_1'}, \ \  x_2 = \frac{x'}{y_1 - y_1'},\ \ y'  \text{ and } Y_j = \frac{y_j - y_j'}{y_1 - y_1'}, \ \ j = 2, 3,\cdots n.
\end{equation}

Here $x_3=\rho_\ff,$ $x_1=\rho_L$ and $x_2=\rho_R,$  are boundary defining functions for $\ff$  $ L$ and $R$ faces respectively. 
Then 
\begin{gather}
\begin{gathered}
x\p_x= x_1\p_{x_1}, \;\ x\p_{y_j}= x_1 \p_{Y_j}, \;\ j\geq 2, \\
x\p_{y_1}= x_1(x_3\p_{x_3}-x_1\p_{x_1}-x_2\p_{x_2}- Y_j\p_{Y_j}),
\end{gathered}\label{comp-cor}
\end{gather}
where a repeated index indicate sum over that index.  In these coordinates,
\begin{gather*}
\gamma= \frac{1}{\ka_R} \log x_2- \frac{1}{\ka_L} \log x_1 \;\ \ka_R=\ka(y'), \; \ka_L=\ka( x_3+y_1',y_2'+ x_3 Y_2, \ldots, y_{n}'+x_3 Y_n)
\end{gather*}
 If  $(\xi_1, \xi_2,\xi_3,\eta_j, \eta_j')$ denote the dual variables to 
$(x_1,x_2,x_3,Y,y'),$  then substituting $\xi_j$ by $\xi_j -\p_{x_j}\ga,$ $\eta_j$ by $\eta_j-\p_{y_j}\ga,$ we find that
\begin{gather*}
 p_{L,\ga}=  \ha \ka_L^2(x_1\xi_1- \frac{1}{\ka_L})^2+\ha x_1^2 H^{jk}(x_1x_3, y'+ x_3 Y)\mcu_j \mcu_k-\ha,\\
 \text{ where }
 \mcu_1=x_3\xi_3-x_1\xi_1-x_2\xi_2-Y_j\eta_j +a_L+a_R+a_1x_3\log x_1, \;\  a_\bullet=\frac{1}{\ka_\bullet},\\
a_1=\p_{y_1'} a_L, \; a_j=\p_{Y_j} a_L,\;  \mcu_j= Y_j\eta_j - a_j \log x_1, \, 2\leq j \leq n,
 \end{gather*}
and therefore,
\begin{gather}
\wp_L= -\ka_L \xi_1+ \ha \ka_L^2x_1\xi_1^2+\ha x_1 H^{jk}(x_1x_3, y'+ x_3 Y)\mcu_j \mcu_k. \label{formula-pl}
\end{gather}

Now we have to verify that  $H_{\wp_L}$ matches \eqref{fields}.      Let us consider the coefficients $A_j$ first.  They are obtained by differentiating $\wp_L$ in $\xi_j$ or $\eta_j.$ But, $\wp_L$ is a polynomial of degree two in the dual variables $(\xi,\eta)$ with $C^\infty$ coefficients depending on $(x,Y,y').$  So the $A_j$ should be polynomials of degree one in $(\xi,\eta)$ with $C^\infty$ coefficients depending on the base variables.   Notice that the second order terms are of the form $x_1 \xi_j \xi_k$ or $x_1 \xi_j \eta_k,$ so the terms of degree one in $\xi$ are of the form $x_1B_{jk}\xi_k.$  As for the log terms, they only appear in $x_1 \mcu_j\mcu_k$ terms,  but
$\p_{\xi_m}(x_1\mcu_j \mcu_k)= x_1\mcu_j \p_{\xi_m} \mcu_k+  x_1\mcu_k\p_{\xi_m} \mcu_j,$ and this shows the general form of $A_j.$    Now the special form of $A_1$ comes from differentiating  the first two terms in \eqref{formula-pl} with respect to $\xi_1.$  Similarly,  only the terms $x_1 \mcu_1 \mcu_j,$ $1\leq j \leq n,$ contain $\xi_2$ and $\xi_3,$ but in fact these show up as $x_2\xi_2$ and $x_3\xi_3,$ so when we differentiate this product all the terms have a factor $x_1x_2$ in the case of $A_2(x,\xi)$ and $x_1x_3$ in the case of $A_3(x,\xi).$

For the $\wta_j(x,\xi)$ terms,  when we differentiate $\wp_L$ in $(x,Y,y')$ get a polynomial fo degree at most two in $(\xi,\eta).$ This describes the general form of $\wta_j(x,\xi).$  Perhaps the only issue is the appearance of the log terms, and nothing worse. When  $\p_{x_1}$ hits the log term in $\mcu_j,$ the $x_1$ term in front of $H^{jk}$ cancels the term in $\frac{1}{x_1}.$   The $\log x_1$ terms come from either
$\p_{x_j}(x_1H^{jk}\mcu_j \mcu_k)$ or $\p_{Y_j} (x_1 H^{jk} \mcu_j \mcu_k),$  and are as in \eqref{fields}. This  concludes the proof of the Lemma.
  \end{proof}

Next we need to prove that integral curves of $H_{\wp_\bullet}$  extend up to $\{\rho_\bullet=0\}$ and have the polyhomogeneous expansions stated in \eqref{param-A0}. We want to use $x_1$ as the parameter along the integral curves of $H_{\wp_\bullet}.$ According to \eqref{fields}, the coefficient of $\p_{x_1}$ of $H_{\wp_L}$ is equal to
\begin{gather*}
A_1(x,\xi)=-\ka_L(x) + G_1(x) x_1^2\log x_1+ \sum_{k} B_{m,k}(x) x_1 \xi_k,
\end{gather*}
and since  $\ka(y)\geq \ka_0>0,$ it follows that near any point  $(x_0,\xi_0)$ with $x_{01}=0,$ we have the following asymptotic expansion
\begin{gather*}
A_1(x,\xi)^{-1} \sim  -\frac{1}{\ka_L}+ \sum_{k, \alpha} (x_1\log x_1)^k (x_1\xi)^\alpha H_{k,\alpha}(x), \text{ where  } \\ H_{k,\alpha} \in \CI, \; 
\alpha=(\alpha_1, \ldots, \alpha_{2n+2}) \in \mn^{2n+2}, \; (x_1\xi)^\alpha= (x_1\xi_1)^{\alpha_1}\ldots(x_1 \xi_{2n+2})^{\alpha_{2n+2}}.
\end{gather*}
So, if we divide $H_{\wp_L}$ by $A_1(x,\xi)$ we obtain a vector field which has the following asymptotic expansion  near $(x_0,\xi_0):$
\begin{gather*}
A_1(x,\xi)^{-1} H_{\wp_L}\sim  
\p_{x_1}+ \sum_{j=2}^{2n+2} A_j(x,\xi) B_{j}(x, x_1\log x_1, x_1\xi) \p_{x_j} + \sum_{j=1}^{2n+2} \wta_j(x,\xi) \wtb_{j}(x, x_1\log x_1, x_1\xi) \p_{\xi_j},\\
\text{ where }  A_j(x,\xi) \text{ and } \wta_j(x,\xi) \text{ are as in \eqref{fields} and }\\
B_j(x, x_1\log x_1, x_1 \xi) \sim \sum_{k,\alpha} (x_1\log x_1)^\alpha (x_1\xi)^\alpha B_{j,k,\alpha}(x), \;\  B_{j,k,\alpha} \in \CI, \\
\wtb_j(x, x_1\log x_1, x_1 \xi) \sim \sum_{k,\alpha} (x_1\log x_1)^\alpha (x_1\xi)^\alpha \wtb_{j,k,\alpha}(x), \;\  \wtb_{j,k,\alpha} \in \CI.
\end{gather*}
Now we need the following result about polyhomogeneous odes:
\begin{lemma} \label{asym-expL} Let $U, V \subset \mr^n$ be open subsets with $\overline{U}\subset V,$ and let $F_m(x,y),$ $1\leq m \leq n$ be such that for  there exist $C>0,$ $M>0$ and $N\in \mn$ be such that 
\begin{gather}
\begin{gathered}
\left|F_m(x,y)\right| \leq  C(-\log x)^N ,\;\  x \in (0,1), \;\ y \in V, \\
\left| \nabla_y F_m(x,y)\right| \leq M  (-\log x)^N \;\ x \in (0,1), \;\ y \in V. 
\end{gathered} \label{Lip}
\end{gather}
Then for any $p=(p_1,\ldots,p_n)\in U$ there exists $\eps>0$ such that  for $1\leq m\leq n,$ the initial value problem
\begin{gather}
\begin{gathered}
\frac{dy_m}{dx}= F_m(x,y),  \;\ 1\leq m \leq n, \\
y_m(0)=p_m
\end{gathered}\label{IVP}
\end{gather}
 has a unique solution $y(x)=(y_1(x),\ldots, y_n(x)),$  with $y(x)\in V$ for $x\in [0,\eps).$  Moreover,  if  for
 $x\in (0,1),$ $y\in V,$  one has $y=(y',y''),$ $y'=(y_1,\ldots, y_k),$ $y''=(y_{k+1}, \ldots, y_n),$ and the functions $F_m(x,y)$ have asymptotic expansions of the type
\begin{gather}
\begin{gathered}
F_m(x,y)\sim \sum_{j=0}^\infty (x \log x)^j B_{m,j}(x,y',xy''), \; B_{m,j}\in \CI  \text{ if } 1\leq m \leq k, \text{ and } \\
F_m(x,y)\sim  \sum_{j=0}^\infty (x \log x)^j( B_{m,j}(x,y', xy'')+  (\log x)  C_{m,j}(x,y',xy'') + (\log x)^2 D_{m,j}(x,y', xy'')) +  \\
 \sum_{j=0}^\infty (x \log x)^j(\sum_{r=k+1}^n y_r E_{r,m,j}(x,y',xy'')+ (\log x ) y_r \widetilde{E}_{r,m,j}(x,y',xy'')+
  \sum_{r,s=k+1}^n y_r y_s F_{r,s,m,j}(x,y',xy'')), \\
\text{ if } k+1\leq m \leq n, \; B_{m,j}, C_{m,j}, D_{m,j}, E_{r,m,j}, F_{r,s,m,j} \in \CI,
\end{gathered}\label{expfm}
\end{gather}
then $y_m(x),$ $1\leq m \leq n,$  have the following polyhomogeneous expansions at $\{x=0\}:$
\begin{gather}
\begin{gathered}
y_m(x)-p_m \sim \sum_{j=1}^\infty \sum_{k=0}^{j} x^j(\log x)^k  Y_{j,k,m}(p), \;\ Y_{j,k,m}\in C^\infty(U), \;\  1\leq m \leq k, \\ 
y_m(x)-p_m \sim \sum_{j=1}^\infty \sum_{k=0}^{j+1} x^j(\log x)^k  Y_{j,k,m}(p), \;\ Y_{j,k,m}\in C^\infty(U), \;\  k+1\leq m \leq n.
\end{gathered}\label{phexp}
\end{gather}
in the sense that for any $J\in \mn$ and $\mu>0,$ 
\begin{gather}
\begin{gathered}
\left|y_m(x)-p_m-\sum_{j=1}^J \sum_{k=0}^{j+\bullet} x^{j}(\log x)^k  Y_{j,k,m}(p)\right|\leq C(J,\eps) x^{J-\mu},\\
\bullet=0 \text{ if } m \leq k, \text{ and } \bullet=1 \text{ otherwise}.
\end{gathered}\label{defexp}
\end{gather}
\end{lemma}
\begin{proof}    We will use a contraction argument in an appropriately defined space of functions to prove the existence and uniqueness of the solution. We will then show that the asymptotic expansion is valid for the unique solution.  Let $\del>0$ be small enough so that  
\begin{gather*}
Q(p,\del)=\{y\in V: |y_j-p_j| \leq \del, \;\ 1\leq j \leq n\} \subset V, \;\ \forall p \in U,
\end{gather*}
and for $\eps>0$ let 
\begin{gather*}
\mcc=C\left([0,\eps]; Q(p,\del)\right)=\{ \phi: [0,\eps] \longrightarrow Q(p,\del) \text{ continuous } \} \\
\text{ equipped with the norm }  ||\phi||=\sup_{x\in [0,\eps]} |\phi(x)|
 \end{gather*}
For $\phi\in \mcc$ we define the map $T(\phi)=(T_1(\phi),\ldots, T_n(\phi)),$ where 
\begin{gather*}
T_m(\phi(x))= p_m+\int_0^x  F_m(t, \phi(t))\; dt, \;\ 1\leq m \leq n.
\end{gather*}
Then in view of   the second inequality in \eqref{Lip}, given two functions $\phi,\psi \in \mcc,$
\begin{gather*}
\left|T_m(\phi)(x)-T_m(\psi)(x)\right| \leq   M ||\phi-\psi|| \int_0^x  (-\log t)^N \; dt
\end{gather*}
But since
\begin{gather*}
\int_0^x (-\log t)^N\; dt= x N! \sum_{r=0}^N \frac{1}{r!} (-\log x)^{r}, 
\end{gather*}
we conclude that  for $x<1,$
\begin{gather}
\int_0^x (-\log t)^N\; dt \leq N! x(-\log x)^N, \label{inteta}
\end{gather}
and therefore
\begin{gather*}
\left|T_m(\phi)(x)-T_m(\psi)(x)\right| \leq M N! x(-\log x)^N ||\phi-\psi||.
\end{gather*}
The function $x(-\log x)^N$ is increasing in the interval $(0,e^{-N}),$ and so if $x < \eps$  and $\eps<e^{-N},$
\begin{gather*}
||T_m(\phi)(x)-T_m(\psi)(x)||\leq M N !  \eps(-\log \eps)^N ||\phi-\psi||
\end{gather*}
and we pick $\eps>0$ such that   $\eps<e^{-N}$ and 
\begin{gather}
M N! \sqrt{n}  \eps(-\log \eps)^N<1. \label{cheps}
\end{gather} 
 With this value of $\eps,$ we need to find $\del$ which guarantees that
$T: \mcc \longmapsto \mcc.$  Now we use the first inequality in  \eqref{Lip} to deduce that if $Q(p,\del)\subset V,$ and $\phi\in \mcc,$
\begin{gather*}
\left|T_m(\phi(x))-p_m\right| \leq C  \int_0^x (-\log t)^N \; dt.
\end{gather*}
Again because of \eqref{inteta} we  conclude that 
\begin{gather*}
|T_m(\phi)-p_m| \leq C N! \eps (-\log \eps)^N.
\end{gather*}
We pick 
\begin{gather}
\begin{gathered}
\del=C N! \sqrt{n} \eps (-\log \eps)^N, \text{ with } \eps \text{ small enough }  \\
\text{ such that } Q(p,\del)\subset V, \;\ \forall p \in U.\
\end{gathered}\label{chdel}
\end{gather}
  Therefore with  $\eps$ and $\del$ such that \eqref{cheps} and \eqref{chdel} hold, we have shown that  $T:\mcc \longmapsto \mcc,$ and 
  $T$ is a contraction.  Since $\mcc$ is a complete metric space, $T$  has a unique fixed point $y(x)=(y_1(x),\ldots,y_n(x))$  which satisfies
\begin{gather*}
y_m(x)=p_m +\int_0^x  F_m(t, y(t))\; dt, \;\ 1\leq m \leq n.
\end{gather*}
Therefore, $y$ is differentiable in $x>0$ and satisfies \eqref{IVP}. We still need to prove that $y_m(x)$ satisfies \eqref{phexp}.

Let $\phi, \psi \in \mcc,$ and consider the sequence $T^J\phi$ and $T^J\psi,$ $J\in \mn.$ From the definition 
\begin{gather*}
 T_m^J\phi(x)-T_m^J\psi(x)=  \int_0^x  \left(F_m(t, T^{J-1} \phi(t))- F_m(t, T^{J-1} \psi(t))\right)  \; dt, \;\ 1\leq m \leq n,
\end{gather*}
and so from \eqref{Lip},
\begin{gather*}
\left|T_m^J\phi(x)-T_m^J\psi(x)\right| \leq M \int_0^x (-\log t)^N |T^{J-1}\phi(t)-T^{J-1}\psi(t)|\; dt.
\end{gather*}
Iterating this formula we obtain
\begin{gather*}
\left| T_m^J\phi(x)-T_m^J\psi(x)\right| \leq \\ M^J \int_0^x(-\log t)^N \int_0^{t} (-\log t_1)^N  \ldots \int_0^{t_{J-1}} (-\log t_J)^N
|\phi(t_J)-\psi(t_J)| \; dt_J \; dt_{j-1} dt_{j-2}\ldots dt_{1} dt \leq \\
 \left(M \int_0^x (-\log t)^N \; dt \right)^J ||\phi-\psi||.
 \end{gather*}
 We then deduce from \eqref{inteta} that
 \begin{gather*}
 \left| T_m^J\phi(x)-T_m^J\psi(x)\right|  \leq  ( M N! x (-\log x)^N)^J ||\phi-\psi||.
 \end{gather*}

Now, for $J$ fixed and $r\in \mn,$ one can write
\begin{gather*}
T^{J+r+1}(\phi)- T^J(\phi)= \sum_{\alpha=0}^r (T^{J+\alpha}(T\phi)-T^{J+\alpha}(\phi)),
\end{gather*}
and we  conclude that, with our choice of $\eps,$
\begin{gather}
\begin{gathered}
\left|T^{J+r+1}(\phi)(x)- T^J(\phi)(x)\right| \leq  \sum_{\alpha=0}^r |T^{J+\alpha}(T\phi)-T^{J+\alpha}(\phi)| \leq \\
\sum_{\alpha=0}^r [M N! x(-\log x)^N]^{J+\alpha}\; ||T\phi-\phi|| \leq \\
  ||T\phi-\phi|| (M N! x(-\log x)^N)^{J} \sum_{\alpha=0}^\infty( MN!\eps (-\log\eps)^N)^\alpha=\\
   ||T\phi-\phi|| (M N! x(-\log x)^N)^{J} \frac{1}{1-MN!\eps (-\log\eps)^N}.
\end{gathered} \label{auxasy}
\end{gather}

Since $T$ is a contraction, if one picks any $\phi\in \mcc,$ the sequence  defined by $\phi_k= T(\phi_{k-1}),$ and $\phi_0=\phi$ converges to the solution $y$ uniformly.  So we pick $\phi=p,$  and by taking the limit as $r\rightarrow \infty$ in  \eqref{auxasy} we deduce that for any $L\in \mn$ and $\mu>0,$ there exists $C_{L,\mu}>0$ such that
\begin{gather*}
\left| y(x)- T^L(p)(x)\right| \leq C(L,\mu) x^{L-\mu},
\end{gather*}
and so we only need to show that,  fixed $L,$ then in the sense of \eqref{defexp}, 
\begin{gather}
\begin{gathered}
T_m^L(p)-p_m \sim \sum_{j=1}^\infty \sum_{k=0}^{j} x^j(\log x)^k F_{L,j,k}(p) \;\ F_{L,j,k}(p) \in C^\infty(U), \;\ m \leq k \text{ and }\\
T_m^L(p)-p_m \sim \sum_{j=1}^\infty \sum_{k=0}^{j+1} x^j(\log x)^k F_{L,j,k}(p) \;\ F_{L,j,k}(p) \in C^\infty(U), \;\  k+1\leq m \leq n.
\end{gathered}\label{iterG}
\end{gather}
We prove this by induction in $L$ and begin with $L=1.$   From the definition of $T,$
\begin{gather*}
T_{m}(p)(x) - p_m= \int_0^x  F_m(t, p)\; dt, \;\ 1\leq m \leq n.
\end{gather*}

But  since the coefficients of the expansions in \eqref{expfm} are $\CI$ functions of the type $B(x,y',xp''),$ then   
$B(x,p',xp'') \sim \sum_{j=0}^\infty B_j(p) x^j,$  $B_j(p)=\frac{1}{j!} \p_x^j B(x,p', xp'')|_{x=0},$
\begin{gather*}
\begin{gathered}
F_{m} (t,p) \sim \sum_{j,l=0}^\infty  t^l  (t \log t)^j B_{m,j,l}(p), \;\ 1\leq m \leq k, \\
F_m(t, p) \sim \sum_{j,l=0}^\infty t^j (t\log t)^l(\wtf_{m,j,l}(p) +( \log t) \wtf_{1,m,j,l}(p) + (\log t)^2  \wtf_{2,m,j,l}(p),) \;\ k+1\leq m \leq n.
\end{gathered}
\end{gather*}
But since for $j,k \in \mn_0,$ 
\begin{gather}
\int_0^x  t^j(\log t)^k dt= \frac{x^{j+1}}{j+1} \sum_{r=0}^k \frac{(-1)^r}{(j+1)^r}\frac{k!}{(k-r)!} (\log x)^{k-r}, \label{int-log}
\end{gather}
and so \eqref{iterG} is satisfied for $L=1.$

Suppose that \eqref{iterG}  is correct for $L.$ So we write
\begin{gather}
T_m^{L+1}(p)(x)-p_m=  \int_0^x \left(F_m(t, T^L(p)(t))- F_m(t,p)\right) \; dt. \label{diff-Int}
\end{gather}
Let us consider the case $1\leq m \leq k$ first. In view of  \eqref{expfm},  
\begin{gather*}
F_m(t,T^L(p)(t))- F_m(t,p) \sim \sum_{j=0}^\infty (t\log t)^j( B_{m,j}(t,  T_L(p)'(t), tT^L(p)''(t))-B_{m,j}(t,p', tp''),
\end{gather*}
But since $B_{m,j}\in C^\infty,$ its Taylor series expansion gives that
\begin{gather*}
B_{m,j}(t,  T_L(p)'(t), tT^L(p)''(t))-B_{m,j}(t,p', tp'')\sim  \\ \sum_{l,j,\beta_1,\beta_2} d_{m,j,l,\beta_1,\beta_2} t^{l}(T_m^L(p)'(t)-p_m')^{\beta_1}(tT_m^L(p)''(t)-tp_m')^{\beta_2}.
\end{gather*}
But since \eqref{iterG} holds for $L,$
\begin{gather*}
T_m^L(p)'(t)-p_m'\sim \sum_{j=0}^\infty\sum_{r=0}^j t^j(\log t)^rF_{L,m,j,r}(p), \;\ 1 \leq m \leq k, \\
tT_m^L(p)''(t)-tp_m''\sim \sum_{j=0}^\infty\sum_{r=0}^j t^j(\log t)^r F_{L,m,j,r}(p), \;\ k+1\leq m \leq n.
\end{gather*}
and 
\begin{gather}
(\sum_{j=0}^\infty \sum_{l=0}^j t^j (\log t)^l A_{j,l})^N \sim  \sum_{j=0}^\infty \sum_{l=0}^j t^j (\log t)^l A_{j,l,N}.  \label{power-of}
\end{gather}
So we conclude that
\begin{gather}
T_m^{L+1}(p)(x)-p_m \sim  \sum_{j=0}^\infty\sum_{r=0}^j \int_0^x  t^j (\log t)^r \wtf_{m,j,r}(p)  \; dt, \;\ \wtf_{m,j,r} \in \CI,
\end{gather}
and the result follows from \eqref{int-log}.

The  case of $k+1\leq m \leq n$ is similar.  Again, we start from  \eqref{diff-Int}, and observe that the terms in 
$(x\log x)^j((\log x)  C_{m,j}(x,y',xy'') + (\log x)^2 D_{m,j}(x,y', xy''))$ and $(x\log x)^j (\log x) y_r \widetilde{E}_{r,m,j}(x,y',xy'')$ can be handled exactly as above, leading to an expasnion of the form \eqref{iterG} because of the additional power of $\log x.$  We will analyze terms like  $\sum_{j} (x\log x)^jy_r y_s \wtf(x,y', xy''),$ $\wtf\in \CI,$ and we have to   consider the integral
\begin{gather}
\int_0^x (t\log t)^j \left( T_r(p)(t) T_s(p)(t) \wtf(t,T^L(p)'(t), tT^L(p)''(t))- p_r p_s \wtf(t, p', tp'')\right) \; dt \label{diff-Int1}
\end{gather}
We write
\begin{gather*}
T_r(p)(t) T_s(p)(t) \wtf(t,T^L(p)'(t), tT^L(p)''(t))- p_r p_s \wtf(t, p', tp'')= \\ T_r(p)(t) T_s(p)(t)( \wtf(t,T^L(p)'(t), tT^L(p)''(t))- \wtf(t, p', tp''))+ \\
(T_r(p)(t)-p_r) T_s(p)(t)\wtf(t, p', tp'')+ (T_s(p)(t)-p_s)p_r\wtf(t, p', tp''),
\end{gather*}
and the same argument used above plus \eqref{iterG} gives that
\begin{gather*}
T_r(p)(t) T_s(p)(t) \wtf(t,T^L(p)'(t), tT^L(p)''(t))- p_r p_s \wtf(t, p', tp'') \sim \sum_{j=0}^\infty \sum_{l=0}^{j+2} t^j(\log t)^l Y_{j,l,L}(p),
\end{gather*}
 As before,  we substitue this into \eqref{diff-Int1} and use \eqref{int-log} to conclude that \eqref{iterG} holds for $L+1.$
\end{proof}

Now we  conclude the proof of Theorem \ref{soj}.

\begin{proof}   Let $\La$ be the Lagrangian manifold defined by \eqref{defla0},  then by definition, 
$\beta_0^*\La$ is obtained by the joint flow-out of 
\begin{gather*}
\wtla_0=\beta_0^*(\La)=\beta_0^*\left( \{ (z,\zeta, z',\zeta'): z=z', \; \zeta=-\zeta', \; |\zeta|_{g^*(z)}=1\}\right)
\end{gather*}
under $H_{p_{R}}$ and $H_{p_{L}}.$
In other words
 \begin{gather*}
\beta_0^*\La = \bigcup_{t_1, t_2 \geq 0} \exp(t_1 H_{p_{L}} )\circ\exp (t_2 H_{p_{ R}}) \La_0.
 \end{gather*}

On the other hand, also by definition,  with $p_{\bullet,\gamma}$ defined in Theorem \ref{soj},
 \begin{gather*}
\La^* = \bigcup_{t_1, t_2 \geq 0} \exp(t_1 H_{p_{L,\gamma}} )\circ\exp (t_2 H_{p_{R,\gamma}}) \wtla_0.
 \end{gather*}

But since  in the interior of $T^*(\xo),$ the map  $S_\ga$ defined in \eqref{shift-map} preserves the symplectic structure in the interior, and since $p_{\bullet,\ga}=S_\ga^*p_{\bullet}= p_\bullet\circ S_\ga,$ it  follows  that $\beta_0^* \La= S_\ga(\La^*),$ or in other words,
 \begin{gather*}
 (m,\nu)\in \La^* \Leftrightarrow  (m,\nu-d\gamma)\in \beta_0^*\La,
 \end{gather*}
  and so we conclude that
\begin{gather}
\begin{gathered}
\La^*-d\gamma=\beta_0^*\La  \text{ in the interior of } T^*(\xo). 
\end{gathered} \label{shift}
\end{gather}
But, we know that due to the non-trapping assumption $\La$ is $C^\infty$ in $T^*(\intx \times \intx),$ and since in the interior of $\xo,$  $\beta_0$ is a diffeomorphism and $\gamma$ is $C^\infty,$  it follows that $\La^*$ is a $C^\infty$ Lagrangian submanifold in the interior of $T^*(\xo).$

 In the interior of $ \xo,$ $p_{\bullet,\gamma}$  vanishes on $\La^*,$  $\bullet=L,R,$ and  since $\wp_\bullet=\frac{1}{\rho_\bullet} p_{\bullet,\gamma},$ it follows that the integral curves of $H_{\wp_{\bullet}}$ and the integral curves of $H_{p_{\bullet,\ga}}$ coincide 
 on $\La^*.$  Therefore, in the interior of $ \xo$ and up to the front face,  $\La^*$ is the union of integral curves of $H_{\wp_{L}}$ and $H_{\wp_R}$ emanating from $\La_0.$    The vector fields $H_{p_R}$ and $H_{p_L}$ commute,  hence their Poisson bracket $\{p_R,p_L\}=H_{p_R}p_L=0,$ since $S_\ga$ is preserves the symplectic structure, and $\{p_{R,\ga},p_{L,\ga}\}=0$ and hence $[H_{p_{R,\ga}},H_{p_{L,\ga}}]=0.$  On the other hand, 
$\wp_\bullet=\rho_\bullet^{-1}p_{\bullet,\ga},$ and hence
\begin{gather*}
H_{\wp_\bullet}= \rho_\bullet^{-1} H_{p_{\bullet,\ga}}+ p_{\bullet,\ga} H_{\rho_\bullet^{-1}},
\end{gather*}
and 
\begin{gather*}
H_{\wp_\bullet}= \rho_\bullet^{-1} H_{p_{\bullet,\ga}} \text{ on  the set } \{p_{\bullet,\ga}=0\}.
\end{gather*}
So we conclude that
\begin{gather*}
\{ \wp_L, \wp_R\}=0  \text{ on the set  } \{ p_{R,\ga}=p_{L,\ga}=0\},
\end{gather*}
and 
\begin{gather}
[H_{\wp_R}, H_{\wp_L}]=0  \text{ on  } \{\wp_R=\wp_L=0\}. \label{comm-Hamil}
\end{gather}

  Let $x=(x_1,x_2,x_3, x')$ in $\mr^{2n+2}$  be local coordinates valid near  $\ff\cap L \cap R$ such that
\begin{gather}
\ff=\{x_3=0\}, \;  L=\{x_1=0\}  \text{ and } R=\{x_2=0\}. \label{localco}
\end{gather}
and that the symplectic form $\omega^0= d\xi \wedge d x.$   We know that $\La^*$ is a Lagrangian submanifold of $T^*\{x_1>0, \ x_2>0, \ x_3> 0\}$  up to the front face $\ff=\{x_3=0\}$ and  which intersects $\ff$ transversally.  There are vector fields $H_{\wp_R}$ and $H_{\wp_L}$ tangent to $\La^*$ that are $C^\infty$ up to 
 $\{x_3=0\},$   commute on $\La^*$  in view of \eqref{comm-Hamil}, and  \eqref{fields} holds for $H_{\wp_\bullet}.$   Moreover,
 $H_{\wp_R}$ is tangent to $\ff$ and $L,$
while $H_{\wp_L}$ is tangent to $\ff$ and $R.$ 

Let
\begin{gather}
\mcf=T_{\{x_1=x_2=0\}}^*(\mr_s \times \{x: x_1\geq 0, x_2\geq 0, x_3\geq 0\}), \label{defmcf}
\end{gather}
and let  $p=(0,0,x_3,\xi_3, x', \xi')),$ denote a point on  $\mcf \cap \La^*.$
So, for $\eps$ small enough we  define
\begin{gather*}
\Psi_0: [0,\eps)\times [0,\eps) \times \mcf \longrightarrow U_0 \subset T^*( \{ x_1\geq 0, x_2\geq 0, x_3\geq 0\})\\
\Psi_0(t_1,t_2,p) = \exp(-t_1 H_{\wp_R})\circ  \exp(-t_2 H_{\wp_L})(p),
\end{gather*}
and
\begin{gather*}
\Psi_1: [0,\eps)\times [0,\eps) \times \mcf \longrightarrow U_1 \subset T^*( \{ x_1\geq 0, x_2\geq 0, x_3\geq 0\})\\
\Psi_1(t_1,t_2,p) = \exp(-t_1 \p_{x_1})\circ  \exp(-t_2 \p_{x_2})(p).
\end{gather*}
Since the vector fields $H_{\wp_R},$ $H_{\wp_L}$ commute and  $\p_{x_1}$ and $\p_{x_2}$ commute, both maps are well defined and moreover
\begin{gather*}
\Psi_0^* H_{\wp_R}=-\p_{t_1}, \;\  \Psi_0^* H_{\wp_L}=-\p_{t_2},  \\
\Psi_1^* \p_{x_1}=-\p_{t_1}, \;\ \Psi_1^* \p_{x_2}=-\p_{t_2}. 
\end{gather*}
Hence,
\begin{gather}
\begin{gathered}
\Psi=\Psi_0\circ \Psi_1^{-1}: U_1 \longrightarrow U_0, \\
\Psi^* H_{\wp_R}=\p_{x_1},  \;\  \Psi^* H_{\wp_L}=\p_{x_2} \text{ and } \Psi|_{\mcf\setminus 0}= \id.
\end{gathered}\label{diffeo-ext}
\end{gather}
Moreover,  if $\omega^0$  is the symplectic form in the interior of  $T^*(\xo),$ in coordinates \eqref{localco} valid near $\mcf,$
\begin{gather*}
\Psi^*\omega^0= \omega^0.
\end{gather*}

Now $\Upsilon=\Psi^{-1}(\La^*)$ is a $C^\infty$ Lagrangian in  the interior $\{ x_1> 0, \; x_2>0, x_3\geq 0\}$  and
both $\p_{x_1}$ and $\p_{x_2}$ are tangent to $\Upsilon.$  But this implies that  for any point $p \in \Upsilon,$ the integral curves of
 $\p_{x_j},$ $j=1,2,$ starting at a point $p\in \Upsilon$ are contained in $\Upsilon.$  Therefore, for any $p=(x_1,\xi_1, x_2,\xi_2, x_3,\xi_3,x', \xi')\in \Upsilon,$ with 
 $x_1$ and $x_2$ small enough, we have
$$\{x_1-t_1,\xi_1, x_2-t_2,\xi_2, x_3,\xi_3, x',\xi'\} \subset \Upsilon.$$
  By taking $t_1$ and $t_2$ large enough, this gives an extension $\overline{\Upsilon}$ of $\Upsilon$ to $\{x_1=0\}\cup \{x_2=0\}.$ Now $\Psi(\overline{\Upsilon})$ is the desired Lagrangian extension of $\La^*.$   Of course the same method works in the regions away from the codimension three corner.  

This argument can be used to show that $\La^* \cap T_{\{\rho_\bullet=0\}}^*( \xo)$  is a $\CI$ Lagrangian submanifold of $T^*\{\rho_\bullet=0\}.$   To see that, observe that we
have constructed  local symplectic coordinates
$(x,\xi)$ such that $R=\{x_1=0\}$ and $L=\{x_2=0\}$ and  $\Psi^*H_{\wp_R}=H_{\xi_1}=\p_{x_1}$ and $\Psi^*H_{\wp_L}=H_{\xi_2}=\p_{x_2}.$ Therefore $\Psi^*\wp_R= \xi_1+C_1$ and $\Psi^*\wp_L=\xi_2+C_2.$   But since $\wp_R(p)=\wp_L(p)=0=\xi_1(p)=\xi_2(p)=0,$  it follows that $C_1=C_2=0.$   So $\xi_1=\xi_2=0$ on $\Upsilon.$ But  $\Upsilon$ is foliated by submanifolds
\begin{gather*}
\Upsilon_{a}= \Upsilon \cap \{x_1=a\}, \;\  \Upsilon^{a}= \Upsilon \cap \{x_2=a\}, \;\ 
\end{gather*}
which are Lagrangian submanifolds of $T^*\{x_j=a\},$ $j=1,2$ because $\xi_j=0$ on $\Upsilon.$  In particular this shows that $\Upsilon_0= \La^* \cap \{\rho_R=0\}\subset T^*\{\rho_R=0\}$ and $\La^*\cap \{\rho_L=0\}\subset T^*\{\rho_L=0\}$  are Lagrangian submanifolds. The same argument shows that and 
$\La^*\cap \{\rho_R=\rho_L=0\}\subset T^*\{\rho_R=\rho_L=0\}$ is a Lagrangian submanifold. 

Once we know that $\La^*$  has a polyhomogeneous extension to the right and left faces, the precise asymptotic expansions given in \eqref{param-A0} follow directly from \eqref{fields} and  the asymptotic expansion \eqref{phexp} in Lemma \ref{asym-expL}.   Moreover,  since $\La_0$ is a $C^\infty$ compact submanifold,  and the vector fields $H_{\wp_\bullet}$ are non-degenerate at $\{\rho_\bullet=0\},$ $\bullet=R,L,$  the extension of $\La^*$ up to $\p T^*(\xo)$ is a compact submanifold of $T^*(\xo).$
 This concludes the proof of Theorem \ref{soj}.
\epf


 As an application of Theorem \ref{soj},  we study the asymptotics of the distance function $r(z,z')$ between $z,z'\in \intx$ as $z,z'\rightarrow \p X,$ in the case where $(\intx,g)$ is a  geodesically convex CCM. In this case there are no conjugate points  along any geodesic in $\intx$ and  $r(z,z')$ is equal to the length of the unique geodesic joining the two points. Moreover, $r(z,z')$  is smooth on $(\intx \times \intx)\backslash\diag$.  This is the case when $(\intx,g)$ is a Cartan-Hadamard manifold, i.e.
 when $\intx$ has non-positive sectional curvature, see \cite{GHL}.   The following is a consequence of Theorem \ref{soj}
\begin{theorem}\label{main}
Let $(\intx, g)$ be a geodesically convex CCM, and let $\rho_L$ and $\rho_R$ be boundary defining functions of  $L$ and $R$  respectively.  For $z,z' \in \intx$,  the lift of the distance function $r(z, z')$  to $\xo$ satisfies 
\begin{gather}
\begin{gathered}
\beta_0^*r = -\gamma+ \mcr, \;\  \gamma=\frac{1}{\ka_L} \log \rho_L+\frac{1}{\ka_R} \log \rho_R
\end{gathered}\label{eqfr}
\end{gather}
$\mcr$ is $C^\infty$ up to $\p(\xo)$ if $\ka$ is constant, and  in general  $d_m\mcr$ has polyhomogeneous expanion at $\{\rho_R=0\}\cup \{\rho_L=0\}$ in the sense of  \eqref{as-ph-exp}.
\end{theorem}
\begin{proof}    Since $(\intx,g)$ is geodesically convex,  $\Lambda\setminus \La_0$ is the graph of the differential of the distance function. In other words,
\begin{gather}\label{graph}
\Lambda \setminus \La_0=\{(z,\zeta,z',\zeta'):  \zeta'= d_{z'}r(z,z'), \; \zeta= d_{z} r(z,z'), \text{ provided } (z,\zeta)\not= (z',-\zeta')\}.
\end{gather}

It then follows that $\beta_0^*\La= \{(m,d_m \beta_0^*r)\},$  but according to \eqref{shift},   $\beta_0^* \La+ d\gamma= \La^*,$ and so
 \begin{gather*}
 \La^*= \{(m, d_m \mcr) \}, \;\ \mcr= \gamma+\beta_0^*r 
 \end{gather*}
 
 But  in view of  \eqref{param-A0} and  in coordinates $(x_1,x_2, x'),$ near $\{x_1=x_2=0\},$ where $x_1=\rho_L,$ $x_2=\rho_R,$
 \begin{gather}
 \p_{x_r} \mcr \sim \sum_{j_1,j_2=1}^\infty \sum_{k_1=0}^{j_1+1}\sum_{k_2=0}^{j_2+1} x_1^{j_1} (\log x_1)^{k_1} x_2^{j_2} (\log x_2)^{k_2}  \mcr_{r,j_1,k_1,j_2,k_2}(x_3,x'), \; r=1,2. \label{diff-eqR}
 \end{gather}
 But we know that $\La^*$ extends to $\{\rho_R=0\}\cap \{\rho_L=0\},$ and is $C^\infty$ there, so $\mcr(0,0,x')$ is $C^\infty.$
 Integrating the equation for $\p_{x_1} \mcr$  restricted to $\{x_2=0\},$  and  integratiingfor $\p_{x_2} \mcr$ restricted to $\{x_1=0\},$ we find that
 \begin{gather}
 \begin{gathered}
 \mcr(x_1,0,x')-\mcr(0,0,x')\sim \sum_{j_1=2}^\infty \sum_{k_1=0}^{j_1} x_1^{j_1} (\log x_1)^{k_1}   \mce_{1, j_1,k_1}(x'), \\
 \mcr(0,x_2,x')-\mcr(0,0,x')\sim \sum_{j_2=2}^\infty \sum_{k_2=0}^{j_2} x_2^{j_2} (\log x_2)^{k_2}   \mce_{2, j_1,k_1}(x'),
 \end{gathered}\label{restr-asym}
 \end{gather}
 Now integrating the equations in \eqref{diff-eqR} for $\p_{x_1} \mcr$ and $\p_{x_2}\mcr,$ we obtain
 \begin{gather*}
 \mcr(x_1,x_2,x')-\mcr(x_1,0,x')\sim \sum_{j_1=1,j_2=2}^\infty \sum_{k_1=0}^{j_1+1}\sum_{k_2=0}^{j_2} x_1^{j_1} (\log x_1)^{k_1} x_2^{j_2} (\log x_2)^{k_2} \mcv_{j_1,k_1,j_2,k_2}(x'), \\
 \mcr(x_1,x_2,x')-\mcr(0,x_2,x')\sim \sum_{j_1=2,j_2=1}^\infty \sum_{k_1=0}^{j_1}\sum_{k_2=0}^{j_2+1} x_1^{j_1} (\log x_1)^{k_1} x_2^{j_2} (\log x_2)^{k_2} \mcw_{j_1,k_1,j_2,k_2}(x'),
 \end{gather*}
and using \eqref{restr-asym} we obtain
\begin{gather*}
\mcr(x_1,x_2,x')-\mcr(0,0,x') \sim \sum_{j_1=2}^\infty \sum_{k_1=0}^{j_1} x_1^{j_1} (\log x_1)^{k_1}   \mce_{1, j_1,k_1}(x')+\\
\sum_{j_1=1,j_2=2}^\infty \sum_{k_1=0}^{j_1+1}\sum_{k_2=0}^{j_2} x_1^{j_1} (\log x_1)^{k_1} x_2^{j_2} (\log x_2)^{k_2} \mcv_{j_1,k_1,j_2,k_2}(x'),\\
\text{ and } \\
 \mcr(x_1,x_2,x')-\mcr(0,0,x')\sim \sum_{j_2=2}^\infty \sum_{k_2=0}^{j_2} x_2^{j_2} (\log x_2)^{k_2}   \mce_{2, j_1,k_1}(x')+\\
 \sum_{j_1=2,j_2=1}^\infty \sum_{k_1=0}^{j_1}\sum_{k_2=0}^{j_2+1} x_1^{j_1} (\log x_1)^{k_1} x_2^{j_2} (\log x_2)^{k_2} \mcw_{j_1,k_1,j_2,k_2}(x').
 \end{gather*}
 Therefore we conclude that the terms in $j_1=1$ and the terms with $k_1=j_1+1$ in the first equation are equal to zero, and similarly, the terms in 
$j_2=1$ and the ones with $k_2=j_2+1$ in the second equation are also equal to zero.  So we conclude that
\begin{gather}
\mcr(x_1,x_2,x')-\mcr(0,0,x')\sim  \sum_{j_1=2,j_2=2}^\infty \sum_{k_1=0}^{j_1}\sum_{k_2=0}^{j_2} x_1^{j_1} (\log x_1)^{k_1} x_2^{j_2} (\log x_2)^{k_2} \mcw_{j_1,k_1,j_2,k_2}(x'). \label{exp-mcr}
\end{gather}

and this proves the Theorem.
\end{proof}

Equation \eqref{eqfr} was the key ingredient in the construction of  a semiclassical parametrix for the resolvent of the Laplacian  carried out in \cite{MSV} when $\intx=\{z\in \mr^{n+1}: |z|<1\}$ is equipped with a metric 
\begin{gather}
 g_\eps=\frac{4 dz^2}{(1-|z|^2)^2} + \chi(\frac{1-|z|^2}{\eps}) H(z,dz), \label{msv-model}
\end{gather}
where $\chi(t)\in C_0^\infty(\mbr),$ with $\chi(t)=1$ if $|t|<1$ and  $\chi(t)=0$ if $|t|>2,$  $H$ is a $C^\infty$ symmetric 2-tensor.   Melrose, S\'a Barreto and Vasy in \cite{MSV} showed that there exists $\eps_0>0$ such that \eqref{eqfr} holds for $(\intx,g_\eps).$    Metrics of the type \eqref{msv-model} appear   in connection with the analysis of the asymptotic behavior of solutions of the wave equation on de Sitter-Schwarszchild space-time, \cite{MSV1}.


\section{The case of  gedodesically convex CCM}\label{SPGC}
In this section we  will construct the parametrix and prove Theorem \ref{resest} in the case when $(\intx, g)$ is a geodesically convex CCM.  As we have shown in Theorem \ref{main},  in this particular case the underlying  Lagrangian submanifold $\La^*$  is globally parametrized by a phase function $\mcr,$ and we  only need to introduce very simple semiclassical Lagrangian distributions to construct the parametrix.    We will carry out the proof of the general case in Section \ref{General-CCM}, after we discuss the semiclassical  Lagrangian distributions in Section \ref{PHLD}. 

We have done the first two steps in the construction of the parametrix: Combining Lemma \ref{step1-const} and Lemma \ref{step2-const} gives an operator 
$G_1'(h,\sigma)=G_0(h,\sigma)+G_{1}(h,\sigma)$  holomorphic in $\sigma\in \Omega_\hbar=(1-ch,1+ch) \times-i (-Ch, Ch),$ $c>0,$ $C>0,$ such that 
\begin{gather}
 \begin{gathered}
 P(h,\sigma)G_1'(h,\sigma)-\Id=E_1(h,\sigma), \;\\
E_1(h,\sigma)=E'_1(h,\sigma)+ e^{i \soh r} F_1(h,\sigma), \;\ E'_1\in\rho_\mcs^\infty\Psi_{0,\hb}^{-\infty}(X),\ \ F_1
 \in  \rho_\mcs^\infty \Psi_{0,\hb}^{-\infty,\infty,-\novt,\infty} (X)\\
 \text{ and  with } \beta_{\hb}^* K_{F_1}\text{ supported away from} \ \mcl,\ \mcr,
 \end{gathered}\label{redef-f1}
  \end{gather}
and $K_{E'_1}$  is supported near $\diag_\hb$ while  $\beta_\hb^* K_{F_1}$ supported away from $\diag_{\hb}.$  As mentioned before,  the error $E_1'$ is already good for our purposes, but we need to remove the error $e^{-i\soh r} F_1,$ since it does not vanish at the semiclassical face $\mca.$  Here it is very important that $\beta_\hb^*K_{F_1}$ is compactly supported in a neighborhood of $\mcs$ and vanishes to infinite order at $\diag_\hb,$ see Fig.\ref{suppG1}.


 The third step of the construction in this particular case is given by the following lemma.
\begin{lemma}\label{aface}  Let $(\intx,g)$ be a geodesically convex CCM, and let $r:\intx \times \intx  \longmapsto [0,\infty)$ denote the distance function between two points of 
$\intx,$  let $F_1(h,\sigma)$ be as in \eqref{redef-f1} and let $h_0>0$ satisfy \eqref{defh0}. Then there are  operators $G_2(h,\sigma)$ and  $E_2(h,\sigma) $ with  Schwartz kernels $K_{G_2}$ and $K_{E_2}$ such that
$\beta_0^*K_{G_2}= e^{i\soh \beta_0^*r} e^{ih\sigma \beta} U_2(h, \sigma) \beta_0^*|dg(z')|^\ha$ and 
$\beta_0^*K_{E_2}=  e^{i\soh \beta_0^*r} e^{ih\sigma \beta} F_2(h, \sigma) \beta_0^*|dg(z')|^\ha,$  with  $h\in (0,h_0),$ where 
$\beta$ is defined in \eqref{def-gamma},  and that are  holomorphic in $\sigma \in \Omega_\hbar,$  with
\begin{gather*}
 U_2(h, \sigma)\in h^{-\novt-1} \mck_{ph}^{\fnt,\fnt}(\xo\times [0,h_0)),  \text{ vanishing to infinite order at } \diag_0\times [0, h_0) \\
F_2(h, \sigma) \in  h^\infty \mck_{ph}^{\fnt, \fnt}(\xo\times [0,h_0)),
\end{gather*}
such that 
\beqq\label{mcaasym}
(h^2(\Delta_{g(z)}-\knsq)-\sigma^2)G_2(h,\sigma) - e^{i\soh r} F_1(h,\sigma) = E_2(h,\sigma).
\eeqq
\end{lemma}
\begin{proof}  Since $\beta_\hb^*K_{F_1}$ vanishes to infinite order at $\mcs,$  it follows that
$\beta_0^*K_{F_1}\in h^{-\novt} C^\infty( \xo\times [0,1))$  is supported near  $\diag_0\times [0,1)$ and vanishes to infinite order at  
$\diag_0\times [0,1),$ so  it follows that $\beta_0^*K_{F_1}$ has an asymptotic expansion at $\{h =0\}$ of the form
\begin{gather*}
\beta_0^*K_{F_1}\sim h^{-\novt} \sum_{j=0}^\infty  h^j  F_{1,j}(\sigma',m), \; \sigma=1+h \sigma', \; m\in \xo
\end{gather*}
with  $F_{1,j} \in C^\infty( \xo )$  supported near $\diag_0$ but vanishing to infinite order at $\diag_0.$   We will  find $U_2(h,\sigma)$ such that 
\begin{gather}
\begin{gathered}
U_2(h,\sigma) \sim h^{-\novt-1} \sum_{j=0}^\infty h^j U_{2,j}(\sigma',m), \;  \sigma=1+h \sigma',  \\ U_{2,j} \in \mck_{ph}^{\novt,\novt}(\xo) \text{ vanishes to infinite order at } \diag_0, \text{ and } \\
\beta_0^*( P(h,\sigma,D) )e^{i\soh \beta_0^*r} e^{ih\sigma \beta}U_2(h,\sigma)- e^{i\soh \beta_0^*r} e^{ih\sigma \beta} \beta_0^* F_1(h,\sigma)=e^{i\soh \beta_0^*r} e^{ih \sigma \beta} F_2(\sigma,h), \\
\text{ with } F_2 (h,\sigma) \in h^\infty \mck_{ph}^{\novt,\novt}(\xo\times [0,h_0)).
\end{gathered} \label{w-exp}
\end{gather}
  We may choose $\rho_R$ and $\rho_L$ such that  $\rho_R=\rho_L=1$ on the support of $F_1.$   In this case $\gamma=\tgamma=0$ on the support of $F_1,$ and since $h^2\beta=\gamma-\tgamma,$  $e^{i\soh \beta_0^*r} F_1= e^{i\soh \beta_0^*r} e^{ih\sigma \beta} F_1.$  Perhaps one would have tried  an expansion of the form
  $\beta_0^*K_{G_2} =e^{-i\soh \beta_0^*r} U_2(h,\sigma),$ as in \cite{MSV}, however   if this were the case,  then in view of \eqref{defn-conj}  $U_2(h,\sigma)$  would need to have a polyhomogeneous expansion at the left and right faces involving variable powers of $\rho_\bullet,$  and thus it would not  be in the class $\mck_{ph}^{\novt,\novt}(\xo).$

Instead of having to deal with the factor $\rho_\bullet^{\novt}$ in the expansions, it is convenient to work with
\begin{gather}
Q(h,\sigma,D)= x^{-\novt}( h^2(\Delta_g-\knsq)- \sigma^2)x^{\novt}. \label{def-opq}
\end{gather}
Notice that  $Q(h,\sigma,D)-P(h,\sigma,D)= O(h),$ so they have the same semiclassical principal symbol.
Since $\gamma-\tgamma=h^2\beta,$ and $\beta_0^*r=\mcr-\gamma,$ it follows that $\soh \beta_0^*r+ h\sigma\beta=\soh(r+ h^2\beta)=\soh(\mcr-\tgamma).$ So we denote
\begin{gather} 
\begin{gathered} 
Q_L(h,\sigma, D)=\beta_0^*Q(h,\sigma,D), \text{ and } \\
Q_{L,\tgamma}(h,\sigma,D)=e^{i\soh \tgamma } Q_L(h,\sigma,D) e^{-i\soh \tgamma}.
\end{gathered} \label{qtgamma}
\end{gather}
If $U_{2,m}=x^{-\novt} V_{m}$ and  $F_{1,j}= x^{-\novt} G_j,$ then we  are reduced to  finding $V_j,$ such that
\begin{gather}
\begin{gathered}
Q_{L,\mcr, \tgamma}(h,\sigma,D) h^{-\novt-1} \sum_{j=0}^\infty h^j  V_j\sim h^{-\novt}  \sum_{j=0}^\infty h^j G_j, \\
\text{ where } Q_{L,\mcr,\tgamma}(h,\sigma,D)= e^{-i\soh \mcr} Q_{L,\tgamma} e^{i\soh \mcr}= e^{-i\soh (\mcr-\tgamma)} Q_{L} e^{i\soh (\mcr-\tgamma)}
\end{gathered} \label{new-asym}
\end{gather}

In view of \eqref{formlap},  then in local coordinates in $\xo$ it must be of the form
\begin{gather}
 Q_L(h,\sigma,D)= \sum_{jk} a_{jk}(m) hD_jhD_k + i \sum_{j} h b_j(m) hD_j+ \frac{h^2n^2}{4}(\ka_L^2-\ka_0^2)-\sigma^2, \label{opPL}
\end{gather}
and so
\begin{gather*}
Q_{L,\mcr,\tgamma}(h,\sigma,D)= \sum_{jk} a_{jk}(m)(hD_j +\sigma \p_j (\mcr-\tgamma)) (hD_k + \sigma \p_k (\mcr-\tgamma))+ \\
i \sum_{j} h b_j(m,h)(hD_j + \sigma \p_j (\mcr-\tgamma)) + C
\end{gather*}
and this can be written as
\begin{gather}
\begin{gathered}
Q_{L,\mcr,\tgamma}(h,\sigma,D)= \sum_{jk} a_{jk} hD_j hD_k +i \sum_j \tilde B_j hD_j + C, \text{ where } \\
\tilde B_j= -2i \sigma \sum_{k} a_{jk} \p_k (\mcr-\tgamma) +  hb_j, \\
C= \sigma^2\left(\sum_{jk} a_{jk}\p_j(\mcr-\tgamma)\p_k(\mcr-\tgamma)\right)-\sigma^2-i h \sigma \sum_{jk} a_{jk} \p_k \p_j (\mcr-\tgamma)+ \\ i h\sigma  \sum_j b_j \p_j (\mcr-\tgamma)+ \frac{h^2n^2}{4}(\ka_L^2-\ka_0^2).
\end{gathered}\label{form-ql}
\end{gather}

So we need to analyze the behavior of the  term $C$ in \eqref{form-ql} near $\{\rho_L=0\},$ more specifically, we want to analyze
\begin{gather}
C_0(h,\sigma,m)= \sigma^2\left(\sum_{jk} a_{jk}\p_j(\mcr-\tgamma)\p_k(\mcr-\tgamma)\right)-\sigma^2+\frac{h^2n^2}{4}(\ka_L^2-\ka_0^2). \label{def-c0}
\end{gather}
  Of course, by symmetry we also get  the behavior on the right face, we would just need to work with  the operator $Q_{R}(h,\sigma,D)$  on the right factor instead. 

We have that in coordinates \eqref{prod},
\begin{gather}
\begin{gathered}
Q(h,\sigma,D)= h^2\ka^2(y) (xD_x)^2+h^2x^2 \Delta_H+  i h^2x^2 \sum_{j} B_j D_{y_j} +\\  i h^2 x^2 FD_x + \frac{h^2n^2}{4}(\ka(y)^2-\ka_0^2)-\sigma^2. 
\end{gathered}\label{opq-loc}
\end{gather}
 We need to lift  the principal part of this operator to $\xo$ and compute $C_0(h,\sigma,m).$  We work in the regions near the co-dimension 3 corner, and we use coordinates  \eqref{eqc2}.  The computations for the other regions are  simpler.  According to \eqref{comp-cor}, we have
 \begin{gather}
 \begin{gathered}
 Q_L(h,\sigma,D)=\beta_0^* Q(h,\sigma,D)=- h^2\ka_L^2(x_1\p_{x_1})^2- h^2 x_1^2 \sum_{j,k} H^{jk} \mcu_j\mcu_k+ \\
 h^2 x_3 x_1^2F_1 \p_{x_1} + x_1^2x_2F_2 \p_{x_2} + x_1^2 x_3 F_3 \p_{x_3} + x_1^2 \sum_{j} B_j \p_{Y_j},
 \end{gathered} \label{opq-loc0}
 \end{gather}
 and therefore,  in view of \eqref{def-c0}, 
 \begin{gather}
  \begin{gathered}
 C_0(h,\sigma,m)= \sigma^2 \ka_L^2 {x_1}^2(\p_{x_1} (\mcr-\tgamma))^2 - \sigma^2+\frac{h^2n^2}{4}(\ka_L^2-\ka_0^2)+  \\ \sigma^2 {x_1}^2 H^{11}(\mcu_1 (\mcr-\tgamma))^2+  
2\sigma^2 {x_1}^2 \sum_{j=2}^n H^{1j}( \mcu_1(\mcr-\tgamma))(\p_{Y_j}(\mcr-\tgamma)) + \\  \sigma^2 {x_1}^2 \sum_{j,k=2}^n H^{jk} ( \p_{Y_j} (\mcr-\tgamma))( \p_{Y_k} (\mcr-\tgamma)), \\
 \text{ where }  \mcu_1(\mcr-\tgamma)= (x_3\p_{x_3}-{x_1}\p_{x_1}-x_2\p_{x_2}-\sum_{j=1}^n Y_j \p_{Y_j}) (\mcr-\tgamma).
 \end{gathered} \label{term1-eik}
  \end{gather}
 In these coordinates, $\rho_L={x_1}$ and $\rho_R=x_2$ and so $\tgamma=\mu_R\log x_2+ \mu_L \log {x_1},$ where $\mu_R=\mu(y')$ and 
 $\mu_L=\mu(y_1'+u, y'+uY),$ where $Z=(Y_2,\ldots Y_n).$  Since $\sigma^2\ka_L^2 \mu_L^2=\sigma^2-\frac{h^2n^2}{4}(\ka_L^2-\ka_0^2),$  the first two terms in \eqref{term1-eik} give
 \begin{gather*}
\sigma^2 \ka_L^2 {x_1}^2(\p_{x_1} (\mcr-\tgamma))^2-\sigma^2+\frac{h^2n^2}{4}(\ka_L^2-\ka_0^2)=
\sigma^2\ka_L^2 {x_1}^2( \p_{x_1} \mcr- {x_1}^{-1}\mu_L)^2-\sigma^2+\frac{h^2n^2}{4}(\ka_L^2-\ka_0^2)
=  \\ \sigma^2\ka_L^2( {x_1}^2(\p_{x_1} \mcr)^2- 2{x_1} \p_{x_1}\mcr \mu_L+\mu_L^2)-\sigma^2+\frac{h^2n^2}{4}(\ka_L^2-\ka_0^2)=
\sigma^2\ka_L^2( {x_1}^2(\p_{x_1} \mcr)^2- 2\mu_L {x_1} \p_{x_1}\mcr),
\end{gather*}
and hence if $x=(x_1,x_2,x_3),$
\begin{gather*}
C_0(h,\sigma,x,u,Y)= \sigma^2\ka_L^2( {x_1}^2(\p_{x_1} \mcr)^2- 2\mu_L {x_1} \p_{x_1}\mcr)+ 2\sigma^2 {x_1}^2\sum_{j=2}^n H^{1j}( \mcu_1(\mcr-\tgamma))(\p_{Y_j}(\mcr-\tgamma) +\\ \sigma^2 {x_1}^2\sum_{j,k=2}^n H^{jk} (\p_{Y_j} (\mcr-\tgamma))(\p_{Y_k} (\mcr-\tgamma)).
\end{gather*}
Recall that  for $\gamma=\frac{1}{\ka_R}\log x_2+\frac{1}{\ka_L}\log {x_1},$ we can use \eqref{def-gamma} and \eqref{real-anal} to write

\begin{gather*}
\begin{gathered}
C_0(h,\sigma,x,u,Y)=  \sigma^2\ka_L^2( {x_1}^2(\p_{x_1} \mcr)^2- 2\frac{1}{\ka_L} {x_1}\p_{x_1} \mcr -2h^2 \nu_L {x_1} \p_{x_1}\mcr)+ \\ 2\sigma^2 {x_1}^2\sum_{j=2}^n H^{1j}( \mcu_1(\mcr-\gamma+h^2 \beta))(\p_{Y_j}(\mcr-\gamma+h^2\beta)+ \\ \sigma^2 {x_1}^2 \sum_{j,k=2}^n H^{jk} (\p_{Y_j} (\mcr-\gamma+h^2 \beta))(\p_{Y_k} (\mcr-\gamma+h^2\beta)),
\end{gathered}
\end{gather*}
and hence
\begin{gather}
\begin{gathered}
C_0(h,\sigma,x,u,Y)= \sigma^2\ka_L^2( {x_1}^2(\p_{x_1} \mcr)^2- 2\frac{1}{\ka_L} {x_1}\p_{x_1} \mcr) +  2 \sigma^2 {x_1}^2\sum_{j=2}^n H^{1j}( \mcu_1(\mcr-\gamma))(\p_{Y_j}(\mcr-\gamma)+ \\ \sigma^2 {x_1}^2 \sum_{j,k=2}^n H^{jk} (\p_{Y_j} (\mcr-\gamma))(\p_{Y_k} (\mcr-\gamma))- 
2h^2 \sigma^2\ka_L^2 \nu_L {x_1} \p_{x_1}\mcr +  \\  2h^2  \sigma^2 {x_1}^2\sum_{j=1}^n H^{1j}( \mcu_1(\mcr-\gamma))(\p_{Y_j}\beta)+ 
2 h^2 \sigma^2  {x_1}^2\sum_{j=1}^n H^{1j}( \p_{Y_j}(\mcr-\gamma))(\mcu_1\beta)+  \\
2h^2 \sigma^2 {x_1}^2 \sum_{j,k=1}^n H^{jk} (\p_{Y_j} (\mcr-\gamma))(\p_{Y_k} \beta) +2 h^4 \sigma^2\omega^2 \sum_{j=1}^n \mcu_1 \beta \p_{Y_j}\beta+ 
2 h^4 \sigma^2\omega^2 \sum_{j,k=2}^n \p_{Y_j}\beta\p_{Y_k}\beta.
\end{gathered}\label{exp-beta}
\end{gather}

 Recall that  $p=|\zeta|_{g^*}^2-1=0$ on $\La$ and since away from $\diag,$ $\La=\{(z,d_z r, z', d_{z'}r)\},$  and $\beta_0^*r=\mcr-\gamma,$ 
  it follows that $\sum_{jk} a_{jk} \p_j (\mcr-\gamma)\p_k (\mcr-\gamma)=1,$ and so in coordinates  \eqref{eqc2}, this equation gives
  \begin{gather*}
  \begin{gathered}
 \ka_L^2 {x_1}^2(\p_{x_1}(\mcr-\gamma))^2-1+ {x_1}^2 H^{11}(\mcu_1 (\mcr-\gamma))^2+ \\
2 {x_1}^2\sum_{j=1}^n H^{1j}( \mcu_1(\mcr-\gamma))(\p_{Y_j}(\mcr-\gamma)) +x_1^2 \sum_{j,k=1}^n H^{jk} (\p_{Y_j} (\mcr-\gamma))(\p_{Y_k} (\mcr-\gamma))=0,
 \end{gathered}\label{exp-beta1}
  \end{gather*}
and since $\ka_L^2 {x_1}^2(\p_{x_1} \gamma)^2=1,$ it follows that 
\begin{gather}
\begin{gathered}
\ka_L^2({x_1}^2(\p_{x_1}\mcr)^2- 2{x_1}\frac{1}{\ka_L} \p_{x_1}  \mcr) + {x_1}^2 H^{1,1}(\mcu_1 (\mcr-\gamma))^2+ 2{x_1}^2\sum_{j=1}^n H^{1j}( \mcu_1(\mcr-\gamma))(\p_{Y_j}(\mcr-\gamma) + \\ {x_1}^2\sum_{j,k=1}^n H^{jk} (\p_{Y_j} (\mcr-\gamma))(\p_{Y_k} (\mcr-\gamma))=0.
\end{gathered}\label{beta2}
\end{gather}
 Therefore, substituting \eqref{beta2} in  \eqref{exp-beta}, and using that $\sigma=1+h\sigma',$  we deduce that 
\begin{gather*}
C_0(h,\sigma,x,u,Y)= h^2 \widetilde{C}_0(h,\sigma,x,u,Y), \text{ where } \\
\widetilde{C}_0(h,\sigma,x,u,Y)= - 2 \sigma^2\ka_L^2 \nu_L {x_1} \p_{x_1}\mcr +  2  \sigma^2 {x_1}^2\sum_{j=1}^n H^{1j}( \mcu_1(\mcr-\gamma))(\p_{Y_j}\beta)+ \\
2  \sigma^2  {x_1}^2\sum_{j=1}^n H^{1j}( \p_{Y_j}(\mcr-\gamma))(\mcu_1\beta)+  
2 \sigma^2 {x_1}^2 \sum_{j,k=1}^n H^{jk} (\p_{Y_j} (\mcr-\gamma))(\p_{Y_k} \beta) + \\ 2 h^2 \sigma^2x_1^2 \sum_{j=1}^n \mcu_1 \beta \p_{Y_j}\beta+ 
2 h^2 \sigma^2x_1^2 \sum_{j,k=2}^n \p_{Y_j}\beta\p_{Y_k}\beta.
\end{gather*}

Since $\mcr\in \mck_{ph}^{0,0}(\xo),$ and ${x_1}\p_{x_1} \beta, {x_1}\p_u \beta, {x_1}\p_{Y_j}\beta \in \mck_{ph}^{0,0}(\xo)$ and ${x_1}\p_{x_1}\tgamma, {x_1}\p_{u} \tgamma, {x_1}\p_{Y_j} \tgamma\in \mck_{ph}^{0,0}(\xo),$ and $\sigma=1+h\sigma',$  if follows that  from the definition of $\beta$ in \eqref{def-gamma} and \eqref{real-anal} that  
\begin{gather}
\begin{gathered}
 \wtc_0(h,\sigma,x,u,,Y)\sim \sum_{j=0}^\infty h^j C_{0,j} (\sigma',x,u,Y), \text{ such that  }  \\ C_{0,j}(\sigma',x,u,Y) \sim \sum_{l=1}^\infty \sum_{k=0}^l x_1^l(\log x_1)^k C_{0,j,k,l}(x_2,x_3,u,Y), 
 \end{gathered}  \label{exp-C}
  \end{gather}

Now we return to \eqref{form-ql}. We use \eqref{def-gamma}, \eqref{exp-C} and the fact  that $\sigma=1+h \sigma',$ $\sigma'\in (-c,c)\times i (-C,C),$ to conclude that
\begin{gather}
\begin{gathered}
Q_{L,\mcr,\tgamma}(h,\sigma,D) \sim  h^2   Q_{0,L} (h,\sigma',D) +  h ( W+  \vtheta), \text{ where } \\
Q_{0,L}(h,\sigma',D)=\sum_{jk} a_{jk} D_j D_k+ i\sum_j b_j D_j + \sigma'(W+\vtheta) +\\  h\sigma  \sum_{jk} a_{jk}( 2 \p_k\beta D_j- i\p_j\p_k\beta)
+ ih\sigma \sum_j b_j \p_j \beta + \wtc_0(h,\sigma), \\
W=  -2i\sum_j (\sum_k a_{jk} \p_k (\mcr-\gamma)) \p_j, \\ 
\vtheta= -i \sum_{jk}a_{jk} \p_j\p_k (\mcr-\gamma) + i\sum_j b_j \p_j (\mcr-\gamma).
\end{gathered}\label{transp-0}
\end{gather}

Using the formula of the coefficients of $Q_L(h,\sigma,D)$ given by \eqref{opq-loc0}, the polyhomogeneity of $\mcr$ given by \eqref{exp-mcr} and the definition of $\ga,$ we find that
\begin{gather}
\vtheta= x_1 \tilde\vtheta, \;\ \tilde\vtheta\sim \sum_{j=0}^\infty \sum_{k=0}^j x_1^j(\log x_1)^k \tilde\vtheta_{j,k}(x_2,x_3,Y). \label{def-tildevtheta}
\end{gather}

We deduce from \eqref{def-gamma}, \eqref{real-anal},  \eqref{opq-loc} and \eqref{exp-C} that
\begin{gather}
\begin{gathered}
Q_{0,L}(h,\sigma',D)= Q_L + \sigma'(W+\vtheta) + \mcw(h,\sigma'), \\
\text{ where }
Q_L=\sum_{jk} a_{jk} D_jD_k + \sum_{j} i b_j D_j, \\
\mcw(h,\sigma)\sim  \sum_{a=0}^\infty h^a \mcw_a (\sigma',m), \\
\mcw_a(\sigma', m') = \sum_{j} A_{a,j}(\sigma',m)\rho_L D_j + B_a(\sigma', m), \;\ A_{a,j}, B_a \in \mck_{ph}^{0,0} (\xo)
\end{gathered}\label{transp-00}
\end{gather}
Notice that the coefficient of $D_j$ in $\mcw_a(\sigma',m)$ comes from $a_{jk}\p_k\beta,$ which in turn come from the terms of the principal part of the lift of the operator $Q$ to $\xo,$ which is given in terms  of the lifts of the  vector fields  as \eqref{comp-cor}, and so the $a_{jk}$ vanish to order two at $\{\rho_L=0\},$ so the coefficient of $D_j$ is as above.

By replacing this expression into the expansion \eqref{new-asym}, we obtain the following set of transport equations valid in local coordinates
\begin{gather}
\begin{gathered}
(W+\vtheta) V_0= G_0, \; V_0=0 \text{ at } \diag_0, \\
(W+\vtheta) V_r + (Q_L+\sigma'(W+\vtheta)) V_{r-1} + \sum_{a+k=r-1} \mcw_a V_k= G_r, \\ V_r=0 \text{ at } \diag_0, \;\ r\geq 1.
\end{gathered}\label{transp-01}
\end{gather}

Iterating these equations we find that
\begin{gather}
\begin{gathered}
(W+\vtheta) V_r= G_r-\sum_{j=0}^{r-1} (-\sigma')^j [G_{r-j}-(Q_L V_{r-j-1}+ \sum_{a+k=r-1} \mcw_a V_k]), \;\ r\geq 1. 
\end{gathered}\label{iter-transp}
\end{gather}

We need to make invariant sense of $W+\vtheta$  in order to show that these equations  make sense independently of the choice of coordinates.   Let $Q(h,\sigma,D)$ be as in \eqref{def-opq} and let
\begin{gather*}
Q(h,1,D)= x^{-\novt}( h^2(\Delta_g-\knsq)- 1)x^{\novt}.
\end{gather*}

 According to \eqref{def-opq} and \eqref{opPL}
\begin{gather*}
Q_{L,\gamma}(h, D) = e^{\ioh \gamma} \beta_0^*Q(h,1,D)  e^{-\ioh \gamma}= \\ \sum_{jk} a_{jk}(m) (hD_j-  \p_j \gamma)(hD_k-  \p_k \gamma) + \sum_{j} h  b_j(m) (hD_j- \p_j \gamma)=\\
\sum_{jk} a_{jk}(m) hD_j hD_k + i \sum_{j} B_j hD_j +C, \text{ with  } \\
B_j= 2i \sum_{k} a_{jk} \p_k \ga  +  h b_j, \text{ and } \\
 C= -1+\sum_{jk} a_{jk} \p_j\gamma\p_k \gamma+ i h \sum_{jk} a_{jk} \p_k \p_j \ga- i  \sum_j B_j \p_j \ga +\frac{h^2 n^2}{4}(\ka_L^2-\ka_0^2).
\end{gather*}
 The full semiclassical symbol of $Q_{L,\gamma}$ is defined to be
\begin{gather}
\begin{gathered}
q_{L,\gamma}=\sigma^{sc}(Q_{L,\gamma})= \sum_{jk} a_{jk}( \nu_j \nu_k -2 \p_k \gamma \nu_j + \p_j\gamma\p_k\gamma) -1+ \\
ih( \sum_{jk} a_{jk} \p_j\p_k \gamma+ \sum_j b_j( \nu_j-\p_j\ga))+ \frac{h^2 n^2}{4}(\ka_L^2-\ka_0^2).
\end{gathered}\label{sc-psym}
\end{gather}
The semiclassical principal symbol of $Q_{L,\gamma}$  is equal to
\begin{gather*}
\sigma_0^{sc} (Q_{L,\gamma})= q_{L,\gamma}^0= \sum_{jk} a_{jk}( \nu_j \nu_k -2 \p_k \gamma \nu_j + \p_j\gamma\p_k\gamma) -1
\end{gather*}
and the semiclassical subprincipal symbol of $Q_{L,\gamma}$ is defined to be
\begin{gather*}
q^{s}= q_{L,\gamma} - q_{L,\gamma}^0-\frac{1}{2i} \sum_{j} \p_{m_j} \p_{\nu_j} q_{L,\gamma}^0.
\end{gather*}
These quantities are invariantly defined, see for example \cite{Martinez}. Hence  the semiclassical subprincipal symbol of  $Q_{L,\ga}$ is given by
\begin{gather*}
q_{L,\gamma}^{s}= i\sum_{jk} a_{jk} \p_j\p_k \gamma+i \sum_j b_j( \nu_j-\p_j\ga)+\frac{i}{2}\sum_{jk}( 2\p_j(a_{jk} )\nu_k -2\p_j(a_{jk}\p_k\ga)= \\
i \sum_j b_j (\nu_j-\p_j \ga)+i\sum_{jk} \p_j(a_{jk})(\nu_k-\p_k \ga),
\end{gather*}

We know that $\La^*\setminus \wtla_0=\{(m,d_m \mcr)\}$ so  $\mu_k=\p_k \mcr$ on $\La^*,$ and hence
\begin{gather}
q_{L,\gamma}^{s}|_{\La^*}=i \sum_j b_j \p_j(\mcr- \ga)+i  \sum_{jk} \p_j(a_{jk})\p_k(\mcr- \ga). \label{subprinc}
\end{gather}

We may think of a function $\psi(m)$ defined on $\xo$ as a function  on 
$\La^*.$ On the other hand, we know  that $q_{L,\gamma}=0$ on $\La^*,$ and so $H_{q_{L,\gamma}}$ is tangent to $\La^*.$ Moreover, we know from \eqref{sc-psym} that the Hamilton vector field of  
$q_{L,\gamma}$  acting on functions  of the base variable  satisfies 
\begin{gather}
H_{q_{L,\ga}} \psi(m)= \sum_{j} \p_{\nu_j} p_{L,\ga} \p_{j}\psi(m)=2\sum_{j,k} a_{jk}(\nu_k-\p_k\ga)\p_j \psi(m). \label{par-vf}
\end{gather}
but again  we may substitute $\nu_k=\p_k\mcr$ in \eqref{par-vf} and so 
\begin{gather}
H_{q_{L,\ga}} \psi(m)= \sum_{j} \p_{\nu_j} p_{L,\ga} \p_{j}\psi(m)=2\sum_{j,k} a_{jk}( \p_k\mcr-\p_k\ga)\p_j \psi(m)=-\frac{1}{i} W \psi(m). \label{HPW}
\end{gather}
However, \eqref{subprinc} is not equal to $\vtheta.$   Now it is important to use that that we are working with half-densities.  Recall, see for example Section 25.2 of \cite{Hormander}, that if  $\Sigma$ is a $C^\infty$ vector field and $\mcl_{\Sigma}$ denotes the Lie derivative along $\Sigma,$ then
\begin{gather}
\mcl_\Sigma (f |dm|^\ha)= (\Sigma f+ \ha \div(\Sigma) f ) |dm|^\ha. \label{Lied}
\end{gather}
In particular,
\begin{gather}
\mcl_{H_{q_{L,\ga}}}(f |dm|^\ha)=( H_{q_{L,\ga}}f + \left(\sum_{j,k} \p_j(a_{jk} \p_k(\mcr-\ga))\right) f )|dm|^\ha.
\end{gather}
Therefore, it follows from \eqref{subprinc}, \eqref{HPW} and \eqref{Lied} that
\begin{gather*}
(W+\vtheta) (f |dm|^\ha)= (-i \mcl_{H_{q_{L,\ga}}}+ i \sum_{j,k} \p_j(a_{jk} \p_k(\mcr-\ga)) +\vtheta) (f |dm|^\ha)= (\frac{1}{i} \mcl_{H_{q_{L,\ga}}}+ q_{L,\ga}^{s}) (f |dm|^\ha).
\end{gather*}

Therefore the first equation in \eqref{transp-0} becomes
\begin{gather}
 (\frac{1}{i} \mcl_{H_{q_{L,\ga}}}+ q_{L,\ga}^{s}) (V_0(\sigma',m) \omega^\ha)= G_0(\sigma',m) \omega^\ha, \;\ V_0=0 \text{ near} \diag_0, \label{eq-zero-symb}
 \end{gather}
where $\omega= \beta_0^* |dg(z')|.$  This equation can be solved globally and we need to understand the behavior of its solutions at the left face.  Of course, to do that we need to go back and work in local coordinates valid near the left face. Moreover, recall that  
$Q_{L,\gamma}$ and $P_{L,\gamma}$ defined in \eqref{scps-1}, have the same semiclassical principal part, so  $H_{q_{L,\gamma}}= H_{p_{L,\gamma}}.$  But
in  the proof of Theorem \ref{soj}, we defined $p_{L,\gamma}=2\rho_L \wp_{L}$ (there we worked with $\ha(h^2\Delta-\sigma^2)$) and  so
$H_{q_L}= 2\rho_L H_{\wp_L}+ 2\wp_LH_{\rho_L}.$ Since $\wp_L=0$ on $\La^*$ and $G_1$ is supported near $\diag_0,$  and  in view of 
\eqref{def-tildevtheta} we conclude that near the left face we have
\begin{gather}
( H_{\wp_{L}}+ \ha\tilde\vtheta ) V_0=0, \;\  V_0\in C^\infty(\{\rho_L>\del\}), \label{first-T}
\end{gather}
where $\tilde\vtheta$ satisfies the expansion in \eqref{def-tildevtheta}.    If $\zeta(s)$ denotes an integral curve of $H_{\wp_L},$ then the solution to \eqref{first-T}  with initial data at $x_1=\del$ satisfies
\begin{gather*}
\frac{d}{dx_1} \left[ \exp(-\int_{x_1}^\del \tilde\vtheta(\zeta(s)) \; ds) V_0(\zeta(t))\right]=0,
\end{gather*}
and so given a point $(x_1,Z)=(x_1, x_2,x_3,Y),$ let $\zeta(s)$ be the integral curve of $H_{\wp_L}$ joining $(z_1,Z)$ to a point on the surface 
$\{x_1=\del\},$  then
\begin{gather*}
V_0(x_1,Z) =V_0(\zeta(\del)) \exp(\int_{x_1}^\del \tilde\vtheta(\zeta(s)) \; ds),
\end{gather*}

Now we have to use the fact that $\tilde\vtheta$ is a function of the base variables only, and in view of \eqref{param-A0},
\begin{gather*}
\zeta_1(s)=s,\\
\zeta_r(s)\sim \zeta_{r,0}+ \sum_{j=1}^\infty\sum_{k=0}^j s^j(\log s)^k X_{r,j,k}(\zeta_0), \\
\end{gather*}
On the other hand, we deduce from \eqref{def-tildevtheta} that
\begin{gather*}
\tilde\vtheta(\zeta(s))\sim \sum_{j=0}^\infty\sum_{k=0}^j s^j (\log s)^k \tilde\vtheta_{j,k}( \zeta_2(s), \ldots, \zeta_{2n+2}(s)), \;  \tilde\vtheta_{j,k}\in C^\infty.
\end{gather*}
But in view of Taylor's formula
\begin{gather*}
 \tilde\vtheta_{j,k}( \zeta_2(s), \ldots, \zeta_{2n+2}(s))- \tilde\vtheta_{j,k} (\zeta_{2,0}, \ldots, \zeta_{2n+2,0})\sim
 \sum_{\alpha} C_\alpha ( \zeta(s)-\zeta_0)^\alpha \sim \sum_{j=0}^\infty\sum_{k=0}^j s^j (\log s)^k \mcw_{j,k}(\zeta_0),
\end{gather*}
and therefore
\begin{gather*}
\tilde\vtheta(\zeta(s)) \sim \sum_{j=0}^\infty\sum_{k=0}^j s^j(\log s)^k \mca_{j,k}(\zeta_0).
\end{gather*}
After we integrate this expression, use Taylor's expansion of the exponential function, we find that 
\begin{gather}
V_0(x_1,Z)\sim \sum_{j=0}^\infty \sum_{k=0}^j x_1^j(\log x_1)^k \mcv_{0,j,k}(Z). \label{res-v0}.\
\end{gather}

Now suppose that
\begin{gather}
V_l(x_1,Z)\sim \sum_{j=0}^\infty \sum_{k=0}^j x_1^j(\log x_1)^k \mcv_{l,j,k}(Z), \, l\leq r-1 \label{res-vl}.
\end{gather}
we want to show that it holds for $l=r.$    The transport equation for $V_r$ is given by  \eqref{iter-transp}, and if denote
\begin{gather*}
\mcg_r(x_1,Z)=G_r-\sum_{j=0}^{r-1} (-\sigma')^j [G_{r-j}-(Q_L V_{r-j-1}+ \sum_{a+k=r-1} \mcw_a V_k)], 
\end{gather*}
then it follows from \eqref{opq-loc0}, \eqref{transp-0}  and \eqref{res-vl} that
\begin{gather*}
\mcg_r(x_1,Z)  \sim \sum_{j=1}^\infty \sum_{k=0}^j x_1^j(\log x_1)^k \mcg_{r,j,k}(Z),
\end{gather*}
and therefore
\begin{gather}
\frac{1}{x_1} \mcg_r(x_1,Z)\sim \sum_{j=0}^\infty \sum_{k=0}^{j+1} x_1^j(\log x_1)^k \mcg_{r,j,k}(Z). \label{asym-Gr}
\end{gather}
If $\zeta(s)$ is as above, then the solution of \eqref{transp-01} is given by
\begin{gather*}
V_0(x_1,Z) = \exp(\int_{x_1}^\del \tilde\vtheta(\zeta(s)) \; ds)\int_{x_1}^\del \left[ \frac{1}{s} \mcg_r(\zeta(s)) \exp(-\int_{s}^\del \tilde\vtheta(\zeta(s_1)) \; ds_1)\right]\, ds.
\end{gather*}
Using the argument above,  we deduce from \eqref{asym-Gr} that
\begin{gather*}
 \frac{1}{s} \mcg_r(\zeta(s)) \exp(-\int_{s}^\del \tilde\vtheta(\zeta(s_1)) \; ds_1)\sim \sum_{j=0}^\infty \sum_{k=0}^{j+1} s^j(\log s)^k \mcz_{r,j,k}(\zeta_0).
 \end{gather*}
 Integrating this, using \eqref{int-log} and combining with the expansion of  $\exp(\int_{x_1}^\del \tilde\vtheta(\zeta(s)) \; ds)$ that we already discussed, we find that $V_r(x_1,Z)$ satisfies \eqref{res-vl}.
 
 So we have constructed a sequence $V_j(x_1,Z)$ satisfying \eqref{transp-01} such that
 \begin{gather*}
 V_j\sim \sum_{l=0}^\infty \sum_{k=0}^l  x^l(\log x)^k \mcv_{j,k,l}(Z),\; \mcv_{j,k,l}\in C^\infty.
 \end{gather*}
 
  Now we apply Borel's lemma in $h$ and $x^j(\log x)^k,$ $k\leq j,$  and we find  
 \begin{gather}
 V(h,x_1,Z) \sim  \sum_{j=0}^\infty \sum_{l=0}^\infty \sum_{k=0}^l h^j x_1^l (\log x_1)^k \mcv_{j,k,l}(Z),  \; \mcv_{j,k,l} \in C^\infty, \label{borel-summ}
 \end{gather}
which  satisfies \eqref{new-asym}. This ends the proof of the Lemma.
\end{proof}

 So far we have constructed an operator $\tG_2(h,\sigma)= G_0(h,\sigma)+ G_{1}(h,\sigma)+ G_2(h,\sigma),$ such that
\begin{gather}
P(h,\sigma,D) \tG_2(h,\sigma)- \id= E_2(h,\sigma), \;\ \beta_0^*K_{E_2(h,\sigma)}= e^{i\soh(\mcr- \tgamma)} F_2(h,\sigma), \;\ F_2\in h^\infty \mck_{ph}^{\novt,\novt}(\xo). \label{stetp-4}
\end{gather}

The fourth step is to remove the error at the front face $\mcf.$   Before we proceed we need the following
\begin{lemma}\label{phg-remainder} Let $\mcr\in \mck_{ph}^{0,0}(\xo)$   
and let $f(h,m)\in h^\infty \mck_{ph}^{0,0}(\xo).$  Then
\begin{gather}
f(h,m)e^{i\soh \mcr} \in h^\infty \mck_{ph}^{0,0}(\xo). \label{phg-remainder1}.
\end{gather}
\end{lemma}
\begin{proof}
Since $\mcr$ is polyhomogeneous with respect to $R$ and $L,$  then according to \eqref{as-ph-exp} in local coordinates $x=(x_1,x_2,x'),$ where $L=\{x_1=0\},$ $R=\{x_2=0\},$
\begin{gather*}
\mcr(x_1,x_2,x')-\mcr(0,0,x') \sim \sum_{j_1,j_2=1}^\infty\sum_{k_1=0}^{j_1}\sum_{k_2=0}^{j_2} x_1^{j_1} x_2^{j_2} (\log x_1)^{k_1}(\log x_2)^{k_2} \mcr_{j_1,j_1,k_1,k_2}(x'),
\end{gather*}
but then by Taylor's formula, for any $\del>0,$
\begin{gather*}
h^{N+1} \left|e^{i\soh\mcr(x_1,x_2,x')} -e^{i\soh \mcr(0,0,x')}-\sum_{j=0}^N \frac{1}{j!} \left(i\soh\left( \mcr(x_1,x_2,x')-\mcr(0,0,x')\right)\right)^j\right| \leq  \\
C(N) h^{N+1} |\soh\left( \mcr(x_1,x_2,x')-\mcr(0,0,x')\right) |^{N+1} \leq C(N,\del) |x|^{N-\del}
\end{gather*}
and therefore  $f(h,m) e^{i\soh \mcr} \in  h^\infty   \mck_{ph}^{0,0}(\xo\times [0,h_0)),$ which proves the Lemma.
\end{proof}

 According to Lemma \ref{aface}, the error $E_2(h,\sigma)$ given in \eqref{mcaasym} is such that
\begin{gather*}
\beta_0^* K_{E_2(h,\sigma)}= h^\infty  e^{i\soh (\mcr-\tgamma)} F_2(h,\sigma), \; F_2\in \mck_{ph}^{\novt,\novt}(\xo\times [0,h_0))
\end{gather*}
But Lemma \ref{phg-remainder} implies that in fact
\begin{gather*}
\beta_0^* K_{E_2(h,\sigma)}= h^\infty  e^{-i\soh \tgamma} \tilde F_2(h,\sigma), \; \tilde F_2\in \mck_{ph}^{\novt,\novt}(\xo\times [0,h_0))
\end{gather*}

\begin{lemma}\label{ff-error}   Let  $E_2(h,\sigma),$  $h\in (0,h_0),$   holomorphic in $\sigma\in \Omega_\hb,$ be such that 
$$\beta_0^*K_{E_2(h,\sigma)}= h^\infty e^{-i\soh \tgamma} F_2(h,\sigma), \; 
F_2 \in  \mck_{ph}^{\novt,\novt}(\xo\times [0,h_0)),$$
and here we may ignore half-densities.  Then there exists an operator $G_3(h,\sigma),$   such that 
\begin{gather*}
\beta_0^* K_{G_3(h, \sigma)}=e^{-i\soh \tgamma}U_3(h,\sigma), \;\ U_3(h,\sigma) \in h^\infty \mck_{ph}^{\fnt, \fnt}(\xo\times [0,h_0)), 
\end{gather*}
and such that 
\begin{gather}
\begin{gathered}
(h^2(\Delta_{g(z)}-\knsq)-\sigma^2)G_3(h,\sigma) - E_2(h,\sigma) = E_3(h,\sigma), \\
\beta_0^* K_{E_3(h,\sigma)}= e^{-i\soh\tgamma} F_3(h,\sigma), \\
F_3(h, \sigma) \in  \rho_\ff^\infty h^\infty \mck_{ph}^{\fnt, \fnt}(\xo\times [0,h_0)).
\end{gathered}  \label{mcaasym-1}
\end{gather}
\end{lemma}
\begin{proof}    The $h^\infty$ factor was important in Lemma \ref{phg-remainder} to reduce the error $E_2$ given by \eqref{mcaasym} to the  form above, but we will not really need it here. Since multiplication by $h$ commutes with $P(h,\sigma,D)$ we will just ignore it in the proof. We use an argument which is basically the same used in \cite{Borthwick} and \cite{MM}.  As $F_2$ is smooth at $\ff$, we can write its Taylor series  at $\{\rho_\ff=0\}:$
\begin{gather*}
F_2 \sim \sum_{j=1}^\infty \rho_\ff^j F_{2, j}, \text{ where } F_{2, j} \in  \mck_{ph}^{\novt,\novt}(\ff\times [0,1)),
\end{gather*}
where $\mck_{ph}^{\novt,\novt}(\ff\times [0,1))$ is the space of functions defined on $\ff \times [0,1)$ which have polyhomogeneous expansions at the right and
 left face.  Our goal is to find 
\begin{gather*}
U_3 \sim \sum_{j = 1}^\infty \rho_\ff^j U_{3, j}, \;\  U_{3, j}\in \mck_{ph}^{\novt,\novt}(\ff\times [0,h_0)) \text{ such that } \\
\beta_0^*(P(h, \sigma,D)) e^{-i\soh \tgamma} U_3 - e^{-i\soh \tgamma} \beta_{0}^* K_{F_2} = e^{-i\soh \tgamma} F_3, \;\ 
 F_3 \in \rho_\mcs^\infty \rho_\ff^\infty \mck^{\novt,\novt}(\xo\times [0,h_0)).
\end{gather*}
 We recall that the normal operator  introduced in \cite{MM} and also used in \cite{Borthwick} is defined as 
\begin{gather*}
\mcn(D)= \beta_0^*P(h,\sigma,D)|_{\ff}, 
\end{gather*}
notice that for example in coordinates \eqref{eqc2}, $\ka_L= \ka(y'+uZ),$ and $\rho_\ff=u,$ so $\ka_L|_{\ff}= \ka(y'),$ but the variables $y'$ serve as parameters for the operator 
$\beta_0^*P(h,\sigma,D),$ So, as observed in \cite{Borthwick,MM}, it follows that, 
\begin{gather*}
\mcn(h,\sigma,D) =h^2(  \Delta_{g_0}-\knsq)- \sigma^2, \text{ where } \\ g_0 \text{  is the metric on the hyperbolic space } g_0= \frac{dx^2}{\ka^2(y') x^2} + \frac{dy^2}{x^2}.
\end{gather*}
So the first step is to solve
\begin{gather*}
\mcn(h,\sigma,D) \wtu_{3,0}= F_0, 
\end{gather*}
and here one needs to establish  the mapping propertie of $\mcn(h,\sigma,D)^{-1}$ given by
\begin{gather}
\begin{gathered}
\mcn(h,\sigma,D)^{-1}: e^{-i\soh \tgamma} \mck_{ph}^{\novt,\novt}(\ff \times [0,1)) \longmapsto  e^{-i\soh \tgamma} \mck_{ph}^{\novt,\novt}(\ff \times [0,1)), \\
\text{ where } \mu=\frac{1}{\ka(y')}\sqrt{1-\frac{h^2 n^2}{4\sigma^2}(\ka(y')^2-\ka_0^2)}, \text{ is constant on the fiber of } \ff \text{ over } (x',y'),
\end{gathered}\label{mapping-prop}
\end{gather}
holomorphically in $\sigma\in \Omega_\hb.$ This was done in Proposition 4.2 of \cite{Borthwick}.    Now extend $U_{3,0}$  to a function in $W_{3,0} \in \mck^{\novt,\novt}(\xo \times [0,h_0)),$ and so
\begin{gather*}
\beta_0^*(P(h, \sigma,D)) e^{-i\soh \tgamma} W_{3,0} - e^{-i\soh \tgamma} \beta_{0}^* K_{F_2} =\rho_{\ff}  e^{-i\soh \tgamma} \mce_2, \;\ \mce_2\in \mck^{\novt,\novt}(\xo\times [0,h_2)).
\end{gather*}
Next we want to find $U_{3,1}$ such that
\begin{gather*}
\mcn(h,\sigma,D) U_{3,1}= F_{2,1}+ \mce_2|_{\rho_\ff=0}.
\end{gather*}
Again, we use \eqref{mapping-prop} to guarantee that this can be solved, and the solution is in he right space.   By induction we construct  $U_{3,j},$ $j=0,1, \ldots,$ and by taking the Borel summation we find $G_3(h,\sigma)$ as desired.
\end{proof} 
The fifth and final step consists of removing the error at the left face.   Recall that so far, using Lemmas \ref{step1-const}, \ref{step2-const}, \ref{aface} and \ref{ff-error},  we have constructed an operator 
$\tG(h,\sigma)=G_0(h,\sigma)+ G_1(h,\sigma)+ G_2(h,\sigma)+ G_3(h,\sigma)$ such that
\begin{gather*}
P(h,\sigma,D)\tG(h,\sigma)-\Id=E_3(h,\sigma), \;\ \beta_0^*K_{E_3(h,\sigma)})= \rho_L^{\novt-i\soh\mu_L}\rho_R^{\novt -i\soh \mu_R} F, \\
F \in h^\infty\rho_{\ff}^\infty \mck_{ph}^{0,0}(\xo\times [0,h_0)),
\end{gather*}
Also recall that $G_0(h,\sigma)$ and $G_1(h,\sigma)$ are supported away from the left and right faces, and that 
\begin{gather*}
\beta_0^*(G_2(h,\sigma)+ G_3(h,\sigma)) =e^{-i\soh \tgamma} \rho_R^\novt\rho_L^\novt \mch= \rho_R^{\novt-i\soh \mu_R}  \rho_L^{\novt-i\soh \mu_L}\mch, \;\ \mch\in \mck_{ph}^{0,0}(\xo\times [0,h_0)),
\end{gather*} 
and here we are not concerned with the structure of $\mch.$   But the whole point about introducing
$\mu_R$ and $\mu_L$  is that  $\novt-i\soh \mu_L$ is an indicial root of $\beta_0^*P(h,\sigma,D),$ which one can verify using local projective 
coordinates valid near the left face. This implies that 
\begin{gather*}
\beta_0^*P(h,\sigma,D)   \rho_L^{\novt-i\soh \mu_L}  \rho_R^{\novt-i\soh \mu_R}\mch= \rho_L^{\novt-i\soh \mu_L+1}  \rho_R^{\novt-i\soh \mu_R}\widetilde{\mch},
\end{gather*} 
and so the error $E_3(h,\sigma)$ in fact satisfies
\begin{gather*}
\beta_0^*K_{E_3(h,\sigma)}= \rho_L^{\novt-i\soh\mu_L+1}\rho_R^{\novt -i\soh \mu_R} F, \;\
F \in h^\infty\rho_{\ff}^\infty \mck_{ph}^{0,0}(\xo\times [0,h_0)),
\end{gather*}
and since $x=\rho_L\rho_\ff,$ $x'=\rho_\ff \rho_R,$ this implies that the kernel of $E_3(h,\sigma)$ satisfies
\begin{gather*}
K_{E_3(h,\sigma)}= x^{\novt-i\soh \mu(y)+1} {x'}^{\novt-i\soh \mu(y')} Z(x,y,x',y'), \;\  Z \in h^\infty \mck^{0,0}(X\times X \times [0,h_0)),
\end{gather*}
$\mck_{ph}^{0,0}(X\times X \times [0,1))$  denotes the space of functions smooth  in $\intx \times \intx \times [0,\infty)$ with polyhomogeneous expansion at $\{x=0\}\cup\{x'=0\}.$  So we need to prove

\begin{lemma}\label{left-face}   Given $E_3(h,\sigma)= x^{\novt-i\soh \mu(y)+1} {x'}^{\novt-i\soh \mu(y')} Z(h,x,y,x',y'), \;\  Z \in h^\infty \mck^{0,0}(X\times X \times [0,h_0)),$  there exists
\begin{gather*}
G_4(x,y,x',y')\sim x^{\novt-i\soh \mu(y)+1} {x'}^{\novt-i\soh \mu(y')}( W_0(y,x',y')+ \sum_{j=0}^\infty \sum_{k=0}^{j} x^j (\log x)^k W_j(y,x',y'), \;
 \end{gather*}
with $W_j(y,x',y')$ polyhomogeneous in $x'$ and vanishing to infinite order at $h=0,$ such that
\begin{gather}
P(h,\sigma,D) G_4(h,\sigma)-E_3(h,\sigma)=E_4(h,\sigma) \in  h^\infty x^{\infty} {x'}^{\novt-i\mu(y')} \mck^{0,0}(X\times X \times [0,h_0)). \label{power-series}
\end{gather}
\end{lemma}
\begin{proof}  Since $P(h,\sigma,D)$ does not depend on $(x',y')$  and $h$ commutes with $P(h,\sigma,D),$ we treat these as parameters and do not take them into account in the computations.  So we have
\begin{gather*}
E_3(h,\sigma) \sim x^{\alpha(y)} E_{3,0}(y) + x^{\alpha(y)} \sum_{j=1}^\infty \sum_{k=0}^{j} x^j (\log x)^k E_{3,j}(y), \;\ \alpha=\novt-i\soh \mu(y)+1,
\end{gather*}
and we want 
\begin{gather*}
G_4(h,\sigma) \sim x^{\alpha(y)} G_{4,0}(y) + x^{\alpha(y)} \sum_{j=1}^\infty \sum_{k=0}^{j} x^j (\log x)^k G_{4,j}(y).
\end{gather*}
such that \eqref{power-series} holds. We substitute these expressions in the left side of \eqref{power-series} and match the coefficients.  Recall that $P(h,\sigma,D)$ is given by
\begin{gather*}
P(h,\sigma,D)=\\ -h^2\ka^2(y)((x\p_x)^2-nx\p_x- x^2F(x,y) \p_x) - h^2 x^2 H^{jk}(x,y) \p_{y_j}\p_{y_k} + h^2 x^2 B_j(x,y) \p_{y_j} -\knsq-\sigma^2
\end{gather*}
So the first term must satisfy
\begin{gather*}
\left(-h^2\ka^2(y)( \alpha^2-n\alpha) -\knsq-\sigma^2\right) G_{4,0}(y)= E_{3,0}(y),
\end{gather*}
and the key is that  $\alpha$ is not an indicial root, so the coefficient on the left side is not equal to zero, and this equation can be solved.

Next  we need to show that if $\beta=\novt-i\soh\mu(y)+j,$ $j\not=0,$  and $k\in \mn_0,$ then for any $E(y)$ there exists $G(y)$ such that 
\begin{gather*}
P(h,\sigma,D) x^{\beta} (\log x)^{k}G(y)- x^{\beta}(\log x)^kE(y) = \text{ less singular terms}
\end{gather*}
and we end up with the same equation as above with $\beta$ in place of $\alpha,$ which can be solved because $\beta$ is not an indicial root. 
This proves the Lemma.
\end{proof}

 Then  $G(h, \sigma) = G_0(h,\sigma)+G_{0,1}(h,\sigma)+ G_2(h,\sigma)+ G_3(h,\sigma)+ G_4(h,\sigma)$ and  satisfies  
 \begin{gather*}
 P(h,\sigma,D) G(h,\sigma,D)- \id=E_4(h, \sigma) \text{ as in \eqref{power-series}}
 \end{gather*}
 is the desired parametrix.
 
 Steps 1, 2, 4 and 5 of the construction work for any metric, even non-trapping ones.   We have proved Step 3 for geodesically convex CCM,  and we will now extend it to arbitrary non-trapping CCM.  To do that we have to introduce the space of semiclassical Lagrangian distributions associated with $\La^*.$ This  class of distributions has been studied in the case of $C^\infty$ Lagrangians, see for example \cite{Alexandrova,GuiStern}, but in this case $\La^*$ has polyhomogeneous singularities at the right and left face. Fortunately, most results valid in the $C^\infty$ case easily extend to this one.

\section{Semiclassical Lagrangian Distributions}\label{PHLD}

In the proof of Lemma  \ref{aface} we encountered distributions of the type $e^{i\soh r} U(h,\sigma,m)|dg(z)|^\ha,$ where  $\beta_0^*U(h,\sigma,m)$ has an asymptotic expansion of the form
\begin{gather*}
\beta_0^*U(h,\sigma,m)\sim  h^{-\novt-1} e^{i h \sigma \beta} \sum_{j=0}^\infty  h^j U_j(\sigma',m), \; U_j(\sigma',m)\in \mck_{ph}^{0,0}(\xo), \\
\text{ and } h^2 \beta= \gamma-\tgamma, \;\ \sigma=1+h\sigma',
\end{gather*}
  and as in \cite{Alexandrova}, distributions of the form $e^{i\soh r} U,$ where $U$ is a semiclassical symbol,  are examples of semiclassical Lagrangian distributions in $\intx \times \intx.$ The main difficulty here is to understand the global behavior  of this distribution.  This case of Lemma \ref{aface} is rather special because if $\La$ is the manifold defined in \eqref{defla0}, and $(\intx,g)$ is geodesically convex, the projector $ \Pi:  \La \subset T^*(\intx \times \intx) \longmapsto \intx \times \intx$ is a diffeomorphism, but in general this is not true.   
  
  Lagrangian distributions  were introduced by H\"ormander \cite{Hormander-FIO}   following a long history of work by several people, see references in \cite{Hormander-FIO}. The  almost parallel semiclassical version of this concept has been studied by several people  including Alexandrova \cite{Alexandrova}, Chen and Hassell \cite{ChenHa} and  Guillemin and Sternberg \cite{GuiStern}.   Definition 5.2.1 of Duistermaat's book \cite{Duistermaat} is a definition of semiclassical Lagrangian distributions,  even though this is not said there.   One major difference is that the Lagrangian manifolds associated with semiclassical Lagrangian distributions are not necessarily conic.

The key property in this theory is that Lagrangian submanifolds can be locally parametrized by phase functions, and Lagrangian distributions are locally given by oscillatory integrals with phase given by the function that parametrizes the Lagrangian.  Let $M$ be a $C^\infty$ manifold of dimension $d,$ let $T^*M$ denote its cotangent bundle and let $\omega$ denote the canonical 2-form on $T^*M.$    Let  $\Omega$  be a local chart for $M$ (which we identify with an open subset of $\mr^d$) and let $U \subset \Omega \times \mr^N,$ $N\in \mn_0,$ be an open subset.    A function $\Phi(y,\theta) \in C^\infty(U; \mr)$
is a non-degenerate phase function if  
\begin{enumerate}[{ 1.}]  
\item  $|d_{y,\theta}\Phi(y,\theta)|\geq C (1+|\theta|)^\rho,$ for some $\rho>0$ and all  $(y,\theta)\in U,$ 
\item If $d_\theta \Phi(y,\theta)=0$ for some  $(y,\theta) \in U,$ then 
\begin{gather*}
d_{y,\theta}( \frac{\p \Phi(y,\theta)}{\p \theta_j}), \text{ are linearly independent for } j=1,2,\ldots N. 
\end{gather*}
\end{enumerate}
If these conditions are satisfied,
\begin{gather*}
C_\Phi=\{(y,\theta) \in U: \; d_\theta\Phi(y,\theta)=0\}
\end{gather*}
is a $C^\infty$ submanifold of $U$ of dimension $d$ and the map
\begin{gather*}
T_\Phi: C_\Phi \longrightarrow T^*\Omega \\
(y,\theta) \longmapsto (y, d_y \Phi(y,\theta))
\end{gather*}
is an immersion, and
\begin{gather*}
\La_\Phi=\{(y,d_y\Phi(y,\theta)): \;  (y,\theta)\in C_\Phi\}
\end{gather*}
is an immersed Lagrangian submanifold of $T^*\Omega.$ Moreover, any $d$-dimensional  $C^\infty$  submanifold $\La\subset T^*M$ is Lagrangian if and only if for every 
$(y_0,\eta_0)\in \Lambda$  there is a local chart $\Omega$  near $y_0$ such that 
$\Lambda\cap T^*\Omega =\Lambda_\Phi$ for some non-degenerate phase function $\Phi.$  Notice that if $N=0,$ then  $\La=\{(y,d_y\Phi(y))\},$ and therefore, if  $\Pi: T^* \Omega \longrightarrow \Omega$ is the canonical projector, then $\Pi|_\La: \La\longrightarrow \Omega$ is a diffeomorphism. The converse is also true, if  $\La\subset T^*M$ is Lagrangian 
 $\Pi|_\La: \La \longrightarrow M$ is a diffeomorphism  near $(y_0,\eta_0),$  then there exists a neighborhood 
 $\Omega \ni y_0 $ and  a function $\Phi \in C^\infty (\Omega)$ such that  $\La\cap T^*\Omega=\{(y,d_y\Phi)\}.$   In general this is not possible, and one needs the 
 $\theta$-variables   to parametrize $\La.$  In fact one can always choose a special phase function. We know, see from  for example Section 4.1 of \cite{Alexandrova}, that  if $\La\subset T^* M$ is a Lagrangian submanifold,  and $(y_0,\eta_0)\in \La,$  there exist local coordinates $y=(y',y''),$ $y'=(y_1,\ldots,y_k),$ in  a neighborhood of $y_0$   and  corresponding dual coordinates $(\eta',\eta'')$ and  $C^\infty$ maps
\begin{gather*}
S':  \mr_{y''}^{n-k}\times \mr^k_{\eta'} \longrightarrow \mr^k, \\
(y'',\eta') \longmapsto (S_1'(y'',\eta'), \dots, S_k'(y'',\eta'))\\
\text{ and }
S'': \mr_{y''}^{n-k}\times \mr^k_{\eta'} \longrightarrow \mr^{n-k}, \\
(y'',\eta') \longmapsto (S_{k}''(y'',\eta'), \dots, S_n''(y'',\eta')).
\end{gather*}
such that
\begin{gather}
\La=\{(y,\eta): \ y'=S'(y'',\eta'), \;\ \eta''= S''(y'',\eta')\}.  \label{par}
\end{gather}
In fact, since the canonical form $\omega=d\eta \wedge dy=0$ on $\La,$ it follows that there exists $\tilde{S}(y',\eta'')$ such that
\begin{gather*}
S_j'(y'',\eta')= \p_{\eta_j'} \tilde{S}(y'',\eta'), \;\ 1\leq j \leq k, \\
S_j''(y'',\eta')= -\p_{y_j''} \tilde{S}(y'',\eta'), \;\ k+1\leq j \leq n.
\end{gather*}
This implies that
\begin{gather}
\La=\{ (y, d_{y} \Phi(y,\eta)):  d_\eta \Phi(y,\eta)=0\} \text{ where }  \Phi(y,\eta)= \lan y',\eta'\ran-  \tS(y'',\eta'). \label{special-par}
\end{gather}

  In the case we have studied, $\xo$ is a $C^\infty$ manifold with corners, but the product structure  \eqref{prod}  valid in a tubular neighborhood of $\p X$ gives a way of doubling $X$ across its boundary and extending the metric $x^2g,$ where $x$ is the boundary defining function in \eqref{prod}.  Similarly, the lift of the product structure  \eqref{prod} from $X\times X$  to $\xo$ gives a way of doubling $\xo$ across its boundary  and of extending the lift of the metric from either factor of $X\times X$ as well. So we may think of $\xo$ as a submanifold with corners of a  $C^\infty$ manifold.    In the case of asymptotic constant curvature, the manifold $\La^*$ is a $C^\infty$ Lagrangian manifold  that can be smoothly extended across  the boundary $\p T^*(\xo)$ and therefore, for any $p\in\La^*$ including points on the boundary,  there exist neighborhood   $\Gamma$ of $p$ such that
$\La^*\cap \Gamma=\La_\Phi$ for a non-degenerate phase function $\Phi(m,\theta).$   In other words,
\begin{gather*}
\La^*\cap \Gamma =\{(m,d_m\Phi):  d_\theta \Phi( m,\theta)=0\}.
\end{gather*}

When $\ka$ is not constant, we need to show that the manifold $\La^*,$ which has polyhomogeneous singularities at the right and left faces of $T^*(\xo),$ can be parametrized by a phase function which has polyhomogeneous singularities at the right and left faces of $T^*(\xo),$ but it is $C^\infty$ in the $\theta$ variables.   Here we need all the properties of $\La^*$ established in Theorem \ref{soj}:
\begin{prop}\label{polyphase} Let $(\intx,g)$ be a non-trapping CCM and let $\La^*$ be the manifold defined in Theorem \ref{soj}.  If $\upsilon \in T_{L}^*(\xo)\cap \La^*$ or  $\upsilon \in T_{R}^*(\xo)\cap \La^*$  and 
$\Pi: T^*(\xo)\longrightarrow \xo$ is the canonical projector,  then there exists  an (relatively) open chart $\Omega\ni \Pi(\upsilon)$ and a phase function $\Phi(m,\theta)$ with $(m,\theta) \in U \subset  \Omega \times \mr^N,$ open and $N\in \mn_0,$  such that $\Phi(m,\theta)$ is $C^\infty$ in the interior of $U,$ is $C^\infty$ up to the front face $\ff,$  and  has polyhomogenous expansions at $L$ and $R$ in the sense that if $x=(x_1,x_2,y)$ where
$L=\{x_1=0\},$ $R=\{x_2=0\}$
\begin{gather}
\Phi(x_1,x_2,y,\theta)= \Phi(0,0,y,\theta)+ \sum_{j_1=2, j_2=2}^\infty \sum_{k_1=0}^{j_1}\sum_{k_2=0}^{j_2} x_1^{j_1} x_2^{j_2} (\log x_1)^{k_1} (\log x_2)^{k_2} \Phi_{j_1,j_2,k_1,k_2}(y,\theta), \label{reg-phase-fc}
\end{gather}
$\Phi(0,0,y,\theta), \, \Phi_{j_1,j_2,k_1,k_2}(y,\theta) \in C^\infty,$ and
\begin{gather*}
\La^*\cap T^*\mathring{\Omega}=\{(m, d_m\Phi): d_\theta \Phi=0\}, 
\end{gather*}
and up to the boundary of $\Omega.$
\end{prop}
\begin{proof}  Let  us assume that $\upsilon$ lies in a codimension 2 corner and we will carry out the proof uniformly up to the front face.  The proofs in all the other cases follow the same argument.  Let $(x_1,x_2,y)$ $y \in \mr^{2n}$ be local coordinates near $\Pi(\upsilon)$ such that 
\begin{gather*}
 R=\{x_1=0\} \text{ and } L=\{x_2=0\}.
\end{gather*}
  We shall assume that $x_1(\upsilon)=x_2(\upsilon)=0,$ so the projection of $\upsilon$ to the base, that we shall denote by $\Pi(\upsilon),$ lies in the intersection of the right and left faces.   We know from   Lemma \ref{asym-expL},   that
\begin{gather*}
\La^* \cap\{ x_1=x_2=0\}=\La_\p^*
\end{gather*} 
is a  $C^\infty$ Lagrangian in $T^*\{x_1=x_2=0\}$ which can be thought as a submanifold of $T^* \mr^{2n+2}$ given by 
\begin{gather}
T^*\{x_1=x_2=0\} =\{ (x_1,x_2,y, \xi_1,\xi_2,\eta): x_1=x_2=\xi_1=\xi_2=0\}.  \label{tx1x2}
\end{gather}
  In view of Lemma \ref{vfields}, in these coordinates, $H_{\wp_\bullet},$ $\bullet=R,L$ is given by \eqref{fields}.
 Then pick coordinates $(y,\eta)$ valid  in a neighborhood $\upsilon \in U_\p\subset T^*\mr^{2n}$  such that \eqref{special-par} holds for $\La_\p^*.$  We can  extend $(y,\eta)$ to  coordinates $(\ty,\teta)$  valid in an open set $U\subset T^*(\xo)$  which are constant along the integral curves of $H_{\wp_L}$ and $H_{\wp_R}$  respectively, starting from $U\cap\{x_1=x_2=0\}=U_\p.$ It follows from  \eqref{param-A0} that
\begin{gather*}
\ty= y + F(x_1, x_2,y,\eta), \;\  F(0,0,y,\eta)=0, \\
\teta=\eta+ G(x_1,x_2,y,\eta), \;\ G(0,0,y,\eta)=0,\\
\xi_1= W_1(x_1,x_2, y,\eta), \;\  W_1(0,0,y,\eta)= 0 \\
\xi_2= W_2(x_1,x_2, y,\eta), \;\  W_2(0,0,y,\eta)= 0, 
\end{gather*}
where $F,$ $G,$ $W_1$ and $W_2$ have polyhomogeneous expansions at $R$ and $L.$   However, in view of \eqref{special-par}, we in fact have that on $\La^*,$
\begin{gather*}
\begin{gathered}
\ty= y + \tF(x_1, x_2,y'',\eta'), \;\ \tF(0,0,y'',\eta')=0, \\
\teta=\eta+ \tG(x_1,x_2,y'',\eta'), \;\ \tG(0,0,y'',\eta')=0, \\
\xi_1= \tW_1(x_1,x_2, y'',\eta'), \;\  \tW_1(0,0,y'',\eta')=0 \\
\xi_2= \tW_2(x_1,x_2, y'',\eta'), \;\  \tW_2(0,0,y'',\eta')=0,
\end{gathered}\label{eqs-par-lambda}
\end{gather*}
and this this implies that on $\La^*,$
\begin{gather}
\begin{gathered}
\ty'= \tH_1(x_1, x_2,\ty'',\teta'), \;\  \teta''= \tH_2(x_1,x_2,\ty'',\teta'), \\
\xi_1=X_1(x_1,x_2, \ty'', \teta'), \;\ \xi_2=X_2(x_1,x_2, \ty'', \teta'),
\end{gathered}\label{para-Lag}
\end{gather}
 It also follows from the definition of the coordinates 
$(\ty,\teta)$ that $\La^*$ is a Lagrangian submanifold with respect to 
$\omega= d\xi_1 \wedge d x_1+ dx_2 \wedge d\xi_2+ d \teta \wedge d \ty.$  It then follows from \eqref{para-Lag} that there exist $\tS(x_1,x_2,\ty'',\teta')$ such that
\begin{gather*}
\ty'=\p_{\teta'} \tS, \;\ \eta''=-\p_{\ty}\tS, \\
\xi_1= -\p_{x_1}\tS, \;\ \xi_2=-\p_{x_2} \tS.
\end{gather*}
Therefore, setting $\theta=\teta',$ in the interior,  $\La^*= \La_\tPhi,$ where $\tPhi(x_1,x_2,y,\teta')= \lan \ty',\teta'\ran- \tS(x_1,x_2,\ty'',\teta').$   We are left to prove \eqref{reg-phase-fc}, but this follows from \eqref{param-A0}, and exactly the same  argument used in the proof of \eqref{exp-mcr}, as in Theorem \ref{main}.  
\end{proof}

The following is a consequence of Proposition \ref{polyphase} and \eqref{shift}
\begin{corollary}\label{phase-shift} Let $(\intx,g)$ be a non-trapping CCM, let $\La\subset T^*(\intx\times \intx)$ be the Lagrangian manifold defined in \eqref{defla0} and let $\La^*$ be the manifold defined in Theorem \ref{soj}.  If $\Phi(m,\theta)$ is a polyhomogeneous phase function that locally parametrizes  $\La^*$ in the interior of $T^*(\xo),$ then 
$\Psi(m,\theta)= \Phi(m,\theta)-\gamma$ locally parametrizes $\beta_0^*\La$ in the interior of $T^*(\xo).$
\end{corollary}

We need to introduce the following concept:
\begin{definition}\label{admissible}   Let $(\intx,g)$ be a non-trapping CCM, and let $\La^*\subset T^*(\xo)$ be the Lagrangian submanifold defined in Theorem \ref{soj}.  We say that $\{\mco_j, \; j\in \mn\}$ is an admissible cover of $\La^*$ if  $\mco_j$ is a local chart of $\xo,$ $\{\mco_j, \; j\in \mn\}$ cover $\La^*$  and  there exist phase functions 
$\Phi_j\in C^\infty(\intmco_j \times \mr^{N_j})$  (and here we identify $\mco_j$ with a subset of $\mr^{4n+4}$)  such that
\begin{gather*}
\La^* \cap  T^* \mco_j= \La_{\Phi_j}, \; \Phi_j \in C^\infty(\mco_j \times \mr^{N_j}) \text{ if } \La^* \text{ is } C^\infty,
\end{gather*}
and if $\La^*$ is polyhomogeneous at $R$ and $L,$
\begin{gather*}
\La^* \cap  T^* \mco_j= \La_{\Phi_j}, \; \Phi_j \in C^\infty(\overline{\mco}_j \times \mr^{N_j}) \text{ if } \overline{\mco}_j \cap (R \cup L)= \emptyset, \\
\La^* \cap T^* \intmco_j= \La_{\Phi_j}, \; \Phi_j \in C^\infty(\intmco_j \times \mr^{N_j}) \ \text{ has an expansion \eqref{reg-phase-fc} near } R\cup L,  \\ \text{ if } \overline{\mco}_j \cap (R \cup L)\not= \emptyset.
\end{gather*}
\end{definition}
Notice that expansion \eqref{reg-phase-fc} guarantees that the definition of non-degeneracy can be applied to such phases.

Before we define Lagrangian distributions,  first we need to introduce the class of symbols we will work with
\begin{definition}\label{symbols}    Let $\mco\subset  [0,\infty)_{x_1}\times [0,\infty)_{x_2}\times \mr^n_{x'}$ be a relatively open subset.   We define the space $S( \mco \times (0,h_0)\times \mr^N),$ $N\in \mn_0,$  to be the space of  functions 
$a: (0,h_0) \times \mco \times \mr^N \rightarrow \mc,$ $N\in \mn_0,$  such that 
\begin{enumerate}[{i.}]
\item  If $\overline{\mco} \cap (\{x_1=0\}\cup \{x_2=0\})=\emptyset,$ there exist $a_j\in C^\infty(\mco\times \mr^N)$ with 
\begin{gather}
\sup_{(x,\theta) } | \p_x^\alpha \p_\theta^\beta a_j(x,\theta)| = C_{j,\alpha,\beta} <\infty,  \label{symb-1A}
\end{gather}
and such that for any $J\in \mn,$
\begin{gather}
\sup_{(x,\theta)} |\p_x^\alpha\p_\theta^\beta \left(a(x,h,\theta)-\sum_{j=0}^J h^j a_{j}(x,\theta)\right)| \leq C_{J,\alpha,\beta} h^{J} \label{symb-2A}
\end{gather}
\item If $\overline \mco \cap (\{x_1=0\}\cup\{x_2=0\})\not=\emptyset,$ there exist $C^\infty$ functions $a_\bullet(x',\theta),$ $\bullet \in \mn^5,$ such that 
\begin{gather}
\sup_{(x',\theta)\in \mco\times \mr^N} | \p_{x'}^\alpha \p_\theta^\beta a_\bullet(x',\theta)| = C_{\bullet,\alpha,\beta} <\infty, \label{symb-12}
\end{gather}
and for any $J,L\in \mn,$ and $\del>0$ there exists $C(J,\del)>0$ such that for 
\begin{gather}
\begin{gathered}
\mce_{J,L}(h,x,\theta)= a(x,h,\theta)-\sum_{j_1, j_2=0}^J\sum_{k_1=0}^{j_1}\sum_{k_2=0}^{j_2}
\sum_{l=0}^L x_1^{j_1}(\log x_1)^{k_1} x_2^{j_2}(\log x_2)^{k_2} h^l a_\bullet(x',\theta),\\ 
\bullet=(j_1,j_2,k_1,k_2,l), \; \sup_{(x,\theta)\in \mco\times \mr^N} |\p_x^\alpha\p_\theta^\beta \mce_{J,L}(h,x,\theta)| \leq C_{J,\del} h^{J}|(x_1, x_2)|^{J-\del} \label{symb-2}
\end{gathered}
\end{gather}
\end{enumerate}
\end{definition}
It is a consequence of Borel's lemma, see Theorem 2.1.6 of \cite{Hormander}, that given a sequence $a_\bullet(x',\theta)$ satifsying \eqref{symb-12} one can find a function 
$a(x,h,\theta)$ such that \eqref{symb-2} holds.

Now we define the semiclassical  Lagrangian distributions with respect to $\La^*:$
\begin{definition}\label{lox-lag-d}  Let  $\La^* \subset T^*(\xo)$ be as above, and let  $\ \Omega^\ha$ denote the half-density bundle over $\xo.$  
We say that $A$ is a polyhomogenous Lagrangian  distribution of order $k$  with respect to $\La^*,$ and denote  $A\in I_{ph}^k(\xo, \La^*, \Omega^\ha)$ if there exists an admissible cover
$(\mco_j,\Phi_j)$   of $\La^*,$  and symbols $a_j(z,h,\theta) \in S(\mco_j\times (0,h_0) \times \mr^{N_j})$ compactly supported in $\theta$ such that for each $K\Subset \xo$ there exists $M$ such that  for  $u \in C_c^\infty(K; \Omega^\ha),$ $\lan A, u\ran=\sum_{j=1}^M \lan A_j,u\ran,$  where
 \begin{gather}
 \begin{gathered}
\lan A_j,u\ran= (2\pi h)^{-k-\frac{(d+ 2N_j)}{4}} \int_{\mr^{N}} \int_{K} e^{\ioh \Phi_j(z,\theta)}  a_j(h,z,\theta) u(z) \; d\theta dz, \text{ provided } N_j\geq 1,\\
 \lan A_j,u\ran= (2\pi h)^{-k-\frac{d}{4}} \int_{K} e^{\ioh \Phi_j(z)} a_j(z,h) u(z) \; dz, \;\ \text{ if } N=0,
\end{gathered}\label{def-LagD}
\end{gather}
and $d=2n+2$ is the dimension of $\xo.$
\end{definition}

Let us consider one of these oscillatory integrals in \eqref{def-LagD},
\begin{gather*}
A(a)(z) =(2\pi h)^{-k-\frac{(d+ 2N)}{4}} \int_{\mr^{N}} e^{\ioh \Phi(z,\theta)}  a(h,z,\theta) \; d\theta.
\end{gather*}
According to the definition, the function $\Phi$ may have polyhomogeneous singularities at $R\cup L,$ but it is $C^\infty$ in $\theta.$  Therefore, just as in the case of standard Lagrangian distributions, it follows from the stationary phase theorem that  if 
 \begin{gather*}
C=\{(z,\theta): d_\theta \Phi(z,\theta)=0\}, 
\end{gather*}
and $a_C(h,z,\theta)=a(h,z,\theta)|_{C},$ is the restriction of the symbol a to $C,$    then
 \begin{gather*}
 A(a)(z)- A(a_C)(z) = O(h).
 \end{gather*}
 In other words,  modulo terms of one higher order of $h,$ one may take the restriction of $a(h,z,\theta)$ to $C$  in the definition of $A,$  instead of $a(z,\theta).$

One can define the principal symbol of $A$ basically in the same way as in the case of  Lagrangian distributions introduced by H\"ormander \cite{Hormander,Hormander-FIO}.
First  pick local coordinates $\{\la_j, \; 1\leq j \leq d\}$ on $C,$ extend them to a neighborhood of $C$ and let
 \begin{gather*}
 d_C=| d\la_1 d\la_2 \ldots d\la_{d}| \left| \frac{D(\la,\Phi_\theta')}{D(z,\theta)}\right|^{-1},
 \end{gather*}
 which is independent of the choice of $\{\la_j\}.$   One then defines the half-density valued symbol $a_C(h,z,\theta)\sqrt{d_C}$ on the manifold $C.$    To see that this is well defined and polyhomogeneous up to the boundary  one can  use the expansion of $\Phi(z,\theta)$ given by \eqref{reg-phase-fc}.   In particular, if one takes $\Phi$ as in Proposition \ref{polyphase}, and using the notation of the proof of that Proposition, sets $\la=(x_1,x_2,\ty'',\theta),$ then $(\la, \Phi_\theta')=( x_1, x_2, \ty'', \theta,y'-\p_\theta \tS(x_1,x_2,\ty''),$ and the Jacobian is equal to one.

  Recall that the map
\begin{gather*}
C \longrightarrow \La^*\\
(z,\theta) \longmapsto (z, d_{z}\Phi(z,\theta))
\end{gather*}
is a diffeomorphism in the interior.  It turns out that  the push-forward of $a_C(h,z,\theta) \sqrt{d_C}$ from $C$ to $\La^*$ via the map defined above, which is still denoted by $a\sqrt{d_C},$ is invariant under a change of phase function that locally parametrizes $\La^*.$  We define this class of symbols by $S(\La^*,\Omega_\La^\ha).$  Moreover, it turns out that if $a(h,z,\theta)\in S(\mco\times (0,h_0)\times  \mr^{N_1}),$  $b(h,z,\theta') \in S(\mco\times (0,h_0) \times  \mr^{N_2}),$ and two phase functions
 $\Phi(z,\theta)$ and $\Psi(z,\theta')$  parametrize  $\La^*,$  and 
 \begin{gather*}
  (2\pi h)^{-k-\frac{(d+ 2N_1)}{4}}  \int_{\mr^{N_1}} \int_{\mco } e^{\ioh \Phi (z,\theta)}  a (h,z,\theta) d\theta= \\ (2\pi h)^{-k-\frac{(d+ 2N_2)}{4}} \int_{\mr^{N_2}} \int_{\mco} e^{\ioh \Psi (z,\theta')}  b(h,z,\theta') d\theta', 
  \end{gather*}
  then
  \begin{gather}
  e^{i\frac{\pi}{4} \operatorname{sgn} \Phi_{\theta\theta}''} a \sqrt{d_C}- e^{i\frac{\pi}{4} \operatorname{sgn} \Psi_{\theta'\theta'}''} b \sqrt{d_C}\in h S(\La^*,\Omega_{\La^*}^\ha).  \label{PS-0}
  \end{gather}

 One can use \eqref{PS-0} to define the principal symbol of the Lagrangian distribution $A$ as an element of $S(\La^*,  M_{\La^*}\otimes \Omega_\La^\ha),$ where $M_{\La^*}$ is the Maslov line bundle defined in \cite{Hormander-FIO,Hormander} 
  \begin{gather*}
  I^k(\xo, \La^*, \Omega^\ha)/   I^{k-1}(\xo, \La^*, \Omega^\ha) \longrightarrow  S(\La^*,  M_{\La^*} \otimes \Omega_{\La^*}^\ha)/ hS(\La^*,  M_{\La^*}\otimes \Omega_{\La^*}^\ha)
  \end{gather*}
  which locally is of the form
  \begin{gather}
  \sigma^k(A)=  e^{i\frac{\pi}{4} \operatorname{sgn} \Phi_{\theta\theta}''} a_0 \sqrt{d_C}, \label{princ-sym}
  \end{gather}
  where $\Phi\in C^\infty(\mco\times \mr^N)$ parametrizes $\La^*$ and  $a|_C=a_0.$  Notice that this only involves derivatives in $\theta,$ and is perfectly well-defined for our class of phase functions.

   Now we go back to $\intx \times \intx,$ where $(\intx,g)$ is a non-trapping CCM.  Let $\La\subset T^*(\intx \times \intx)$ be the manifold defined in \eqref{defla0}.   Let $A \in I^k(\intx \times \intx, \La,\Omega^\ha),$ a standard semiclassical Lagrangian distribution, and let $P_L(h,\sigma,D)=h^2 (\Delta_g(z)-\knsq)-\sigma^2,$  defined on the left factor of 
   $\intx\times \intx.$  Its semiclassical principal symbol is equal to $p(z,\zeta,z',\zeta')=|\zeta|_{g^*(z)}^2-1,$ which by definition vanishes on $\La.$  The principal symbol of $P_L A$   is given by the following  analogue of  Theorem 25.2.4 of \cite{Hormander}.  In fact its proof is very similar to that of  the reference, but it can be found in  \cite{ChenHa}.
\begin{prop}\label{van-princ} Let $\La$ be the Lagrangian manifold defined in \eqref{defla0} and let $A \in I^k(\intx \times \intx, \La, \Omega^\ha),$ have principal symbol 
  $a\in S(\La,  M_\La\otimes \Omega_\La^\ha).$  Then 
  $P A \in I^{k-1}(\intx\times \intx, \La, \Omega^\ha)$ and 
  \begin{gather}
\sigma^{k-1}(PA)= \frac{1}{i} \mcl_{H_{p_L}}a + p_L^{s} a, \label{Lie-Der}
\end{gather}
where $H_{p_L}$ is the Hamilton vector field of $p$ lifted to $\intx \times \intx$ from the first fcator, and  $p_L^{s}$ is the semiclassical subprincipal symbol of $p$ also lifted from the first factor.
\end{prop}
Recall that if $p(h,z,\zeta)= p_0(z,\zeta)+ h p_1(z,\zeta) + O(h^2),$ then 
\begin{gather*}
p^{s}(z,\zeta)= p_1(z,\zeta) -\frac{1}{2i} \sum_{j=1}^n \frac{\p^2 p_0(z,\zeta)}{\p z_j \p \zeta_j}.
\end{gather*}

We should remark that the definitions above involve real-valued phase functions, but we can apply them to without a problem to oscillatory integrals with phase $\sigma \Phi,$ as long as $\sigma=1+h \sigma',$ $\sigma'\in (-c,c)\times i (-C,C).$  In this case   $e^{-i\soh \Phi}=e^{-\frac{i}{h} \Phi} e^{-i \sigma' \Phi},$ and the factor $ e^{-i \sigma' \Phi}$ can be thought to be part of the symbol of the Lagrangian distribution. 

\section{A semiclassical parametrix for non-trapping CCM}\label{General-CCM}

In this section we will use the class of semiclassical Lagrangian distributions we have just defined to  extend Lemma \ref{aface} to arbitrary non-trapping CCM.

\begin{lemma}\label{gen-aface}  Let $(\intx,g)$ be a non-trapping  CCM, and let $F_1(h,\sigma)$ be as in \eqref{redef-f1}.  Then there exist  $h_0>0$ and  operators $G_2(h,\sigma)$ and  $E_2(h,\sigma) $ with  Schwartz kernels  $K_{G_2}$ and $K_{E_2}$ such that 
$\beta_0^* K_{G_2}\in e^{-i\soh \tgamma} I_{ph}^{\ha}(\xo, \La^*,\Omega^\ha)$ and $\beta_0^*K_{E_2} \in e^{-i\soh \tgamma} I_{ph}^{-\infty}(\xo, \La^*,\Omega^\ha),$   holomorphic in $\sigma \in \Omega_\hbar,$ with  $h\in (0,h_0),$ where 
$\tgamma$ is defined in \eqref{def-gamma} and are such that 
\beqq\label{main-eq}
(h^2(\Delta_{g(z)}-\knsq)-\sigma^2)G_2(h,\sigma) - e^{i\soh r} F_1(h,\sigma) = E_2(h,\sigma).
\eeqq
\end{lemma}
\begin{proof}   As in the proof of Lemma \ref{aface},  it is convenient to work with $Q(h,\sigma,D)$  given by \eqref{def-opq},  instead of $P(h,\sigma,D),$ and we denote $F_1=x^{\novt}\tF_1.$  and 
$G_2= x^{\novt} \tG_2.$  Also, as in the proof of Lemma \ref{aface}, we begin  by observing that $\beta_0^*K_{\tF_1}$ has an expansion
\begin{gather*}
\beta_0^* K_{\tF_1}\sim h^{-\novt} \sum_{j=0}^\infty h^j \tF_{1,j}(\sigma',m).
\end{gather*}

 So the first step is to find $\tG_{2,0}(h,\sigma)$ with $\beta_0^*K_{\tG_{2,0}} \in e^{-i\soh \tgamma} I_{ph}^{\ha}(\xo, \La^*,\Omega^\ha)$ such that  if 
 $\beta_0^* K_{\wtf_{1,0}}= h^{-\novt} \tF_{1,0},$ then
\begin{gather}
Q(h,\sigma,D) \tG_{2,0}(h,\sigma) - e^{i\soh r} \wtf_{1,0}(h,\sigma)= h \mce_1(h,D) , \; \beta_0^* K_{\mce_1} \in e^{-i\soh \tgamma}I_{ph}^{\ha}(\xo,\La^*,\Omega^\ha).\label{g20}
\end{gather}
Since in $\intx \times \intx,$ $h^{-\novt}e^{i\soh r}\wtf_{1,0}(\sigma',m)$ is a semiclassical Lagrangian distribution of order $-\ha$ with respect to the manifold $\La$ defined in \eqref{defla0}, one would expect that,  again in $\intx \times \intx,$ the kernel of $G_{2,0},$ $K_{\tG_{2,0}} \in  I^{\ha}(\intx \times \intx, \La, \Omega^\ha).$  If $g_{2,0}$ is the semiclassical principal symbol of $K_{G_{2,0}},$  $q$ and $q^{s}$ are the semiclassical principal symbol and subprincipal symbol of $Q(h,\sigma,D),$ then  according to \eqref{Lie-Der}
\begin{gather}
\frac{1}{i} \mcl_{H_{q}}g_{2,0} + q^{s} g_{2,0}= \wtf_{1,0}. \label{transp-gen0}
\end{gather}
This equation can be solved in $\intx \times \intx$ without a problem, but the whole point is to describe the  asymptotic behavior of $g_{2,0}$  at the right and left faces of $\xo.$   Since $\beta_0$ is a diffeomorphism in the interior of $\xo,$ this equation lifts to an equation on $\beta_0^*\La$ in  the interior of $T^*\xo$  given by
\begin{gather}
\begin{gathered}
\frac{1}{i} \mcl_{H_{q_L}}g_{2,0} + q_L^{s} g_{2,0}= \wtf_{1,0}, \text{ on } \beta_0^* \La, \text{ and } 
g_{2,0}=0 \text{ on } \La_0 \times [0,h_0),
\end{gathered}\label{transp-gen2}
\end{gather}
where, by abuse of notation, $g_{2,0}$ also denotes the principal symbol of $\beta_0^* K_{\tG_{2,0}},$  $q_L$ and $q_L^{s}$ denote the semiclassical principal symbol  and subprincipal symbol of  $\beta_0^*Q(h,\sigma,D).$     But as we know, $\beta_0^*\La$ is not smooth up to the right and left faces, and so one should try to work with $\La^*.$  But if $\vphi\in C^\infty(\xo), $ the map $T^*(\xo) \ni(m,\nu) \longmapsto (m,\nu+d\vphi)\in T^*(\xo)$ 
preserves the canonical 2-form, and we also know from \eqref{def-gamma} that $Q_{L,\gamma}=e^{\ioh \gamma} Q_L(h,\sigma,D) e^{-\ioh \gamma}$ and $Q_{L,\tgamma}=e^{\ioh \tgamma} Q_L(h,\sigma,D) e^{-\ioh \tgamma}$ satisfy $Q_{L,\gamma}-Q_{L,\tgamma}=O(h^2),$ so they have the same principal and subprincipal symbols.   So if $q_{L,\gamma}$ and $q_{L,\gamma}^{s}$ are the semiclassical principal and subprincipal symbols of $Q_{L,\gamma}(h,\sigma,D),$  equation \eqref{transp-gen2} becomes
\begin{gather}
\begin{gathered}
 \frac{1}{i}\mcl_{H_{q_L}} g^*_{2,0} + q_{L,\gamma}^{s} g^*_{2,0}= \wtf_{1,0}, \text{ on } \La^* \; g_{2,0}= e^{-i\soh \tgamma} g_{2,0}^*, \\
g_{2,0}=0 \text{ at } \diag_0 \times [0,h_0).
\end{gathered}\label{transp-gen1}
\end{gather}
  Again, due to the non-trapping assumptions, this equation can be solved up to $L.$ However,  to understand its asymptotics at $L$ one needs to work in local coordinates. This was essentially done in \eqref{first-T}  near $L.$   Let $(\mco_j,\Phi_j)$ be an admissible cover near the left face.  In the interior of each $\mco_i,$   we have
\begin{gather*}
\beta_0^* K_{\tG_{2,0}}(m,\sigma,h)=  (2\pi h)^{-\novt-1-\frac{N}{2}} e^{-i\soh \tgamma} \int_{\mr^N} e^{-i\soh \Phi_j(m,\theta)} g^*_{2,0,j}(m, \theta) \; d\theta,
\end{gather*}
and hence
\begin{gather*}
e^{i\soh \tgamma} Q(h,\sigma,D)  \beta_0^*K_{\tG_{2,0}}= \\ (2\pi h)^{-\novt-1-\frac{N}{2}} e^{-i\soh \tgamma}  \int_{\mr^N} e^{-i\soh \Phi_j(m,\theta)} Q(h,\sigma, D+\sigma d_m(\tgamma+ \Phi(m,\theta))) g^*_{2,0,j}(m,\theta) \; d\theta.
\end{gather*}
As mentioned above,  we are working with $\sigma=1+h\sigma',$ $\sigma'\in (-c,c)\times i (-C,C),$ and in this case $e^{-i\soh \Phi_j(m,\theta)} =e^{-\frac{i}{h} \Phi_j(m,\theta)} e^{-i\sigma' \Phi_j(m,\theta)},$ and the latter part can be viewed as part of the amplitude $g_{2,0,j}$ in the definition of the oscillatory integral.  By doing this we are just simplifying the construction of the symbols.

Since  $\wtf_{1,0}$ is compactly supported in the interior of $T^*(\xo),$ then as in the proof of Lemma \ref{aface}, we want
\begin{gather*}
Q(h,\sigma, D+\sigma d_m(\tgamma+ \Phi(m,\theta))) g^*_{2,0,j}(m,\theta)=h B_j(m,\theta).
\end{gather*}
Here of course $\Phi(m,\theta)$ plays the exact same role of $\mcr$ in the proof of Lemma \ref{aface}.  The key feature of $\Phi_j(m,\theta)$ is that it satisfies \eqref{reg-phase-fc}, just like $\mcr$ satisfied \eqref{exp-mcr}.   As in \eqref{transp-0},  we are reduced to solving the transport equation
\begin{gather}
(W+\vtheta) g_{2,0,j}=0, \;\ g_{2,0,j} \in C^\infty(\{x_1>\del\}), \label{new-transp-eq}
\end{gather}
where $W$ and $\vtheta$ are given by \eqref{transp-0} with $\mcr$ replaced by $\Phi_j(m,\theta).$ Since $\Phi_j(m,\theta)$ satisfies \eqref{reg-phase-fc},  then  \eqref{new-transp-eq} can be solved in the same way as  \eqref{first-T}.  In view of the discussion above, this defines  a semiclassical Lagrangian distribution  in the sense that  
$e^{i\soh \tgamma} \beta_0^*K_{\tG_{2,0}} \in I_{ph}^{\ha}(\xo, \La^*,\Omega_\ha)$ and such that  the operator $\tG_{2,0}$ whose kernel is $K_{\tG_{2,0}}$ satisfies \eqref{g20}.  

The next step is to find  $\tG_{2,1}$ such that
$e^{i\soh \tgamma} \beta_0^*K_{\tG_{2,1}} \in I^{\ha}(\xo, \La^*,\Omega^\ha)$ and such that
\begin{gather*}
 Q(h,\sigma,D) \tG_{2,1}(h,\sigma)-e^{i\soh} \wtf_{1,1}- \mce_1(h,D)= h \mce_2(h,D), \;\ \beta_0^* K_{\mce_2} \in I^{\ha}(\xo,\La^*,\Omega^\ha).
\end{gather*}

Again,  if $g_{2,1}$ denotes the principal symbol of $e^{i\soh \tgamma} \beta_0^*K_{G_{2,1}},$ and $e_1$ denotes the principal symbol of  
$e^{i\soh \tgamma} \beta_0^*K_{\mce_1},$ then in the interior it must satisfy 
\begin{gather}
\begin{gathered}
\frac{1}{i} \mcl_{H_{q_L}}g_{2,1} + q_L^{s} g_{2,1}= \wtf_{1,1}+ e_1, \text{ on } \beta_0^* \La \text{ and }
g_{2,1}=0 \text{ on } \La_0 \times [0,h_0),
\end{gathered}\label{transp-gen3}
\end{gather}
This equation  can be solved without a problem in the interior, and again the only issue is to determine the asymptotic behavior of $g_{2,1}$ at the left face. Again, we work in local coordinates valid in an neighborhood $\mco \subset \xo$ of a point $m_0\in L$ and arrive at the analogue of \eqref{iter-transp}, again with $\Phi_j(m,\theta)$ replacing $\mcr$ given by
\begin{gather*}
(H_{\wp_L} + \frac{1}{\rho_L}\vtheta) g_{2,1}= \frac{1}{\rho_L} e_1+ \frac{1}{\rho_L} (\sum_{jk} a_{jk} D_j D_k+ \mcw_0) g_{2,0}
\end{gather*}
as in \eqref{iter-transp}, where $\mcw_a$ is the analogue of the operator defined as in \eqref{transp-01}, with $\mcr$ replaced by $\Phi(m,\theta).$ Again, the only property of $\mcr$ that was important was its polyhomogeneity with respect to the left face, which if course is shared by 
$\Phi(m,\theta).$    The same argument used in the proof of Lemma \ref{aface} can be used to show that $g_{2,1}$ has a polyhomogeneous expansion at $L.$

The higher order terms are handled in the same way.  So we have constructed a sequence of polyhomogeneous symbols, and we now take the Borel sum both in $h$ and in the polyhomogeneous terms. This ends the proof of Lemma \ref{gen-aface}
\end{proof}

We then proceed exactly as in the case of geodesically convex CCM  to remove the error at the front face and the left face. This proves Theorem \ref{para-structure} and conlcudes the construction of the parametrix.

\section{Resolvent estimates}\label{Resest1}

The proof of Theorem \ref{resest} follows from the asymptotics of the parametrix $G(h,\sigma)$ and the remainder $E(h,\sigma)$ established in Theorem \ref{para-structure}.  The main point is the following result from (the proof of )Theorem 3.25 of  \cite{Mazzeo-Edge}, see also Lemma 6.2 of \cite{MSV}:
\begin{lemma}\label{schurs}
Suppose that the Schwartz kernel of an operator $B: C_0^\infty(X) \longrightarrow C^{-\infty}(X),$ trivialized by $|dg(z')|,$ satisfies 
\begin{gather}
|\beta_0^* K_B| \leq C \rho_L^{\alpha}\rho_R^{\beta}, \label{bound-AB}
\end{gather}
then we have four situations:
\begin{gather*}
\text{ If }  \alpha,\beta>n/2, \text{ then }  \|B\|_{\mcl(L^2)}\leq C' C. \\
\text{  If } \alpha=n/2, \;\ \beta>n/2, \text{ then } \||\log x|^{-N} B\|_{\mcl(L^2)}\leq C' C, \text{ for } N>\ha. \\
\text{ If } \alpha>n/2, \;\ \beta=n/2, \text{  then }  \|| B|\log x|^{-N}\|_{\mcl(L^2)}\leq C' C, \text{ for } N>\ha. \\
\text{ If } \alpha=\beta=n/2, \text{ then } \||\log x|^{-N} B |\log x|^{-N}\|_{\mcl(L^2)}\leq C' C, \;\  N>\ha.
\end{gather*}
\end{lemma}

The operator $G(h,\sigma)$ from Theorem \ref{para-structure} can be written as $G(h,\sigma)= \tG_1(h,\sigma)+ \tG_2(h,\sigma),$ where $\tG_1(h,\sigma)=G_0(h,\sigma)+ G_1(h,\sigma)$ and $\tG_2(h,\sigma)= G_2(h,\sigma)+G_3(h,\sigma)+G_4(h,\sigma).$ In view of the construction, 
$\beta_\hb^*K_{\tG_1}$ is bounded and supported near the semiclassical front face, and hence  $\beta_0^*K_{\tG_1}$ is bounded and supported near $\diag_0 \times [0,h_0),$ and so it satisfies  \eqref{bound-AB} for any $\alpha$ and $\beta.$  Since $\La^*$ is compact, one only needs finitely many oscillatory integrals,  and the second part of the kernel which  is given by the semiclassical parametrix is of the form 
\begin{gather*}
\beta_0^*K_{ \tG_2(h,\sigma)}= \sum_{j=1}^J h_j^{-1-\novt-\frac{N_j}{2} }\rho_R^\novt \rho_L^\novt e^{-i\soh \tgamma} W_j, \;\  W_j\in L^\infty, 
\end{gather*}
One can pick the largest $N_j.$ In the case of geodesically convex CCM, $N_j=0.$
Therefore, 
\begin{gather*}
\beta_0^*K_{ \rho^a \tG_2(h,\sigma)\rho^b } =h^{-N} \rho_\ff ^{a+b}\rho_R^{a+\novt} \rho_L^{b+\novt} e^{-i \soh \tgamma} W,
\end{gather*}
but in view of \eqref{real-anal} and \eqref{def-tgamma}
\begin{gather*}
|e^{-i\soh \tgamma}| \leq C \rho_R^{\frac{\im\sigma}{h\ka_R}} \rho_L^{\frac{\im\sigma}{h\ka_L}}.
\end{gather*}
Therefore, if for any boundary defining function $\rho,$
\begin{gather*}
|\beta_0^*K_{ \rho^a  \tG_2(h,\sigma)\rho^b }| \leq C h^{-N}  \rho_\ff ^{a+b} \rho_R^{a+\novt+ \frac{\im \sigma}{h \ka_R}} \rho_L^{b+\novt+\frac{\im\sigma}{h \ka_L}}
\end{gather*}

In particular, if  $a, b > \frac{\im\sigma}{h \ka_0}$ and $a+ b\geq 0$,  Lemma \ref{schurs} guarantees that there exists $C>0$ such that 
\begin{gather*}
\|\rho^a G(h, \sigma) \rho^b f\|_{L^2(X)}\leq Ch^{-N} \|f\|_{L^2(X)},
\end{gather*}

On the other hand,  Theorem \ref{para-structure} also gives that the kernel of remainder $E(h,\sigma)$ satisfies
\begin{gather*}
\beta_0^*K_{E(h,\sigma)}= h^\infty  \rho_{ff}^\infty \rho_L^\infty e^{-i\soh \tgamma} \rho_R^{\novt} \mce, \;\ \mce \in L^\infty.
\end{gather*}
and again, Lemma \ref{schurs} gives that and for any $L> 0$ there exists $C_L>0$ such that
\begin{gather*}
\|\rho^{-b} E(h, \sigma) \rho^b f\|_{L^2(X)}\leq C_L h^{L} \|f\|_{L^2(X)}.
\end{gather*} 
Thus for $h$ sufficiently small, $\Id + \rho^{-b}E(h, \sigma)\rho^b$ is invertible and  hence
\beq
\rho^a P(h,\sigma)^{-1} \rho^b = h^2 \rho^aG(h, \sigma) \rho^b(\Id + \rho^{-b}E(h, \sigma)\rho^b)^{-1}.
\eeq
So we conclude that
\begin{theorem} If $(X,g)$ is a non-trapping convex CCM,  there exists $N>0$ such that  and if $a, b > \frac{\im\sigma}{h \ka_0}$ and $a+ b\geq 0$, there exists $C>0$ such that 
\begin{gather}
\|\rho^a P(h, \sigma)^{-1} \rho^b f\|_{L^2(X)}\leq Ch^{-N} \|f\|_{L^2(X)}. \label{sc-res-est}  
\end{gather}
In fact we also have estimates for when $a=\frac{\im \sigma}{h}$ and/or  $a=\frac{\im \sigma}{h}.$
\end{theorem}
The holomorphic continuation and $L^2$ estimates of $\rho^a R(\la)\rho^b $ follows from \eqref{sc-res-est}   and the fact that 
$(\re\la)^2R(\la)= P(h,\sigma),$ where $h=(\re\la)^{-1}$ and $\sigma^2=1+i\frac{\im\la }{\re\la}.$ This proves Theorem \ref{resest} .

We will use Theorem \ref{resest} and the results of Datchev and Vasy to prove resolvent estimates for metrics with hyperbolic trapping and prove Theorem \ref{resest-gen}.

\begin{proof}  The main idea of the proof consists splitting the manifold $X$ and the operator $\lap_{g}$ in a way that one can apply the results of Datchev and Vasy \cite{DaVa1}.     This is done as in \cite{MSV1}.  As in the Introduction,  we assume that  $x \in C^\infty(X)$ is a boundary defining function and as in \cite{DaVa1}, let
\begin{gather*}
X= X_0\cup X_1, \;\ X_0=\{ x< 2\eps \}, \;\ X_1=\{ x> \frac{\eps}{2}  \}.
\end{gather*}
Let $(\intx, g_0)$ be a non-trapping CCM and let $g$ be a $C^\infty$ metric on $\intx$ such that $g=g_0$ in $X_0$ and suppose that in 
$X_1$ the trapped set of $g,$  that is,  the set of maximally extended geodesics of $g$ which 
are precompact, is normally hyperbolic. We  define  $\tilde X_1$ to be another Riemannian manifold  extending $X_1'=\{ x>\eps\},$ which is Euclidean outside some compact set.  Let $P_1$ be a self-adjoint second order differential operator such that the operator 
$P_1|_{X_1}=\lap_{g}|_{X_1}$ and suppose the principal symbol of $P_1$ is equal to the Laplacian of the metric on  $\tilde X_1.$  Let
\begin{gather*}
P_2=h^2 P_1-i \Upsilon, \;\ h\in (0,1),
\end{gather*}
where $\Upsilon\in C^\infty(\tilde X_1; [0,1])$ is such that $\Upsilon=0$ on $X_1$ and $\Upsilon=1$ on $\tilde X_1\setminus X_1'.$
Thus, $P_1- 1$ is semiclassically elliptic on a neighborhood of  $X_1\setminus X_1'$.  In particular, this
implies that  $X_1$ is bicharacteristically convex in $X,$ i.e. no bicharacteristic of  $P_1-1$ leaves $X_1$ and returns later.  By Theorem 1 of \cite{WuZw} there exist positive constants $C,$ $c,$ $N$ and $\del$ independent of $h$ such that
\begin{gather*}
|| (P_2-\sigma)^{-1} f||_{L^2(\tilde X_1)} \leq C h^{-N} ||f||_{L^2(\tilde X_1)}, \;\ \sigma \in (1-c,1+c) \times (-\del h, \del h).
\end{gather*}

On the other hand,   since $\lap_{g_0}|_{X_0}=\lap_{g}|_{X_0},$  and the semiclassical  resolvent for $\Delta_{g_0}$ satisfies \eqref{sc-res-est}, Theorem 2.1 of \cite{DaVa1}  implies that if $a, b > \frac{\im\sigma}{h \ka_0}$ and $a+ b\geq 0$, there exists $C>0$ and $N>0$ such that 
\begin{gather}
\|\rho^a (h^2(\Delta_{g_0}-\knsq)-\sigma^2)^{-1} \rho^b f\|_{L^2(X)}\leq Ch^{-N} \|f\|_{L^2(X)}. \label{sc-res-est-2}  
\end{gather}
The high energy resolvent estimate \eqref{HERE} follows easily from this one.

\end{proof} 

\section{Acknowledgements}
 
 The first author visited the Mathematics Department of the Universidade Federal de Santa Catarina at Florian\'opolis  (UFSC)  during  the northern summer 2015  when part of this work was done.  His visit to UFSC was made possible by a grant of Professor Visitante Especial from CAPES, Brazil.  The work was also supported by a grant from the Simons Foundation (\#349507, Ant\^onio S\'a Barreto).


\end{document}